\documentclass[leqno,a4pape]{article}
\usepackage{amsmath,amssymb,amsthm,bm,mathrsfs}
\usepackage{bbm}
\usepackage{color}
\usepackage{layout}
\newcommand{\R}{\mathbb{R}}
\newcommand{\N}{\mathbb{N}}

\newcommand{\Div}{\mathrm{div}}
\newcommand{\rot}{\mathrm{rot}}
\newcommand{\pt}{\partial}
\newcommand{\ep}{\varepsilon}
\newcommand{\vphi}{\varphi}

\newcommand{\Rot}{\mathrm{Rot}}

\newtheorem{proposition}{Proposition}[section]
\newtheorem{theorem}{Theorem}[section]
\newtheorem{lemma}{Lemma}[section]
\newtheorem{corollary}[theorem]{Corollary}

\theoremstyle{definition}
\newtheorem{remark}{Remark}[section]
\newtheorem{definition}{Definition}[section]
\newtheorem*{assumption}{Assumption}
\makeatletter

\@addtoreset{equation}{section}
\makeatother

\usepackage{layout}
\usepackage{xcolor}

\title{Helmholtz-Weyl decomposition on a time dependent domain for
 time periodic Navier-Stokes flows with large flux}

\author{Haru Kanno\thanks{Environment and Information Sciences, Yokohama National University, Yokohama, Japan} \and
    Takahiro Okabe\thanks{Graduate School of Engineering Science, Osaka University, Toyonaka, Japan}\\
    \and
    Erika Ushikoshi\footnotemark[1]\,\,\footnotemark[2]
}
\date{}

\begin{document}
\maketitle
\begin{abstract}
        We consider the Helmholtz-Weyl decomposition on a time dependent bounded domain $\Omega(t)$ in $\R^3$.
        Especially, 
         we investigate the domain dependence of each component in the decomposition, 
         namely, the harmonic vector fields (i.e., $\Div$ and $\rot$ free vectors),
        vector potentials, and scalar potentials 
        equipped with suitable boundary conditions, when $\Omega(t)$ moves along to $t\in\R$.
        As an application, we construct a time periodic solution of the incompressible Navier-Stokes equations
        for some large boundary data with non-zero fluxes.
\end{abstract}

\section{Introduction}
Let $\Omega \subset \R^3$ be a bounded domain with the smooth boundary. 
We consider the Helmholtz-Weyl decomposition for a solenoidal vector field $b$ on $\Omega$, i.e., 
$\Div\, b=0$ in $\Omega$, such as
\begin{equation}\label{eq;HW dec}
	b = h + \rot\, w \quad \text{in } \Omega,
\end{equation}
where $h$ is a harmonic vector field, $\Div\, h=0$, $\rot\, h=0$ in $\Omega$ 
with the boundary condition $h \times \nu =0$ on $\pt \Omega$,  
and $w$ satisfies $\Div\, w=0$ in $\Omega$ with $w\cdot \nu=0$ on $\pt \Omega$. 
Here, $\nu$ denotes the unit outward normal vector to $\pt \Omega$. 
Such a decomposition for smooth vectors was proved by Weyl \cite{Weyl},
and Bendali, Dominguez and Gallic \cite{BDG} 
expanded it in the Sobolev space $H^s(\Omega)$, $s\geq 0$.
Recently, Kozono and Yanagisawa \cite{Kozono Yanagisawa IUMJ} 
derived \eqref{eq;HW dec} for $L^r$-vector fields, $1<r<\infty$. 
Furthermore, a lot of relevant studies have been developed by Foia\c{s} and Temam \cite{Foias Temam}, 
Yoshida and Giga \cite{Yoshida Giga}, Griesinger \cite{Griesinger} 
and Bolik and von Wahl \cite{Bolik Wahl}, and so on.

In this paper, introducing  a time dependent domain $\Omega(t)$ for $t\in \R$, 
we consider the Helmholtz-Weyl decomposition \eqref{eq;HW dec} on each $\Omega(t)$, i.e., 
\begin{equation}\label{eq;HW dec(t)}
    b(t) = h(t) + \rot\, w(t) \quad \text{in } \Omega(t).
\end{equation}
For \eqref{eq;HW dec(t)}, our main interest is to clarify the domain dependence of 
the harmonic vector field $h(t)$ and the vector potential $w(t)$ 
when $\Omega(t)$ moves 
along to time $t \in \R$.
Especially, time continuity and time differentiability for $h(t)$ and $w(t)$ 
in suitable function spaces are concerned.

The Helmholtz-Weyl decomposition has been developed and applied 
in various fields of the mathematical and numerical analysis 
for the electromagnetism, magnetohydrodynamics and also fluid mechanics.  
In particular, it is suggested that \eqref{eq;HW dec} plays an effective role in the solvability 
of the boundary value problem of the incompressible Navier-Stokes equations. 
Indeed, using \eqref{eq;HW dec} Kozono and Yanagisawa \cite{Kozono Yanagisawa MathZ} 
constructed the stationary solution for large boundary data. 
The second author applied it to the time periodic problem of the Navier-Stokes equations.
From this viewpoint, to investigate the domain dependence of \eqref{eq;HW dec(t)} 
makes us expect to apply it to such problems on a time dependent domain $\Omega(t)$. 

For our aim, we focus on the fact that the harmonic vector field $h(t)$ and 
the vector potential $w(t)$ is determined by the following (elliptic) problems
\begin{equation}\label{eq;h and w equations}
    \begin{cases}
        \rot\, h(t) = 0 &\text{in } \Omega(t),\\
        \Div\, h(t) =0  &\text{in } \Omega(t),\\
        h(t) \times \nu(t)=0 &\text{on } \pt\Omega(t),
    \end{cases}
    \quad\text{and}\quad
    \begin{cases}
        -\Delta w(t) = \rot\, b(t) &\text{in } \Omega(t), \\
        \rot\, w(t) \times \nu(t) = b(t) \times \nu(t)
        &\text{on } \pt \Omega(t),
        \\
        w(t) \cdot \nu(t) = 0 & \text{on } \pt \Omega(t),
    \end{cases}
\end{equation}
where $\nu(t)$ is the outward unit normal vector to $\pt\Omega(t)$.
Therefore,
our problem is essentially reduced to  an analysis of 
the domain dependence of the solutions of the (elliptic) equations
with the suitable boundary conditions.

There are many results on the domain perturbation problem for the elliptic equations.
See, Hadamard \cite{hada}, Garabedian and Schiffer \cite{gar-schi}, Peetre \cite{Peetre}, for instance.
Especially, Fujiwara and Ozawa \cite{Fujiwara Ozawa} gave us a strategy to investigate the domain dependence of the solutions
to the strictly elliptic equation with the boundary condition of the Dirichlet type in the
sense of Agmon, Douglis and Nirenberg \cite{ADN}, 
where the Hadamard variational formula of the Green function was established.
    In \cite{Fujiwara Ozawa}, they transformed the problem on $\Omega(t)$ into 
    one on some fixed domain $\widetilde{\Omega}$, and derived
    the uniform a priori estimate of the transformed problem, 
which
plays an essential role in a quantitative analysis for 
the variation of the solutions along to the domain perturbation.
Then, Kozono and the third author \cite{Kozono Ushikoshi} applied their method to the Stokes equations.

Concerning to \eqref{eq;h and w equations}, 
it is known that the space of harmonic vector fields $h(t)$ has finite dimension
and is spanned by the gradient vectors of solutions of the 
following Laplace equations with Dirichlet boundary conditions:
\begin{equation}\label{eq;IntroHarmonic}
    \begin{cases}
        \Delta q_k(t) = 0 & \text{in } \Omega(t), \\
        q_k(t)|_{\Gamma_\ell(t)}=\delta_{k,\ell} &\text{for } \ell=0,\dots,K,
    \end{cases}
\end{equation}
for $k=1,\dots,K$,
provided $\pt \Omega(t)$ consists of mutually disjoint $C^\infty$ closed surfaces 
$\Gamma_0(t),\dots,\Gamma_K(t)$
which satisfy $\Gamma_1(t),\dots,\Gamma_K(t)$ lie inside of $\Gamma_0(t)$
and provided $\Omega(t)\setminus \bigcup_{\ell=1}^L \Sigma_\ell(t)$
is simply connected, where mutually disjoint $C^\infty$ surfaces $\Sigma_1(t),\dots,\Sigma_L(t)$ 
transversal to $\pt\Omega(t)$, see \cite[Appendix A]{Kozono Yanagisawa IUMJ}.
Therefore, in order to consider the time dependence of $h(t)$, 
it is essential to investigate the feature of $q_k(t)$.
For the analysis for $q_{k}(t)$, the method as in \cite{Fujiwara Ozawa}
 and \cite{Kozono Ushikoshi} can be applicable, since \eqref{eq;IntroHarmonic} is the Dirichlet problem.
 Transforming \eqref{eq;IntroHarmonic} to one on a fixed domain $\widetilde{\Omega}$,
 we can investigate the time dependence of the transformed function $\widetilde{q}_k(t)$ on $\widetilde{\Omega}$
 by the associated a priori estimate.

    On the other hand, the analysis of the vector potential $w(t)$  
    is rather nontrivial.
    In fact, though \eqref{eq;h and w equations} for $w(t)$ forms the strictly elliptic problem in the sense of \cite{ADN}, we can not make use of the above approach for the harmonic vector field $h(t)$. 
    The difficulty comes from the fact that the lower order term of the solutions 
    remains in the a priori estimate, which is an essential difference from the above Dirichlet problem.
    This fact is inevitable since the rotation operator under the nonslip boundary condition has
    the nontrivial kernel, especially when the domain is multi-connected.
    Indeed, Foia\c{s} and Temam \cite{Foias Temam} showed the null space is 
finite dimensional and spanned by the gradient vectors of the
solutions of the Laplace equations with the Neumann boundary conditions
(with jump conditions),
giving the explicit basis.
Thus, we need to find out a quite new approach which allows us a quantitative analysis,
e.g., time continuous and differentiability for
the transformed vector potential $\widetilde{w}(t)$ on $\widetilde{\Omega}$.

To overcome this difficulty, we focus on the domain dependence 
of the constants $C(\Omega)>0$
in the estimate of the Helmholtz-Weyl decomposition such as
\begin{equation*}
    \| w \|_{H^2(\Omega)} \leq C(\Omega) \| b \|_{H^1(\Omega)}
    \quad \text{where } h+\rot\,w = b
    \quad \text{for all } b \in H^1(\Omega).
\end{equation*}
    Indeed, we establish $\sup\limits_{0\leq t \leq T} C\bigl(\Omega(t)\bigr)<\infty$ for $T>0$.
    Due to  uniform boundedness of the constant
    we obtain that $\|\widetilde{w}(t)\|_{H^2(\widetilde{\Omega})}$ is bounded on $[0,T]$.
    Then, the compactness of the Sobolev embedding 
    gives us a method to manage  the lower order term 
    in the a priori estimate.

    However, to derive the time differentiability of $\widetilde{w}(t)$,
    the estimate $\sup\limits_{0\leq t \leq T} C\bigl(\Omega(t)\bigr)<\infty$ is not
    enough to control the difference quotient via the a priori estimate.
    Therefore, we have to introduce a new different way, by considering more qualitative features of the vector potentials.
    Indeed, we need to investigate the time variation of the component of 
    the null space of the rotation operator with non-slip boundary condition. 
    However, it seems not to be possible to deal with directly the components of the null space,
    due to the lack of the quantitative information of the basis which was obtained by Foia\c{s} and Temam \cite{Foias Temam}.
    To avoid such a difficulty, we find that 
    the fluxes (of the cross-sections of $\Omega(t)$)
\begin{equation*}
    \int_{\Sigma_\ell(t)} w(t)\cdot \nu_\ell(t) \,dS=0,
    \quad \text{for } \ell=1,\dots,L,
\end{equation*}
are conserved under the coordinate transformation 
by the diffeomorphism.
This means that the orthogonality of the vector potential to the null space of the rotation
is preserved under the coordinate transform.
Due to this orthogonality, restricting our vector potentials within the orthogonal compliment
of the null space,
we can dominate the Sobolev norm of the vector potential by the estimate 
as in Bolik and von Wahl \cite{Bolik Wahl},
and \cite{Kozono Yanagisawa IUMJ} without an effect of topology of the domain.

From the above observation, we succeed to obtain the domain dependence 
of the Helmholtz-Weyl decomposition on a time dependent domain.
Moreover, we can expand such a property for solenoidal vector fields to that for general vector fields.
\medskip

    As an application of the decomposition, we consider the following 
    time periodic problem of 
the incompressible
Navier-Stokes equations on a time dependent domain $\Omega(t)\subset\R^3$
which is bounded with smooth boundary $\pt \Omega(t)$ for $t\in\R$:
\begin{equation}\tag{N-S}
    \left\{
    \begin{split}
        &\pt_t v -\Delta v + v\cdot \nabla v + \nabla \pi=f
        \quad \text{in } \bigcup_{t\in \R} \Omega (t) \times \{t\}, 
        \\
        & \Div\, v = 0 
        \quad \text{in } \bigcup_{t\in \R} \Omega (t) \times \{t\},
        \\
        &v=\beta
        \quad \text{on } \bigcup_{t\in \R}\pt \Omega(t)\times \{t\}.
    \end{split}
    \right.
\end{equation}
Here, $v=v(x,t)=\bigl(v^1(x,t),v^2(x,t),v^3(x,t)\bigr)$ and $\pi(x,t)$ 
denote the unknown velocity and pressure of the fluid at $(x,t) 
\in \bigcup_{t\in\R} \Omega(t)\times \{t\}$, respectively.
While $f=f(x,t)=\bigl(f^1(x,t),f^2(x,t),f^3(x,t)\bigr)$ is the given external force
and 
$\beta =\beta(x,t)=\bigl(\beta^1(x,t),\beta^2(x,t),\beta^3(x,t)\bigr)$ 
is the given boundary data.

In this paper, let $Q_\infty:=\bigcup_{t\in\R}\Omega (t) \times \{t\}$ 
be a noncylindrical space-time domain. 
Then, we assume that there exists a cylindrical domain 
$\widetilde{Q}_\infty=\widetilde{\Omega}\times \R$ and a level-preserving 
$C^\infty$ diffeomorphism $\Phi: \overline{Q}_\infty \to \overline{\widetilde{Q}}_\infty$.
For more detail, see Section \ref{subsec;notations} below.

    To deal with (N-S), introducing a suitable solenoidal extension $b(t)$ 
satisfying $\Div\, b(t)=0$ in $\Omega(t)$ and $b(t)|_{\pt\Omega(t)}=\beta(t)$ on $\pt \Omega(t)$,
we reduce (N-S) to one with the homogeneous boundary condition.
Then, we transform the problem on the noncylindrical domain $Q_\infty$ to 
that on the cylindrical domain $\widetilde{Q}_\infty$.

    The problem on the time dependent domain has been developed by
    Ladyzhenskaya \cite{Ladyzhenskaya},
    Fujita and Sauer \cite{Fujita Sauer} for initial-boundary value problem, 
    and by Morimoto \cite{Morimoto},
    Inoue and Wakimoto \cite{Inoue Wakimoto}, Miyakawa and Teramoto \cite{Miyakawa Teramoto}
    and Salvi \cite{Salvi} for time periodic problem, and so on.

    As for the time periodic problem and also the stationary problem under the nonhomogeneous boundary conditions,
there has been a difficulty for the analysis of the convection term with respect to the extension $b$,
especially, when a domain is  multi-connected. 
To get around this difficulty, the Helmholtz-Weyl decomposition \eqref{eq;HW dec} 
plays an effective role for the case that $\Omega$ is a fixed domain.
Indeed, \cite{Kozono Yanagisawa MathZ} and \cite{Okabe JEE} constructed a stationary solution
and a time periodic solution of (N-S), respectively, for
some large boundary data with non-zero fluxes.
On the other hand, for the case of time dependent domain $\Omega(t)$,
Miyakawa and Teramoto \cite{Miyakawa Teramoto}
constructed a time periodic weak solution
with the period $T>0$, 
provided $\beta$ is small enough.
So, we intend to adopt the Helmholtz-Weyl decomposition \eqref{eq;HW dec(t)}
into the time periodic problem on the time dependent domain $\Omega(t)$.
By the virtue of the time dependence of the decomposition \eqref{eq;HW dec(t)},
we succeed to estimate the convection term and to construct a time periodic solution for some large boundary data.
Furthermore, 
we note that the advantage is that our existence theorem is formulated only by the given data on $\Omega(t)$,
without any information on the fixed domain $\widetilde{\Omega}$ and the diffeomorphism.
\medskip

This paper is organized as follows.
In the Section 2, we state our main results on the Helmholtz decomposition and the existence of the time periodic solutions of (N-S).
Section 3 devoted to the preliminaries for the diffeomorphism and the estimate of the convection term.
In  Section 4, the time dependence of the Helmholtz-Weyl decomposition is concerned.
Especially,
we shall give a proof of Theorem \ref{thm;time depend HW dec div} in Section \ref{subsec;rot} 
and Corollary \ref{cor;time depend HW dec gen} in Section \ref{subsec;cor}.
In Section 5, we construct a time periodic solution of the (transformed) Navier-Stokes equations.
%
%
%
%
\section{Main results}
%
This section consists of three parts.
Firstly, we state an assumption on domains and prepare several notations,
and then we consider the formulation of the problem on (N-S) and a precise definition of weak solutions of (N-S) 
in Section \ref{subsec;notations}.
In Section \ref{subsec;HW}, we state our main results on the domain dependence of the Helmholtz-Weyl decomposition.
In the final subsection, we state the existence of time periodic solution of (N-S) on 
a noncylindrical domain. 
%
\subsection{Notations and setting}\label{subsec;notations}
%
\subsubsection{Assumption on the domain}
To begin with, we impose the following assumption on $\{\Omega(t)\}_{t\in\R}$.
\begin{assumption} 
    $\bigl\{\Omega(t)\bigr\}_{t\in\R}$ is a family of bounded domains in $\R^3$ with smooth boundary $\pt\Omega(t)$
    and satisfies the followings.
    \begin{itemize}
        \item[(i)] For the noncylindrical domain $Q_\infty=\bigcup\limits_{t\in\R} \Omega(t)\times\{t\}$, there exist a cylindrical domain $\widetilde{Q}_\infty=\widetilde{\Omega}\times \R$ and a level-preserving 
        $C^\infty$ diffeomorphism $\Phi: \overline{Q}_\infty \to \overline{\widetilde{Q}}_\infty$;
        \begin{equation*}
        (y,s)=\Phi(x,t):=\bigl( \phi(x,t),t\bigr)
        :=\bigl(\phi^1(x,t),\phi^2(x,t),\phi^3(x,t),t\bigr)
        \end{equation*}
    such that
    \begin{equation} \label{eq;J}
    \det \left( \frac{\pt \phi^i}{\pt x^j}(x,t)\right)
    \equiv J(t)^{-1} >0 
    \quad \text{for } (x,t) \in \overline{Q}_\infty.
    \end{equation}
    \item[(ii)]
    There exists $K \in \N$ such that for all $t\in\R$, $\pt \Omega(t)$ consists of
    $K+1$ connected components $\Gamma_0(t),\dots,\Gamma_K(t)$ of 
     $C^\infty$ closed surfaces which satisfy that
$\Gamma_1(t),\dots,\Gamma_K(t)$ lie inside of $\Gamma_0(t)$,
$\Gamma_k(t) \cap \Gamma_{k^\prime}(t)= \emptyset$ for $k\neq k^\prime$, 
and that
\begin{subequations}\label{eq;KL}
\begin{equation}\label{eq;K}
    \pt \Omega(t)= \bigcup_{k=0}^K \Gamma_{k}(t).
\end{equation} 
Furthermore, there exists $L\in\N$ such that 
\begin{equation}\label{eq;L}
\text{$\dot{\Omega}(t):=\Omega(t)\setminus \Sigma(t)$ is simply connected for $t\in\R$, where
$\Sigma(t):=\displaystyle \bigcup_{\ell=1}^L \Sigma_\ell(t)$,}
\end{equation}
\end{subequations}
and $\Sigma_1(t),\dots,\Sigma_L(t)$ are $C^\infty$ surfaces transversal to
$\pt\Omega(t)$ with $\Sigma_\ell(t)\cap \Sigma_{\ell^\prime}(t)=\emptyset$ 
for $\ell \neq \ell^\prime$, for all $t\in\R$.
Here, we may assume that $\phi\bigl(\Sigma(t),t\bigr)=\phi\bigl(\Sigma(t^\prime),t^\prime\bigr)$
for all $t,t^\prime\in\R$,
without loss of generality.
    \end{itemize}
\end{assumption}
\begin{remark}
        As is mentioned in Miyakawa and Teramoto \cite[Theorem 4.3]{Miyakawa Teramoto},
        the condition \eqref{eq;J} is of no restriction.
\end{remark}
In this paper, we express the inverse of $\Phi$ as
    \begin{equation}\label{eq;diffeo inverse}
        \Phi^{-1}(y,s)=\bigl(\phi^{-1}(y,s),s\bigr) := \bigl(\phi^{-1}_1(y,s),\phi_2^{-1}(y,s),
        \phi_3^{-1}(y,s),s\bigr)
        \quad \text{for } (y,s) \in \widetilde{Q}_\infty.
    \end{equation}
Moreover, for $T>0$ we put 
\begin{equation}\label{eq;QT}
    Q_T:=\bigcup\limits_{0<t<T}\Omega(t)\times\{t\}
    \quad\text{and}\quad \widetilde{Q}_T:=\widetilde{\Omega}\times (0,T).
\end{equation}

Here, it should be noted that since $\phi(\cdot,t) : \Omega(t) \ni x \mapsto y \in \widetilde{\Omega}$ 
is a diffeomorphism, there is a one-to-one relation
between  vector fields $u=(u^1,u^2,u^3)$ on $\Omega(t)$
and $\widetilde{u}=(\widetilde{u}^1,\widetilde{u}^2,\widetilde{u}^3)$ on $\widetilde{\Omega}$ 
such that, for $i=1,2,3$, 
\begin{equation}\label{eq;identify}
    \widetilde{u}^i = \sum_{\ell} \frac{\pt y^i}{\pt x^\ell} u^\ell 
    \quad{\text{equivalently,}}\quad
    u^i = \sum_{k} \frac{\pt x^i}{\pt y^k} \widetilde{u}^k.
\end{equation}
Hence, we can identify the vector fields on $\Omega(t)$ with that on $\widetilde{\Omega}$
in the sense of the above relation.
So, we may write just $\widetilde{u}$
when $u$ on $\Omega(t)$ and $\phi(\cdot,t)$ are obvious from the context.
In the present paper, we often use a notation of a vector (or function) with wide-tilde, like $\widetilde{u}$, 
in order to express an arbitrary vector field (or function) on $\widetilde{\Omega}$ as well.

\subsubsection{Notations}
We shall next introduce function spaces and notations. 
Let $\Omega$ be a subset in $\R^3$, or $\R^4$. 
Then, $C_{0,\sigma}^\infty(\Omega)$ denotes the set of 
all $C^\infty$-solenoidal vectors $\psi$ with compact support in $\Omega$, i.e., $\Div\, \psi=0$ in $\Omega$.
Let $L^r_\sigma(\Omega)$ be the closure of $C_{0,\sigma}^\infty(\Omega)$ with respect to 
the $L^r$-norm, $1<r<\infty$. 
Furthermore, $(\cdot,\cdot)$ is the duality pairing between $L^r(\Omega)$ and $L^{r^\prime}(\Omega)$,
where $\frac{1}{r}+\frac{1}{r^\prime}=1$. $H^1_{0,\sigma}(\Omega)$ is the closure of 
$C_{0,\sigma}^\infty(\Omega)$ in the $H^1$-norm. $L^r(\Omega)$ and $H^s(\Omega)$
stand for the usual  (vector valued) Lebesgue space and the $L^2$-Sobolev space for $s\in \R$, respectively.
When $X$ is a Banach space, $\|\cdot\|_{X}$ denotes the norm on $X$. 
Furthermore, for $1<r<\infty$ and $m\in\N$, $C^m(I;X)$, $L^r(I;X)$ and $H^m(I;X)$ denote the set of 
$X$-valued $C^m$-functions, $L^r$-functions and $H^m$-functions over the interval $I\subset \R$, 
respectively.

In the present paper, we often emphasize  variables for operators, for example, 
$\Delta_x$, $\nabla_x$, $\rot_x$, and so on.

%
\subsubsection{Formulation of problem}\label{subsec;formulation}
%
To investigate the existence of  time periodic solutions of (N-S), firstly,
we shall reduce (N-S) to the  problem with the homogeneous boundary condition.

 Let us consider $\beta(t)\in H^{\frac{1}{2}}\bigl(\pt \Omega(t)\bigr)$ satisfies the general flux condition 
\begin{equation*}\tag{G.F.C.}
    \int_{\pt \Omega(t)}\beta(t)\cdot\nu(t)\,dS=0\quad \text{for }t\in \R,
\end{equation*}
where $\nu(t)$ denotes the outward unit normal vector to $\pt\Omega(t)$.
By the Bogovski\u{\i} theorem, we can take some extension $b(t) \in H^1\bigl(\Omega(t)\bigr)$ of $\beta(t)$ 
such that
$\Div\, b(t)=0$ in $\Omega(t)$ for each $ t\in \R$. See, Bogovski\u{\i} \cite{Bogovskii}, Borchers and Sohr \cite{Borchers Sohr}.
Moreover, to reduce the equations, we also consider $b \in H^1(Q_\infty)$ for a while.

Then, putting $u=v-b$, we obtain the following initial-boundary value problem:
\begin{equation*}\tag{N-S$^\prime$}
    \begin{cases}
        \pt_t u -\Delta u + b\cdot \nabla u + u\cdot\nabla b + u\cdot \nabla u +\nabla \pi=F 
        &\text{in } \bigcup\limits_{0<t<\infty} \Omega(t)\times \{t\},
        \\
        \Div\, u =0 & \text{in } \bigcup\limits_{0<t<\infty} \Omega(t)\times \{t\},
        \smallskip\\
        u=0 & \text{on }  \bigcup\limits_{0<t<\infty} \pt \Omega(t)\times\{t\},
        \\
        u(\cdot,0)=a &\text{in }  \Omega(0),
    \end{cases}
\end{equation*}
where, formally, $F=f -\pt_t b + \Delta b - b\cdot \nabla b$.

Next, we reduce our problem to that on the cylindrical domain $\widetilde{\Omega}\times(0,\infty)$
via a one-to-one transformation of a vector field from $\Omega(t)$ to $\widetilde{\Omega}$,
preserving the divergence of the vector field, see Proposition \ref{prop;IW div} below.

We introduce a vector field $\widetilde{u}$ on $\bigcup\limits_{0<t<\infty} \Omega(t)\times\{t\}$ 
are defined by, for $i=1,2,3$, 
\begin{equation}\label{eq;transform u}
    \widetilde{u}^i(y,s) :=  \sum_{\ell} 
    \frac{\pt y^i}{\pt x^\ell}u^\ell(\phi^{-1}(y,s),s):=
    \sum_{\ell}
    \frac{\pt \phi^i}{\pt x^\ell}\bigl(\phi^{-1}(y,s),s\bigr)
    u^\ell \bigl(\phi^{-1}(y,s),s\bigr), 
\end{equation}
for $(y,s)\in \widetilde{\Omega}\times(0,\infty)$, 
and a scalar function $\widetilde{\pi}$ on $\widetilde{\Omega} \times (0,\infty)$ 
defined by
\begin{equation}\label{eq;transform pi}
    \widetilde{\pi}(y,s) := \pi\bigl(\phi^{-1}(y,s),s\bigr),
\end{equation}
for $(y,s)\in \widetilde{\Omega}\times(0,\infty)$,
 when the vector field $u$ and the function $\pi$ on $\bigcup\limits_{0<t<\infty}\Omega(t)\times \{t\}$ 
 are given.

 Now, according to Inoue and Wakimoto \cite{Inoue Wakimoto} and Miyakawa and Teramoto \cite{Miyakawa Teramoto}, 
 we obtain the transformed equations on $\widetilde{\Omega}\times (0,\infty)$ 
 of the initial-boundary value problem of (N-S$^\prime$):
\begin{equation*}\tag{N-S$^*$}
    \begin{cases}
        \pt_s \widetilde{u} -L\widetilde{u}+ M\widetilde{u} + N[\widetilde{b},\widetilde{u}]
        +N[\widetilde{u},\widetilde{b}] + N [\widetilde{u},\widetilde{u}]
        +\nabla_g \widetilde{\pi} =\widetilde{F},
        & \text{in } \widetilde{\Omega}\times (0,\infty),
        \\
        \Div_y\, \widetilde{u}=0 & \text{in } \widetilde{\Omega}\times (0,\infty),
        \\
        \widetilde{u}= 0 & \text{on }\pt \widetilde{\Omega}\times(0,\infty),
        \\
        \widetilde{u}(\cdot,0)=\widetilde{a} & \text{in }\widetilde{\Omega},
    \end{cases}
\end{equation*}
where, formally, $\widetilde{F}= \widetilde{f} - \pt_s \widetilde{b} -M\widetilde{b} 
+ L\widetilde{b} -N[\widetilde{b},\widetilde{b}]$.
Here for $i=1,2,3$,
\begin{align*}
    [L\widetilde{v}\,]^i&:=[L(s)\widetilde{v}\,]^i:=\sum_{k}g^{jk}\nabla_j\nabla_k \widetilde{v}^i,
    \\
    &\Biggl(= \sum_{k,\ell} \frac{\pt }{\pt y^\ell} \biggl(
        g^{k\ell}\frac{\pt \widetilde{v}^i}{\pt y^k}
    \biggr)
    + 2\sum_{j,k,\ell} g^{k\ell} \Gamma^i_{jk} \frac{\pt \widetilde{v}^j}{\pt y^\ell}
    + \sum_{j,k,\ell} \biggl(
        \frac{\pt}{ \pt y^k}\bigl( g^{k\ell} \Gamma^i_{j\ell}\bigr)
        +\sum_{n} g^{k\ell} \Gamma^n_{j\ell}\Gamma^i_{kn}
    \biggr)\widetilde{v}^j
    \Biggr),
    \\
[M\widetilde{v}\,]^i&:=\sum_{j} 
\frac{\pt y^j}{\pt t}\nabla_j\widetilde{v}^i 
+
\sum_{j,k} \frac{\pt y^i}{\pt x^k}
\frac{\pt^2 x^k}{\pt y^j \pt s} \tilde{v}^j,
\\
N[\widetilde{u},\widetilde{v}\,]^i
&:=\sum_{j} \widetilde{u}^j \nabla_j \widetilde{v}^i,
\\
[\nabla_g \widetilde{\pi}\,]^i &:= \sum_{j} g^{ij} \frac{\pt \widetilde{\pi}}{\pt y^j},
\end{align*}
and for $i,j,k=1,2,3$,
\begin{align*}
    &g^{ij}:=\sum_{k} \frac{\pt y^i}{\pt x^k}\frac{\pt y^j}{\pt x^k},\qquad
    g_{ij}:=\sum_{k} \frac{\pt x^k}{\pt y^i}\frac{\pt x^k}{\pt y^j}, 
    \\
    &\nabla_j \widetilde{v}^{i} := \frac{\pt \widetilde{v}^i}{\pt y^j}
    + \sum_{k} \Gamma^{i}_{jk} \widetilde{v}^k,
    \qquad
    \nabla_k\nabla_j \widetilde{v}^i:=
    \sum_{k} \frac{\pt}{\pt y^k} \bigl(\nabla_j \widetilde{v}^i\bigr)
    +
    \sum_{k,\ell} \bigl(\Gamma^i_{k\ell}\nabla_j\widetilde{v}^\ell
    - \Gamma^\ell_{kj}\nabla_\ell \widetilde{v}^i
    \bigr),
    \\
    &\Gamma^{k}_{ij}:=\frac{1}{2}\sum_{\ell} g^{k\ell}\biggl(
        \frac{\pt g_{i\ell}}{\pt y^j} + \frac{\pt g_{j\ell}}{\pt y^i}
        -\frac{\pt g_{ij}}{\pt y^\ell}\biggr)
        = \sum_{\ell} \frac{\pt y^k}{\pt x^\ell} 
        \frac{\pt^2 x^\ell}{\pt y^i \pt y^j}.
\end{align*}
Here, $\bigl(g^{ij}\bigr)$ is the Riemann metric, $\nabla_j$, for $j=1,2,3$, is the covariant differentiation with 
respect to the Riemann connection induced from the metric $\bigl(g_{ij}\bigr)$. Moreover, 
we note that $\bigl(g^{ij}\bigr)^{-1}=\bigl(g_{ij}\bigr)$.

Hereafter, in this paper, since we often consider the case that the value of the variable $s$ is equal to $t$, 
we identify $s$ with $t$, except for the expression of  derivatives with respect to $s$.
Namely, we use $\pt_s \widetilde{u}(s)|_{s=t} =\pt_s\widetilde{u}(t)$ 
distinguishing the variable $s$ from the value $t$,
since
the transformed vector of $\pt_t u$ corresponds to $\pt_s \widetilde{u}+M\widetilde{u}$ by the 
diffeomorphism. 

According to the transformation of vector fields, we introduce the following inner product 
on $L^2(\widetilde{\Omega})$ 
for each $t\in \R$
defined by 
\begin{equation}\label{eq;inner product L2}
    \langle \widetilde{u}, \widetilde{v} \rangle_t :=
    \int_{\widetilde{\Omega}} \sum_{i,j} g_{ij}(y,t) \widetilde{u}^i(y) \widetilde{v}^j(y) J(t)\,dy
    \quad \text{for all } \widetilde{u}, \widetilde{v} \in L^2(\widetilde{\Omega}).
\end{equation}
Furthermore, for each $t\in \R$, we adopt the following inner product 
on $H^1_{0}(\widetilde{\Omega})$ defined by
\begin{equation}\label{eq;inner product H1}
    \langle \nabla_g \widetilde{u}, \nabla_g \widetilde{v} \rangle_t
    :=
    \int_{\widetilde{\Omega}} 
    \sum_{i,j,k,\ell} g_{ij}(y,t)g^{k\ell} (y,t) \nabla_k \widetilde{u}^{i} \nabla_\ell \widetilde{v}^j 
    J(t)\,dy  
    \quad \text{for all } \widetilde{u}, \widetilde{v} \in H^1_0(\widetilde{\Omega}).
\end{equation}
Here, for each $t\in\R$ we note that
\begin{align}\label{eq;equivalence innerproduct L2}
    \langle \widetilde{u}, \widetilde{v} \rangle_t&=
    \int_{\Omega(t)} \sum_{i} u^i(x)v^i(x)\,dx=(u,v),
    \\
    \label{eq;equivalence innerproduct H1}
    \langle \nabla_g \widetilde{u}, \nabla_g \widetilde{v} \rangle_t&=
    \int_{\Omega(t)} \sum_{i,j} \frac{\pt u^i}{\pt x^j}(x)\frac{\pt v^i}{\pt x^j}(x) \,dx
    =(\nabla u, \nabla v),
\end{align}
where vector fields $\widetilde{u}$, $\widetilde{v}$ on $\widetilde{\Omega}$ and 
$u$, $v$ on $\Omega(t)$ are as in \eqref{eq;identify}, respectively.\medskip

Then, we state the definition of weak solutions of the initial-boundary value problem (N-S$^\prime$).
\begin{definition}[Weak solution]
    Let $T>0$. Let $a \in L_\sigma^2\bigl(\Omega(0)\bigr)$, $f \in L^2(Q_T)$
    and let $b$ be a vector field on $Q_T$ which satisfies $\Div\, b(t)=0$ in $\Omega(t)$ for all $t\in [0,T]$ and $\widetilde{b} \in C^1\bigl([0,T];H^1(\widetilde{\Omega})\bigr)$.
    A measurable (vector valued) function $u$ on $Q_T$ is called a weak solution of (N-S$^\prime$), 
    if, under the transform \eqref{eq;transform u},
    \begin{itemize}
        \item[(i)] $\widetilde{u} \in L^\infty \bigl(0,T;L^2_\sigma(\widetilde{\Omega})\bigr)
        \cap L^2\bigl(0,T;H^1_{0,\sigma}(\widetilde{\Omega})\bigr)$,
        \item[(ii)] the relation 
        \begin{multline*}
            \int_0^T \Bigl(
            -\bigl\langle \widetilde{u}(\tau), \pt_s \widetilde{\Psi}(\tau)\bigr\rangle_\tau
            +\bigl\langle \nabla_g\widetilde{u}(\tau), \nabla_g \widetilde{\Psi}(\tau)\bigr\rangle_\tau
            \\+\bigl\langle N\bigl[\widetilde{b}(\tau), \widetilde{u}(\tau)\bigr], \widetilde{\Psi}(\tau)\bigr\rangle_\tau
            +\bigl\langle N\bigl[\widetilde{u}(\tau), \widetilde{b}(\tau)\bigr], \widetilde{\Psi}(\tau)\bigr\rangle_\tau
            +\bigl\langle N\bigl[\widetilde{u}(\tau), \widetilde{u}(\tau)\bigr], \widetilde{\Psi}(\tau)\bigr\rangle_\tau
            \Bigr)\,d\tau
            \\
            =\bigl\langle \widetilde{a}, \widetilde{\Psi}(0) \bigr\rangle_0
            + \int_0^T \bigl\langle \widetilde{F}(\tau),\widetilde{\Psi}(\tau)\bigr\rangle_\tau \,d\tau,
        \end{multline*}
        holds for all $\widetilde{\Psi} \in C^1_0\bigl([0,T);H^1_{0,\sigma}(\widetilde{\Omega})\bigr)$,
        where 
        \begin{equation*}\bigl\langle\widetilde{F}(t), \widetilde{\Psi}(t) \bigr\rangle_t:=
        \bigl\langle \widetilde{f}(t)- \pt_s \widetilde{b}(t)-M\widetilde{b}(t) 
        -N\bigl[\widetilde{b}(t),\widetilde{b}(t)\bigr],\widetilde{\Psi}(t)\bigr\rangle_t 
        -\bigl\langle \nabla_g \widetilde{b}(t),\nabla_g\widetilde{\Psi}(t)\bigr\rangle_t.
        \end{equation*}
    \end{itemize}
\end{definition}
\begin{remark}\begin{itemize}
    \item[(a)] The assumption $\widetilde{b}\in C^1\bigl([0,T];H^1(\widetilde{\Omega})\bigr)$ 
    implies $b \in H^{1}(Q_T)$. Hence, the above definition make sense also on $Q_T$.
    \item[(b)]
    $\widetilde{b} \in H^1\bigl(0,T;H^1(\widetilde{\Omega})\bigr)$ is enough for the above definition.
\end{itemize} 
\end{remark}
%
\subsection{Domain dependence of the Helmholtz-Weyl decomposition}\label{subsec;HW}
%
In this section, we state our main results on the domain dependence of the Helmholtz-Weyl decompositions.

Let us introduce the Helmholtz-Weyl decomposition on a bounded domain $\Omega$, 
according to Kozono and Yanagisawa \cite{Kozono Yanagisawa IUMJ}.
For this purpose, we introduce the following function spaces,
\begin{align*}
    V_{\mathrm{har}}(\Omega)
    &:= \bigl\{h \in C^\infty(\overline{\Omega})\,;\, 
    \Div\,h=\rot\,h=0 \text{ in }\Omega,\;\; h\times \nu =0 \text{ on } \pt \Omega\bigr\},
    \\
    X_{\mathrm{har}}(\Omega) &:=\bigl\{ u \in C^\infty(\overline{\Omega})\,;\, \Div\, u= \rot\, u =0
    \text{ in }\Omega, \;\; u\cdot \nu=0 \text{ on } \pt\Omega \bigr\},
    \\
    X_\sigma^2(\Omega) &:=\bigl\{ u \in L^2(\Omega)\,;\, \Div\,u=0 \text{ in } \Omega, \;\;
    \rot\, u \in L^2(\Omega), \;\; u\cdot \nu =0 \text{ on } \pt\Omega\bigr\},
    \\
    Z^2_\sigma(\Omega) &:=\bigl\{ 
        u \in X^2_\sigma(\Omega)\,;\, u \perp X_{\textrm{har}}(\Omega)\bigr\},
\end{align*}
for $\Omega \subset \R^3$ which is a bounded domain with the smooth boundary, where 
$\nu$ denotes the outward unit normal vector to $\pt \Omega$ and we consider 
$u\cdot \nu=0$ in weak sense, i.e., 
\begin{equation*}
    \langle u \cdot \nu , \psi \rangle_{\pt \Omega}
    := (u, \nabla \psi ) + (\Div\, u, \psi) 
    \quad \text{for all } \psi \in C^\infty(\overline{\Omega}). 
\end{equation*}

Kozono and Yanagisawa \cite{Kozono Yanagisawa IUMJ}
established the Helmholtz-Weyl decomposition for $L^r$-vector fields, for $1<r<\infty$.
In our situation, since we deal with only $H^1$-framework, we prepare the following proposition,
excerpting  the $H^1$-part from \cite[Theorem 2.1 and 2.4]{Kozono Yanagisawa IUMJ}. 
\begin{proposition}[{\cite[Theorem 2.1 and 2.4]{Kozono Yanagisawa IUMJ}}] \label{prop;KY}
    \begin{itemize}
        \item[(i)] For every $b\in H^1(\Omega)$ which satisfies $\Div\, b=0$ in $\Omega$, then 
        there uniquely exist a harmonic vector field $h\in V_{\mathrm{har}}(\Omega)$
        and a vector potential $w \in Z^2_\sigma(\Omega)\cap H^2(\Omega)$ such that
        \begin{equation}\label{eq;KY dec}
            h + \rot \, w =b \quad \text{in } \Omega,
        \end{equation}
        {with the estimate\,}
        \begin{equation*}
            \| h \|_{H^1(\Omega)} + \|w\|_{H^2(\Omega)} \leq C \| b \|_{H^1(\Omega)},
        \end{equation*}
        where the constant $C>0$ depends on $\Omega$, but is independent of $b$.
        \item[(ii)] For every $f \in H^1(\Omega)$, there uniquely exist
        a harmonic vector field $h\in V_{\mathrm{har}}(\Omega)$,
        a vector potential $w \in Z^2_\sigma(\Omega)\cap H^2(\Omega)$ 
        and a scalar potential $p \in H^2(\Omega)\cap H^1_0(\Omega)$ such that
        \begin{equation*}
            h + \rot \, w + \nabla p=f  \quad \text{in } \Omega,
        \end{equation*}
        {with the estimate\,}
        \begin{equation*}
            \| h \|_{H^1(\Omega)} + \|w\|_{H^2(\Omega)} + \|p\|_{H^2(\Omega)} \leq C \| f \|_{H^1(\Omega)},
        \end{equation*}
        where the constant $C>0$  depends on $\Omega$, but is independent of $f$.
    \end{itemize}
    
\end{proposition}
\begin{remark}
    In the decomposition, the uniqueness and regularity of the vector potential $w$ within $Z^2_\sigma(\Omega)$
    immediately follows from \cite[Theorem 2.4]{Kozono Yanagisawa IUMJ}, since $w \perp X_{\mathrm{har}}(\Omega)$.
\end{remark}

Here, the following theorem is on the domain dependence, i.e., time dependence of the Helmholtz-Weyl decomposition.
\begin{theorem}\label{thm;time depend HW dec div}
    Let $\bigl\{\Omega(t)\bigr\}_{t\in\R}$ be as in Assumption.
    Let $T>0$ and let $b(t) \in H^1\bigl(\Omega(t)\bigr)$ 
    satisfy $\Div\,b(t)=0$ in $\Omega(t)$ for $t\in [0,T]$.
    Suppose the Helmholtz-Weyl decomposition of $b(t)$ in $\Omega(t)$, i.e.,
    \begin{equation*}
        h(t) + \rot\, w(t) = b(t)\quad \text{in } \Omega(t),
    \end{equation*}
    where $h(t)\in V_{\mathrm{har}}\bigl(\Omega(t)\bigr)$ and
    $w(t) \in Z_\sigma^2\bigl(\Omega(t)\bigr) \cap H^2\bigl(\Omega(t)\bigr)$
    are uniquely determined, for all $t\in [0,T]$.
    \begin{itemize}
        \item[(i)] If the transformed vector field $\widetilde{b}(t)$ on $\widetilde{\Omega}$
        from $b(t)$ on $\Omega(t)$, defined by \eqref{eq;transform u} satisfies
        \begin{equation*}
            \widetilde{b} \in C\bigl([0,T];H^1(\widetilde{\Omega})\bigr),
        \end{equation*}
        then the transformed vector fields $\widetilde{h}(t)$
        and
        $\widetilde{w}(t)$ on $\widetilde{\Omega}$ from $h(t)$ and $w(t)$ on $\Omega(t)$, respectively, 
        satisfy
        \begin{equation*}
            \widetilde{h} \in C\bigl([0,T];H^1(\widetilde{\Omega})\bigr)
            \quad\text{and}\quad
            \widetilde{w} \in C\bigl([0,T];H^2(\widetilde{\Omega})\bigr).
        \end{equation*}
        \item[(ii)] Furthermore, if $\widetilde{b} \in C^1\bigl([0,T];H^1(\widetilde{\Omega})\bigr)$
        then
        \begin{equation*}
            \widetilde{h} \in C^1\bigl([0,T];H^1(\widetilde{\Omega})\bigr)
            \quad\text{and}\quad
            \widetilde{w} \in C^1\bigl([0,T];H^2(\widetilde{\Omega})\bigr).
        \end{equation*}
    \end{itemize}
\end{theorem}
\begin{remark}
    \begin{itemize}
        \item[(a)] The time dependence of $\widetilde{h}(t)$ and $\widetilde{w}(t)$ 
        comes from the variation of the domain and the time regularity of data $\widetilde{b}(t)$.
        In particular, for the harmonic vector fields, since the basis 
        $\nabla q_1(t),\dots, \nabla q_K(t)$ of $V_{\mathrm{har}}(\Omega)$ are determined
        only by the domain 
        and since the contribution from $b(t)$ appears only in the coefficients, 
         we can separately discuss these two aspects for $\widetilde{h}(t)$.
        \item[(b)] We can extend Theorem \ref{thm;time depend HW dec div}
        to $b(t) \in W^{1,r}\bigl(\Omega(t)\bigr)$ with $\Div\, b(t) =0$ in $\Omega(t)$, $t\in [0,T]$,
        for all $1 <r <\infty$, just by replacing the norm by one on the $L^r$-Sobolev spaces.
        Namely, $\widetilde{h} \in C^m\bigl([0,T]; W^{1,r}(\widetilde{\Omega})\bigr)$
        and $\widetilde{w} \in C^m \bigl([0,T];W^{2,r}(\widetilde{\Omega})\bigr)$, provided 
        $\widetilde{b} \in C^m \bigl( [0,T];W^{1,r}(\widetilde{\Omega})\bigr)$ with some $m \in \{0,1\}$.
    \end{itemize}
\end{remark}

Since we can decompose $f(t)=b(t)+\nabla p(t)$ with $\Div\, b(t)=0$ with $\Omega(t)$ 
for every $f\in H^1\bigl(\Omega(t)\bigr)$,
investigating the time dependence of $p(t)$, 
we immediately obtain the following generalization. 
\begin{corollary}\label{cor;time depend HW dec gen}
    Let $\bigl\{\Omega(t)\bigr\}_{t\geq 0}$ be as in Assumption.
    Let $T>0$ and let $f(t) \in H^1\bigl(\Omega(t)\bigr)$ for $t \in [0,T]$.
    Suppose the Helmholtz-Weyl decomposition of $f(t)$ in $\Omega(t)$, i.e.,
    \begin{equation*}
        h(t) + \rot\, w(t) + \nabla p(t) = f(t) \quad \text{in } \Omega(t),
    \end{equation*}
    where $h(t)\in V_{\mathrm{har}}\bigl(\Omega(t)\bigr)$,
    $w(t) \in Z_\sigma^2\bigl(\Omega(t)\bigr) \cap H^2\bigl(\Omega(t)\bigr)$
    and
    $p(t) \in H^2\bigl(\Omega(t)\bigr)\cap H^1_{0}\bigl(\Omega(t)\bigr)$
    are uniquely determined, for all $t\in [0,T]$.
    \begin{itemize}
        \item[(i)] If the transformed vector filed $\widetilde{f}(t)$ in $\widetilde{\Omega}$ from $f(t)$,
        by \eqref{eq;transform u} satisfies
        \begin{equation*}
            \widetilde{f} \in C\bigl([0,T];H^1(\widetilde{\Omega})\bigr),
        \end{equation*}
        then the transformed vector fields 
        $\widetilde{h}(t)$, 
        $\widetilde{w}(t)$ on $\widetilde{\Omega}$ by \eqref{eq;transform u} and 
        the transformed function $\widetilde{p}(t)$ on $\widetilde{\Omega}$ by \eqref{eq;transform pi} satisfy
        \begin{equation*}
            \widetilde{h} \in C\bigl([0,T];H^1(\widetilde{\Omega})\bigr),
            \quad
            \widetilde{w} \in C\bigl([0,T];H^2(\widetilde{\Omega})\bigr)
            \quad\text{and}\quad
            \widetilde{p} \in C\bigl([0,T];H^2(\widetilde{\Omega})\cap H^1_0(\widetilde{\Omega})\bigr).
        \end{equation*}
        \item[(ii)] Furthermore, if $\widetilde{f}\in C^1\bigl([0,T];H^1(\widetilde{\Omega})\bigr)$
        then
        \begin{equation*}
            \widetilde{h} \in C^1\bigl([0,T];H^1(\widetilde{\Omega})\bigr),
            \quad
            \widetilde{w} \in C^1\bigl([0,T];H^2(\widetilde{\Omega})\bigr)
            \quad\text{and}\quad
            \widetilde{p} \in C^1\bigl([0,T];H^2(\widetilde{\Omega})\cap H^1_0(\widetilde{\Omega})\bigr).
        \end{equation*}
    \end{itemize}
\end{corollary}
\subsection{Existence of time periodic solutions to (N-S$^\prime$)}
Under Assumption, it was shown that $\dim V_{\mathrm{har}}\bigl(\Omega(t)\bigr)=K$
for each $t\in \R$,
by \cite{Kozono Yanagisawa IUMJ}. 
Moreover, 
$V_{\mathrm{har}}\bigl(\Omega(t)\bigr)$ is spanned by 
the gradient flows of the solutions of the following Laplace equations with the 
Dirichlet boundary condition, i.e., for $k=1,\dots,K$, 
\begin{equation}\label{eq;Harmonic basis equation}
    \begin{cases}
        \Delta q_k(t) = 0 & \text{in } \Omega(t),
        \\
        q_k(t)|_{\Gamma_j(t)} = \delta_{k,j} &\text{for } j=0,\dots,K.
    \end{cases}
\end{equation} 
By the Schmidt orthonormalization, the orthonormal basis $\eta_1(t),\dots,\eta_K(t)$
of $V_{\mathrm{har}}\bigl(\Omega(t)\bigr)$ in  $L^2$-sense is obtained
with the relation $\displaystyle \eta_j(t)=\sum_{k=1}^K\alpha_{jk}(t)\nabla q_k(t)$, 
$j=1,\dots,K$ with some coefficients $\alpha_{jk}(t) \in \R$ for $j,k=1,\dots,K$.
For more detail, see \eqref{eq;cons Vhar}, below.
\smallskip

Then, we state our existence theorem 
of the time periodic solutions of (N-S$^\prime$).
\begin{theorem}\label{thm;time periodic}
    Let $T>0$  and let $\bigl\{\Omega(t)\bigr\}_{t\in\R}$ be as in Assumption.
    Let $f$ be external force on $Q_\infty$ with $ f\in L^2(Q_T)$ 
    and let $\beta(t) \in H^{\frac{1}{2}}
    \bigl(\pt\Omega(t)\bigr)$ with (G.F.C.) for all $t \in \R$ satisfy
    $\widetilde{\beta} \in C^1\bigl(\R;H^{\frac{1}{2}}(\pt \widetilde{\Omega})\bigr)$.
    Furthermore, we assume $\phi^{-1}(\cdot,t+T)=\phi^{-1}(\cdot,t)$, $f(t+T)=f(t)$ and $\beta(t+T)=\beta(t)$ for all $t\in \R$.
    If
    \begin{equation}\label{eq;smallness of harmonic}
        \sup_{0\leq t \leq T} 
        \left\| \sum_{\ell,k=1}^K \alpha_{k\ell}(t)\left(
            \int_{\Gamma_\ell(t)} \beta(t)\cdot \nu(t)\,dS
        \right) \eta_k(t)\right\|_{L^3(\Omega(t))} < \frac{1}{C_s},
    \end{equation}
    then there exist a solenoidal extension $b(t)\in H^1\bigl(\Omega(t)\bigr)$ of $\beta(t)$
    with $ \widetilde{b} \in C^1\bigl([0,T];H^1(\widetilde{\Omega})\bigr)$,
    an initial data $a \in L^2\bigl(\Omega(0)\bigr)$ and a weak solution $u$ of (N-S$^\prime$) such that
    $u(T)= a$ in $L^2_\sigma\bigl(\Omega(0)\bigr)$.
    Here,  $C_s= 3^{-\frac{1}{2}}2^{\frac{2}{3}}\pi^{-\frac{2}{3}}$ is the best constant of the Sobolev embedding 
    $H^1_0(\Omega) \hookrightarrow L^6(\Omega)$.
\end{theorem}
\begin{remark}
    \begin{itemize}
    \item[(a)] We note that $\alpha_{jk}(t)$ and $\eta_k(t)$ are determined only by $\Omega(t)$ and $C_s$ is an absolute constant.
    Hence, the condition \eqref{eq;smallness of harmonic} are formulated only by the given domain $\Omega(t)$ 
    and data $\beta(t)$ on $\pt \Omega(t)$, independent of the diffeomorphism or the fixed domain $\widetilde{\Omega}$.
    \item[(b)] Compared to the previous results of the second author \cite{Okabe JEE},
    since $\alpha_{k\ell}$ and $\eta_k$ depend on $t$, the assumption \eqref{eq;smallness of harmonic} gives us the relation of
    the balances of the domain variations and the flux.
    Indeed, we shall give the following corollaries.
    \end{itemize}
\end{remark}
\begin{corollary}
    Let $T>0$ and let $\{\Omega(t)\}_{t\in\R}$ be as in Assumption with $\phi^{-1}(\cdot,t+T)=\phi^{-1}(\cdot,t)$
    for all $t\in\R$.
    Assume that $\Omega(t)=\lambda(t)\Omega(0)=\{\lambda(t)x \in \R^3\,;\,x \in \Omega(0)\}$ with some
    $\lambda \in C^\infty(\R)$ satisfying $\lambda(t)>0$ and $\lambda(t+T)=\lambda(t)$, for all $t\in\R$.
    Let $\beta(t) \in H^{\frac{1}{2}}\bigl(\pt\Omega(t)\bigr)$ with (G.F.C.) for all $t\in \R$ satisfy
    $\widetilde{\beta} \in C^{1}\bigl(\R;H^{\frac{1}{2}}(\widetilde{\Omega})\bigr)$ 
    and $\beta(t+T)=\beta(t)$ for all $t\in\R$.
    If 
    \begin{equation}\label{eq;cor1}
        \frac{1}{\lambda(t)} \sum_{k=1}^K \left|
            \int_{\Gamma_k(t)} \beta(t)\cdot \nu(t)\,dS
        \right|<C_0 \quad \text{for all } t \in [0,T],
    \end{equation}
    then $\beta$ satisfies the assumption \eqref{eq;smallness of harmonic} in Theorem \ref{thm;time periodic}. Here, the constant $C_0>0$ depends only on $\Omega(0)$.
\end{corollary}
\begin{remark}
        We put $q_k(y,t):= q_k\bigl(y/\lambda(t)\bigr)$ for $y \in \Omega(t)$,
        where $q_k$  is the solution of 
        \eqref{eq;Harmonic basis equation} on $\Omega(0)$.
        Then we easily see that $q_k(t)$ is the solution of \eqref{eq;Harmonic basis equation}
        on $\Omega(t)$ and \eqref{eq;cor1} follows from 
        $\alpha_{jk}(t) = \frac{1}{\sqrt{\lambda(t)}} \alpha_{jk}(0)$ and
        $\eta_k(t)=\frac{1}{\lambda(t)\sqrt{\lambda{(t)}}} \eta_k(0)$.
\end{remark}
\begin{corollary}\label{cor;beta2}
    Let $T>0$ and let $\Omega(t)$ be an annulus $A_{R_0,R_1}(t):=\{x \in \R^3\,;\, 0<R_1(t)<|x|<R_0(t)\}$
    for $t\in \R$ with some $R_0,R_1\in C^\infty(\R)$ 
    satisfying $R_0(t+T)=R_0(t)$ and $R_1(t+T)=R_1(t)$, for all $t\in\R$.
    Let $\beta(t) \in H^{\frac{1}{2}}\bigl(\pt\Omega(t)\bigr)$ with (G.F.C.) for all $t\in \R$ satisfy
    $\widetilde{\beta} \in C^{1}\bigl(\R;H^{\frac{1}{2}}(\widetilde{\Omega})\bigr)$ 
    and $\beta(t+T)=\beta(t)$ for all $t\in\R$.
    If 
    \begin{equation*}
       2^{-\frac{2}{3}}3^{-\frac{1}{3}}\pi^{-\frac{2}{3}}\left( \frac{1}{R_1(t)^3}-\frac{1}{R_0(t)^3}\right)^{\frac{1}{3}} \left|
            \int_{|x|=R_1(t)} \beta(t)\cdot \nu(t)\,dS
        \right|< \frac{1}{C_s} \quad \text{for all } t \in [0,T],
    \end{equation*}
    then $\beta$ satisfies the assumption \eqref{eq;smallness of harmonic} in Theorem \ref{thm;time periodic}. 
\end{corollary}
\begin{remark}
    \begin{itemize}
        \item[(a)] For the stationary problem of (N-S), the same situation was considered 
        by Kozono and Yanagisawa \cite{Kozono Yanagisawa MathZ}.
    \item[(b)] Solving \eqref{eq;Harmonic basis equation}, we explicitly see that 
    $\displaystyle \nabla q_1(x,t)=-\left(\frac{1}{R_1(t)}-\frac{1}{R_0(t)}\right)^{-1} \frac{x}{|x|^2}$
    is a basis of $V_{\mathrm{har}}\bigl(\Omega(t)\bigr)$. So, we can transform \eqref{eq;smallness of harmonic}
    into one in Corollary \ref{cor;beta2} by a direct calculation.
    \end{itemize}
\end{remark}
\section{Preliminaries}
%
In this section, we prepare the three topics. 
Firstly, we prepare the fundamental facts related to the diffeomorphism and consider that
the divergence of a vector field is maintained via the the transformation as in \eqref{eq;identify}.
Next, we investigate the equivalence of the function spaces on $\Omega(t)$ and $\widetilde{\Omega}$.
Finally, we express the operator $\Rot(t)$ of transformed rotation and derive Leray's inequality for the 
transformed rotation $\Rot(t)\,\widetilde{w}(t)$.
%
\subsection{Fundamental facts derived from the diffeomorphism}
We note that, by the notation 
$\phi^{-1}(y,s)=\bigl(\phi^{-1}_1(y,s),\phi^{-1}_2(y,s),\phi^{-1}_3(y,s)\bigr)$ in \eqref{eq;diffeo inverse},
\begin{equation}\label{eq;Kronecker}
    \sum_{\ell} \frac{\pt \phi^{-1}_i}{\pt y^\ell} 
    \bigl(\phi(x,t),t\bigr) 
    \frac{\pt \phi^\ell}{\pt x^j}(x,t) =\delta_{i,j},
    \quad
    \sum_{\ell} 
    \frac{\pt \phi^i}{\pt x^\ell}\bigl(\phi^{-1}(y,s),s\bigr)
    \frac{\pt \phi^{-1}_\ell}{\pt y^j}(y,s)=\delta_{i,j}
\end{equation}
for all $(x,t) \in \overline{Q}_\infty$ and for all $(y,s) \in \overline{\widetilde{Q}}_\infty$,
respectively.
Hereafter, we also write \eqref{eq;Kronecker} as 
$\displaystyle\sum_{\ell} \frac{\pt x^i}{\pt y^\ell}
\frac{\pt y^\ell}{\pt x^j}
=\delta_{i,j}$ and $\displaystyle\sum_{\ell}
\frac{\pt y^i}{\pt x^\ell}\frac{\pt x^\ell}{\pt y_j}=\delta_{i,j}$.
\begin{proposition}\label{prop;IW diffeo}
    We have followings:
    \begin{align}
        \label{eq;IW}
        \eqref{eq;J} &\text{ implies } \sum_{\ell } \Gamma_{i\ell}^\ell=0, \quad i=1,2,3,
        \\
        \label{eq;dsJ}
        \frac{\pt}{\pt s}J(s)&=\sum_{k,\ell} 
        \frac{\pt y^k}{\pt x^\ell}
        \frac{\pt^2 x^{\ell}}{\pt y^k \pt s}
         J(s),
        \\
        \label{eq;dsg}
        \sum_{i,j}\frac{\pt g_{ij}}{\pt s}&=
        \sum_{i,j,k,\ell} \left(
            g_{jk} \frac{\pt y^k}{\pt x^\ell}
            \frac{\pt^2 x^{\ell}}{\pt y^i \pt s}
            +
            g_{ik} \frac{\pt y^k}{\pt x^\ell}
            \frac{\pt^2 x^{\ell}}{\pt y^j \pt s}
        \right).
    \end{align}
\end{proposition}
\begin{proof}
    \eqref{eq;IW} is easily obtained by a similar way to Inoue and Wakimoto 
    \cite[Proposition 2.3]{Inoue Wakimoto}. 
    \eqref{eq;dsJ} is also derived by the similar way. 
    Indeed, let $\Delta_{k\ell}$ be the $(k,\ell)$-cofactor of Jacobi matrix
    $
    \displaystyle 
    \left(\frac{\pt \phi_k^{-1}}{\pt y^\ell}\right)$.
    From \eqref{eq;Kronecker}, we have 
    \begin{equation*}
        \frac{1}{J(s)} \Delta_{k\ell}=\frac{\pt y^\ell}{\pt x^k},
        \quad k,\ell=1,2,3.
    \end{equation*}
    Hence,
    \begin{equation*}
        \frac{\pt }{\pt s}J(s)=
        \sum_{k,\ell} \frac{\pt}{\pt s} \left(\frac{\pt x^k}{\pt y^\ell}\right) \Delta_{k\ell}
        =
        \sum_{k,\ell} \frac{\pt y^\ell}{\pt x^k} 
        \frac{\pt^2 x^k}{\pt y^\ell \pt s} J(s).
    \end{equation*}
    Finally, for \eqref{eq;dsg}, we see that
    $\displaystyle 
    \frac{\pt g_{ij}}{\pt s}
    =\sum_{k} \left(\frac{\pt^2 x^k}{\pt y^i\pt s}\frac{\pt x^k}{\pt y^j}
    +
    \frac{\pt x^k}{\pt y^i}\frac{\pt^2 x^k}{\pt y^j\pt s}\right)$.
    Here, we have
    \begin{equation*}
        \begin{split}
            \sum_{k} \frac{\pt^2 x^k}{\pt y^i \pt s}\frac{\pt x^k}{\pt y^j}
            &=
            \sum_{k,\ell} \delta_{k,\ell}
            \frac{\pt^2 x^\ell}{\pt y^i \pt s}\frac{\pt x^k}{\pt y^j}
            =\sum_{k,\ell,m}
            \frac{\pt x^k}{\pt y^m} \frac{\pt y^m}{\pt x_\ell}
            \frac{\pt^2 x^\ell}{\pt y^i \pt s}\frac{\pt x^k}{\pt y^j}
            \\
            &=
            \sum_{\ell,m}
            g_{jm}  \frac{\pt y^m}{\pt x_\ell}
            \frac{\pt^2 x^\ell}{\pt y^i \pt s}.
        \end{split}
    \end{equation*}
    Hence, we have the desired identity \eqref{eq;dsg}. 
\end{proof}
Furthermore, we also introduce the following identities related to Proposition \ref{prop;IW diffeo}, 
which are effective to deal with operators $L(t)$ and $M$ in Section \ref{subsec;formulation}.
\begin{remark} \rm
\begin{itemize}
\item[(i)]
    By the same argument to derive \eqref{eq;IW}, we also have that
    \eqref{eq;J} implies 
    \begin{equation}\label{eq;IWR}
        \sum_{k,\ell} \frac{\pt x^\ell}{\pt y^k}
        \frac{\pt^2 y^k}{\pt x^\ell \pt x^i}=0, \quad i=1,2,3.
    \end{equation}
\item[(ii)]
Using \eqref{eq;Kronecker} and \eqref{eq;IW}, we can derive
\begin{equation}\label{eq;converse}
 \sum_{k,\ell} \frac{\pt x^\ell}{\pt y^k} 
 \frac{\pt^2 y^k}{\pt x^\ell \pt t}
 =
 -
 \sum_{k,\ell} \frac{\pt y^k}{\pt x^\ell}
 \frac{\pt^2 x^\ell}{\pt y^k \pt s}.
\end{equation}
\end{itemize}
\end{remark}
Let us consider that vector fields $u$ on $\Omega(t)$ and $\widetilde{u}$ on $\widetilde{\Omega}$ 
which are connected by the coordinate transformation \eqref{eq;identify}.
Then the divergence is preserved under the transformation provided \eqref{eq;J} holds.
\begin{proposition}[\cite{Inoue Wakimoto}]\label{prop;IW div}
    Under the assumption \eqref{eq;J},  for vector fields $u$ on $\Omega(t)$ 
    and $\widetilde{u}$ on $\widetilde{\Omega}$ which satisfy \eqref{eq;identify}
    it holds that
    \begin{equation*}
        \Div_x u = \Div_y \widetilde{u}. 
    \end{equation*}
\end{proposition}
\begin{proof}
    By \eqref{eq;IW} we see that
    \begin{equation*}
        \begin{split}
            \sum_{i} \frac{\pt }{\pt x^i} u^i
            &=
            \sum_{i} \sum_{k} \frac{\pt y^k}{\pt x^i}
            \frac{\pt }{\pt y^k} \sum_{\ell} \frac{\pt x^i}{\pt y^\ell} 
            \widetilde{u}^\ell
            \\
            &=\sum_{i,k,\ell}
            \left(
                \frac{\pt y^k}{\pt x^i}\frac{\pt^2 x^i}{\pt y^\ell \pt y^k} 
                \widetilde{u}^\ell
                +
                \frac{\pt y^k}{\pt x^i}
                \frac{\pt x^i}{\pt y^\ell} \frac{\pt \widetilde{u}^\ell}{\pt y^k}
            \right)
            \\
            &=
            \sum_{k,\ell} \Gamma^k_{\ell k} \widetilde{u}^\ell 
            + 
            \sum_{k,\ell} \delta_{k,\ell} \frac{\pt \widetilde{u}^\ell}{\pt y^k}
            =
            \sum_{k} \frac{\pt \widetilde{u}^k}{\pt y^k}.
        \end{split}
    \end{equation*}
\end{proof}
Finally, related to the weak form of (N-S$^\prime$), we introduce the following formula 
about time differentiation.
\begin{proposition}\label{prop;duds}
    Under the assumption \eqref{eq;J}, it holds that
    \begin{equation}\label{eq;dudsv}
        \left\langle \frac{\pt \widetilde{u}}{\pt s}, \widetilde{v} \right\rangle_s
        +
        \langle M\widetilde{u}, \widetilde{v} \rangle_s
        =\frac{d }{ds} \langle \widetilde{u}, \widetilde{v}\rangle_s 
        -\left\langle \widetilde{u}, \frac{\pt \widetilde{v}}{\pt s}\right\rangle_s 
        -\langle \widetilde{u}, M\widetilde{v}\rangle_s 
        \quad \text{for } s\in \R,
    \end{equation}
    for all $\widetilde{u},\widetilde{v} \in C^1\bigl( \R;H^1_0(\widetilde{\Omega})\bigr)$.
    Especially, if $\widetilde{u}=\widetilde{v}$ in \eqref{eq;dudsv}, we have
    \begin{equation}\label{eq;ddsnorm}
        \frac{1}{2}\frac{d}{ds}\| {u}(s) \|_{L^2(\Omega(s))} 
        =\frac{1}{2}\frac{d}{ds} \langle \widetilde{u},\widetilde{u} \rangle_s
        = \left\langle \frac{\pt \widetilde{u}}{\pt s}, \widetilde{u} \right\rangle_s
        + \langle M\widetilde{u} , \widetilde{u} \rangle_s
    \end{equation}
    for all $s \in \R$, where $u(s)$ on $\Omega(s)$ is transformed from $\widetilde{u}(s)$ via \eqref{eq;identify}.
\end{proposition}
\begin{remark} 
    Especially, for the case $s=t$, we note that
    \begin{equation*}\frac{1}{2}\frac{d }{dt}\| u(t) \|_{L^2(\Omega(t))} 
        =\left\langle \frac{\pt \widetilde{u}}{\pt s}(t),\widetilde{u}(t) \right\rangle_t
        +\bigl\langle M \widetilde{u}(t),\widetilde{u}(t) \bigr\rangle_t. 
    \end{equation*}
\end{remark}
\begin{proof}
    We recall 
    $[M \widetilde{u}]^i =
    \displaystyle \sum_{k,\ell} \Gamma^{i}_{k\ell}
    \frac{\pt y^\ell}{\pt t}\widetilde{u}^k
    + \sum_{k,\ell} 
    \frac{\pt y^i}{\pt x^\ell}\frac{\pt^2 x^\ell}{\pt y^k \pt s}\widetilde{u}^k
    + \sum_{k} \frac{\pt y^k}{\pt t} \frac{\pt \widetilde{u}^i}{\pt y^k}$ 
    for $i=1,2,3$.
    By a direct calculation, we obtain
    \begin{multline}\label{eq;Muv}
        \langle M\widetilde{u}, \widetilde{v} \rangle_s  
            +\langle \widetilde{u},M\widetilde{v} \rangle_s 
            \\
            =\sum_{i,j,k,\ell} \int_{\widetilde{\Omega}} g_{ij}
            \left(
                \frac{\pt y^i}{\pt x^\ell} \frac{\pt^2 x^\ell}{\pt y^k \pt s}
                \widetilde{u}^k \widetilde{v}^j
                +
                \frac{\pt y^k}{\pt x^\ell} \frac{\pt^2 x^\ell}{\pt y^k \pt s}
                \widetilde{u}^i \widetilde{v}^j
                +
                \frac{\pt y^j}{\pt x^\ell} \frac{\pt^2 x^\ell}{\pt y^k \pt s}
                \widetilde{u}^i \widetilde{v}^k
            \right) J(s)\,dy,
    \end{multline} 
    making use of 
    \begin{equation*}
        \int_{\widetilde{\Omega}} 
        \sum_{i,j} g_{ij}\sum_{k} \frac{\pt y^k}{\pt t} \frac{\pt \widetilde{u}^i}{\pt y^k}\widetilde{v}^jJ(s)\,dy
        =
        -\sum_{i,j,k} \int_{\widetilde{\Omega}} 
        \frac{\pt}{\pt y^k} \left( 
            g_{ij}\frac{\pt y^k}{\pt t} \widetilde{v}^j
        \right) \widetilde{u}^i J(s)\,dy,
    \end{equation*}
    derived from integration by parts, $\displaystyle \frac{\pt g_{ij}}{\pt y^k} = \sum\limits_{\ell}\bigl(
        \Gamma^\ell_{ik}g_{\ell j} + \Gamma^\ell_{jk} g_{i \ell}
    \bigr)$ and \eqref{eq;converse}.
    On the other hand, by \eqref{eq;dsJ} and \eqref{eq;dsg}, we have
    \begin{equation*}
        \frac{d}{ds} \int_{\widetilde{\Omega}} \sum_{i,j} g_{ij} \widetilde{u}^i \widetilde{v}^j J(s)\,dy
        -\left\langle \frac{\pt \widetilde{u}}{\pt s}, \widetilde{v}\right\rangle_s
        -\left\langle \widetilde{u}, \frac{\pt \widetilde{v}}{\pt s} \right\rangle_s
        =\text{R.H.S. of \eqref{eq;Muv}.} 
    \end{equation*}
    This yields \eqref{eq;dudsv}.
\end{proof}
\subsection{Equivalence between function spaces on $\Omega(t)$ and $\widetilde{\Omega}$}
Let $T>0$ and recall $ Q_T = \bigcup\limits_{0<t<T} \Omega (t) \times \{t\}$ 
and $\widetilde{Q}_T = \widetilde{\Omega}\times (0,T)$.
We investigate the equivalence of $H^m$-norms on $\Omega(t)$ and $\widetilde{\Omega}$
with a uniform constant with respect to $t\in[0,T]$.
\begin{proposition}\label{prop;funcsp}
    Let $T>0$ be fixed. Then there exists a constant $C=C(T)>0$ such that for $m=0,1,2$, 
    for every $t \in [0,T]$ and for a function $\widetilde{f} \in H^m(\widetilde{\Omega})$
    we have 
    \begin{equation}\label{eq;equivH^2}
        \| f \|_{H^m(\Omega(t))} \leq C \| \widetilde{f} \|_{H^m(\widetilde{\Omega})}
        \qquad \text{where } f(x) = \widetilde{f}\bigl(\phi(x,t)\bigr),
    \end{equation}
    and such that, conversely, for $t \in [0,T]$ and for $f \in H^m\bigl(\Omega(t)\bigr)$ we have
    \begin{equation}\label{eq;convequivH^2}
        \|\widetilde{f}\|_{H^m(\widetilde{\Omega})} \leq C \|f\|_{H^m(\Omega(t))}
        \qquad \text{where } \widetilde{f}(y) = {f}\bigl(\phi^{-1}(y,t)\bigr).
    \end{equation}
\end{proposition}
\begin{remark}
    \begin{itemize}
    \item[(i)] We can easily extend \eqref{eq;equivH^2} and \eqref{eq;convequivH^2} for any $m \in \N$ 
    by an inductive argument.  
    Furthermore, we extend them for $H^s\bigl(\Omega(t)\bigr)$ and $H^s(\widetilde{\Omega})$
    for $s >0$ by real interpolation.
    \item[(ii)] \eqref{eq;equivH^2} and \eqref{eq;convequivH^2} hold true for vector fields.
    Indeed, it suffices to consider the estimates for an each component 
    $\displaystyle u^i(x)=\sum_{k} \frac{\pt y^i}{\pt x^k} \widetilde{u}^k$
    or
    $\displaystyle \widetilde{u}^i(y)=\sum_{\ell} \frac{\pt x^i}{\pt y^\ell} u^{\ell}$
    for $i=1,2,3$.
    \item[(iii)] We can replace the interval $[0,T]$ by an arbitrary interval $[T_0,T_1]$ modifying 
    the constants $C=C(T_0,T_1)>0$ in the estimates. 
    \end{itemize}
\end{remark}
\begin{proof}[{Proof of Propostion \ref{prop;funcsp}}]
    We prove only \eqref{eq;equivH^2} since \eqref{eq;convequivH^2} is just an analogy.
    
    Firstly, we note that $
        \| f\|_{L^2(\Omega(t))}^2
        \leq 
        \sup\limits_{0\leq t \leq T} J(t) \,\|\widetilde{f}\|_{L^2(\widetilde{\Omega})}^2 $.
    Furthermore, we see that
    \begin{align*}
        \frac{\pt f}{\pt x^k} (x)
        &=\sum_{m} \frac{\pt \phi^m}{\pt x^k}(x,t)
         \frac{\pt \widetilde{f}}{\pt y^m}\bigl(\phi(x,t)\bigr),
        \\
        \frac{\pt^2 f}{\pt x^k \pt x^\ell}(x)
        &=
        \sum_{m} \frac{\pt^2 \phi^m}{\pt x^k\pt x^\ell}(x,t)
        \frac{\pt \widetilde{f}}{\pt y^m}\bigl(\phi(x,t)\bigr)
        +
        \sum_{m,n}
        \frac{\pt \phi^m}{\pt x^k}(x,t)
        \frac{\pt \phi^n}{\pt x^\ell}(x,t) \frac{\pt^2 \widetilde{f}}{\pt y^m\pt y^n}
        \bigl(\phi(x,t)\bigr),
    \end{align*}
    for $k,\ell=1,2,3$. So we have that by changing variables $x=\phi^{-1}(y,t)$,
    \begin{align*}
        \left\|\frac{\pt f}{\pt x^k}\right\|_{L^2(\Omega(t))}       
        \leq \,& \sum_{m} \sup_{\overline{Q}_T} \left| 
            \frac{\pt \phi^m}{\pt x^k}(x,t) 
        \right| 
        \left\| \frac{\pt \widetilde{f}}{\pt y^m}\right\|_{L^2(\widetilde{\Omega})}
        \Bigl(\sup\limits_{0\leq t \leq T} J(t)\Bigr)^{\frac{1}{2}},
        \\
        \left\|\frac{\pt^2 f}{\pt x^k\pt x^\ell}\right\|_{L^2(\Omega(t))} 
        \leq \, &
        \sum_{m} 
        \sup_{\overline{Q}_T}
        \left| \frac{\pt^2 \phi^m}{\pt x^k\pt x^\ell}(x,t)\right|
        \left\|
            \frac{\pt \widetilde{f}}{\pt y^m}
        \right\|_{L^2(\widetilde{\Omega})}
        \Bigl(\sup\limits_{0\leq t \leq T} J(t) \Bigr)^{\frac{1}{2}}
        \\
        &+
        \sum_{m,n} 
        \sup_{\overline{Q}_T} \left|
            \frac{\pt \phi^m}{\pt x^k}(x,t)
        \frac{\pt \phi^n}{\pt x^\ell}(x,t)
        \right|
        \left\| \frac{\pt^2 \widetilde{f}}{\pt y^m\pt y^n}\right\|_{L^2(\widetilde{\Omega})}
        \Bigl(\sup\limits_{0\leq t \leq T} J(t)\Bigr)^{\frac{1}{2}},
    \end{align*}
for $k,\ell=1,2,3$. These yield \eqref{eq;equivH^2}.
\end{proof}

\subsection{Representation of rotations and Leray's inequality}
In order to represent a notation of rotations, 
it is so effective that we introduce the specific function for indices defined by
\begin{equation}\label{eq;defi sigma}
    \sigma(i)=\begin{cases}
        i & (i \leq 3),
        \\
        i-3 &  (i>3).
    \end{cases}
\end{equation}
Then we note that 
\begin{equation*}
    \rot \,w = \left( 
        \frac{\pt}{\pt x^{\sigma(i+1)}} w^{\sigma(i+2)} 
        -
        \frac{\pt}{\pt x^{\sigma(i+2)}} w^{\sigma(i+1)}
    \right)_{i=1,2,3}.
\end{equation*}
Let us consider to give a representation of the transformation of $\rot\,w$ on $\Omega(t)$ to that on $\widetilde{\Omega}$.
For this purpose,  
we define, for each $s\in\R$, 
\begin{equation}\label{eq;def rot}
    [\Rot(s)\,\widetilde{w}]^i
    :=
    \sum_{j,k,\ell} R^{i,1}_{j,k,\ell}(y,s) \widetilde{w}^\ell
    +
    \sum_{j,k,\ell} R^{i,2}_{j,k,\ell}(y,s) \frac{\pt \widetilde{w}^\ell}{\pt y^k}
    \quad\text{for all }\widetilde{w} \in H^1(\widetilde{\Omega}),
\end{equation}
where we put
\begin{align}
    \label{eq;R1}
    R^{i,1}_{j,k,\ell}(y,s)
    &:=
    \frac{\pt y^i}{\pt x^j}
    \left(
        \frac{\pt y^k}{\pt x^{\sigma(j+1)}} 
        \frac{\pt^2 x^{\sigma(j+2)}}{\pt y^k \pt y^\ell}
        -
        \frac{\pt y^k}{\pt x^{\sigma(j+2)}} 
        \frac{\pt^2 x^{\sigma(j+1)}}{\pt y^k \pt y^\ell}
    \right),
    \\
    \label{eq;R2}
    R_{j, k,\ell}^{i,2}(y,s)
    &:=
    \frac{\pt y^i}{\pt x^j}
    \left(
        \frac{\pt y^k}{\pt x^{\sigma(j+1)}} 
        \frac{\pt x^{\sigma(j+2)}}{  \pt y^\ell}
        -
        \frac{\pt y^k}{\pt x^{\sigma(j+2)}} 
        \frac{\pt x^{\sigma(j+1)}}{ \pt y^\ell}
    \right),
\end{align}
for $y\in \widetilde{\Omega}$  with $x=\phi^{-1}(y,s)$, for $i,j,k,\ell=1,2,3$.

Then, by a direct calculation, we can observe the relation
\begin{equation*}
    [\Rot(t) \,\widetilde{w}]^i(y) = \sum_{\ell} \frac{\pt y^i}{\pt x^\ell} [\rot \, w]^\ell(x)
    \quad \text{for } i=1,2,3,
\end{equation*}
with $x=\phi^{-1}(y,t)$ for each $t\in\R$ and for all $w \in H^1\bigl(\Omega(t)\bigr)$
where $\displaystyle \widetilde{w}^i=\sum_{\ell} \frac{\pt y^i}{\pt x^\ell} w^\ell$,
$i=1,2,3$.

Next, we transform $\rot \,(\theta w)$ 
with a cut-off function $\theta \in C^\infty (\overline{\Omega(t)})$ for each $t\in\R$.
So, for every  $\widetilde{w} \in H^1(\widetilde{\Omega})$ and  for every function 
$\widetilde{\theta} \in C^\infty(\overline{\widetilde{\Omega}})$, 
we put for $s\in \R$
\begin{equation*}
    \bigl[\Rot(s)[\widetilde{\theta},\widetilde{w}]\bigr]^i:=
    \sum_{j,k,\ell} R^{i,1}_{j,k,\ell}(y,s)\, \widetilde{\theta}\widetilde{w}^\ell
    +
    \sum_{j,k,\ell} R^{i,2}_{j,k,\ell}(y,s) 
    \left( 
        \frac{\pt \widetilde{\theta}}{\pt y^k} \widetilde{w}^\ell
        +\widetilde{\theta}\frac{\pt \widetilde{w}^\ell}{\pt y^k}
    \right),
\end{equation*}
for $i=1,2,3$.
Then,
we can easily confirm that for each $t\in \R$
\begin{equation*}
    \bigl[\Rot(t)[\widetilde{\theta},\widetilde{w}]\bigr]^i(y) 
    = \sum_\ell \frac{\pt y^i}{\pt x^\ell} [\rot\,(\theta w)]^\ell(x)
    \quad \text{for } i=1,2,3,
\end{equation*}
with $x=\phi^{-1}(y,t)$,
for all $w \in H^1\bigl(\Omega(t)\bigr)$ and $\theta \in C^\infty\bigl(\overline{\Omega(t)}\bigr)$,
where
$
     \widetilde{\theta}(y)= \theta(x)
$
     and 
$\displaystyle 
     \widetilde{w}^i(y)=\sum_\ell \frac{\pt y^i}{\pt x^\ell} w^\ell(x).
$

In order to deal with the convection term in (N-S$^*$), the following lemma plays 
an essential role, which is related to Leray's inequality.
\begin{lemma}\label{lem;cut rot w}
    Let $T>0$ and let $\widetilde{w} \in C \bigl([0,T]; H^2(\widetilde{\Omega})\bigr)$.
    For $\ep>0$ there exists a function $\widetilde{\theta} \in C^\infty(\overline{\widetilde{\Omega}})$
    such that
    \begin{equation}\label{eq;Leray}
        \bigl\langle 
            N[\widetilde{u},\widetilde{u}], \Rot(t)[\widetilde{\theta},\widetilde{w}(t)]
        \bigr\rangle_t 
        <
        \ep \| \nabla_{y} \widetilde{u}\|_{L^2(\widetilde{\Omega})}^2
        \quad
        \text{for }
        \widetilde{u} \in H^1_{0,\sigma}(\widetilde{\Omega}),
    \end{equation}
    for all $t \in [0,T]$.
\end{lemma}
To prove Lemma \ref{lem;cut rot w}, 
we introduce a cut off function $\widetilde{\theta}$ along to Temam \cite{Temam}.
Firstly, we introduce the following proposition.
\begin{proposition}\label{prop;theta}
    Let $d(y):=\mathrm{dist.}(y,\pt \widetilde{\Omega})$ and let $0<d_*<1$ so that
    $d \in C^\infty\bigl(\overline{\widetilde{\Omega}}_{d_*}\bigr)$, 
    where $\widetilde{\Omega}_{d_*}:=\{ y \in \widetilde{\Omega}\,;\, d(y)<d_*\}$. 
    Further let $0< \rho_* < \min\{1, -1/\log d_*\}$.
    For every $\rho \in (0,\rho_*)$, there exists
    $\widetilde{\theta}_\rho \in C^\infty(\overline{\widetilde{\Omega}})$ such that
    \begin{eqnarray*}
        \widetilde{\theta}_\rho (y) =\begin{cases}
            1, & \text{for } d(y) < \frac{1}{2} e^{-2/\rho}, 
            \\
            0, & \text{for } d(y) > 2 e^{-1/\rho},
        \end{cases}
    \end{eqnarray*}
    with 
    \begin{equation*}
        |\nabla_{y} \widetilde{\theta}(y)| \leq \frac{2\sqrt{3}\rho}{d(y)}.
    \end{equation*}
\end{proposition}
We give a proof of this proposition just for reader's convenience.
\begin{proof}
    Let $\xi(z,\rho)$ be as 
    \begin{equation*}
        \xi(z,\rho):= \begin{cases}
            1, & \text{for } z < e^{-2/\rho},
            \\
            \rho \log \left( \displaystyle
                \frac{e^{-1/\rho}}{z}
            \right),
            & \text{for } e^{-2/\rho} \leq z  < e^{-1/\rho},
            \\
            0,
            & \text{for } z > e^{-1/\rho}.
        \end{cases}
    \end{equation*}
    Since $\xi$ is Lipschitz continuous on $[-1,2d_*]\times [\delta, \rho_*]$ 
    for every sufficiently small $\delta>0$, 
    we note that $\xi$ has a weak derivative with respect to $z$ (and also $\rho$).

    Next we consider the regularization of $\xi$ by the mollifier 
    $\chi_\lambda(z)=\frac{1}{\lambda}\chi\left(\frac{z}{\lambda}\right)$, $\lambda>0$ such as
    \begin{equation*}
        \chi (z) =\begin{cases}
            C\exp\left( -\frac{1}{1-|z|^2}\right),  & \text{for } |z|\leq 1,
            \\
            0, & \text{for } |z| \geq 1,
        \end{cases}
    \end{equation*}
    where $C=\int_{|z|<1} \exp\bigl( -{1}/({1-|z|^2})\bigr)\,dz$.
    Let $2\delta<\rho$ and $\lambda < \frac{1}{4}e^{-2/\delta}(<\frac{1}{4} e^{-2/\rho})$.
    Put 
    \begin{equation*}
        \Theta(z,\rho) := \int_{\R} \chi_{\lambda}(z^\prime) \xi(z-z^\prime,\rho)\,dz^\prime
        \quad \text{for }z\geq 0.
    \end{equation*}
    Then in the case (i) $z<\frac{1}{2}e^{-2/\rho}$, since $z-z^\prime< \frac{1}{2}e^{-2/\rho}+\lambda<e^{-2/\rho}$
    for $|z^\prime|\leq \lambda$,
    we see that 
    \begin{equation*}
        \xi(z-z^\prime,\rho) \equiv 1
        \quad \text{and}\quad
        \pt_z \xi(z-z^\prime,\rho) \equiv 0.
    \end{equation*}
    Hence, $\Theta(z,\rho)=1$ and $\pt_z\Theta(z,\rho) = 0$ in the case (i).
    We consider the case (ii) $\frac{1}{2}e^{-2/\rho} \leq z < 2e^{-1/\rho}$.
    Since noting that $z-z^\prime> z -\lambda> z - \frac{1}{4}e^{-2/\rho}>\frac{1}{2}z$
     for $|z^\prime| \leq \lambda$, we see that
    \begin{equation*}
        |\pt_z \xi(z-z^\prime,\rho)| \leq \frac{\rho}{z-z^\prime} \leq \frac{2\rho}{z}. 
    \end{equation*}
    Hence, we have $\displaystyle | \pt_z \Theta(z,\rho) | \leq \frac{2\rho}{z}$ for the case (ii).
    For the case (iii) $z>2e^{-1/\rho}$ we easily see that $\Theta(z,\rho)\equiv 0$ and 
    $\pt_z \Theta(z,\rho)\equiv 0$.

    Finally we put 
    \begin{equation*}
        \widetilde{\theta}_\rho(y):= \Theta\bigl( d(y),\rho\bigr)
        \quad \text{for } y \in \overline{\widetilde{\Omega}}.
    \end{equation*}
    Then, noting that $|d(y) - d(y^\prime)| \leq |y - y^\prime|$, i.e., $|\nabla_{y} d(y)| \leq \sqrt{3}$,
     $\widetilde{\theta}$ is the desired function.
\end{proof}
Next, we shall prove Lemma \ref{lem;cut rot w} with Proposition \ref{prop;theta}.
\begin{proof}[Proof of Lemma \ref{lem;cut rot w}]
    Let $\theta_\rho$ be in Proposition \ref{prop;theta} with a suitable $\rho>0$ determined later. 
    Referring to \eqref{eq;inner product L2}, we recall that for $t_0 \in [0,T]$, 
    \begin{equation*}
        \begin{split}
            \bigl\langle N[\widetilde{u},\widetilde{u}],
            \Rot(t_0)&[\widetilde{\theta}_\rho,\widetilde{w}(t_0)]
            \bigr\rangle_{t_0}
            \\
            =\,&
            \int_{\widetilde{\Omega}}
            \sum_{i,j}
            g_{ij}(y,t_0)
            \sum_{n} \widetilde{u}^n \nabla_n \widetilde{u}^i 
            \sum_{k,\ell,n}
            \biggl\{
                R^{j,1}_{k,\ell,m} (y,t_0) \widetilde{\theta}_\rho
                \widetilde{w}^{\ell}(t_0)
                \\
                &+
                R^{j,2}_{k,\ell,m} (y,t_0) 
                \Bigl(
                    \frac{\pt\widetilde{\theta}_\rho}{\pt y^m} \widetilde{w}^\ell(t_0)
                    +
                    \widetilde{\theta}_\rho \frac{\pt \widetilde{w}^\ell}{\pt y^m}(t_0)
                \Bigr)
            \biggr\}
            J(t_0)\,dy
            \\
            =:\,&
            \int_{\widetilde{\Omega}}
            \sum_{i,j,k,\ell,m,n}
            \nabla_n \widetilde{u}^i\, \widetilde{u}^n
            W^{i,j}_{k,\ell,m}(y,t_0) J(t_0)\,dy,
        \end{split}
    \end{equation*}
    where we put
    \begin{align*}
        W^{i,j}_{k,\ell,m}(y,t)
        := \,&
        g_{ij}(y,t)
        \biggl\{
                R^{j,1}_{k,\ell,m} (y,t) \widetilde{\theta}_\rho(y)
                \widetilde{w}^{\ell}(y,t)
                \\
                &+
                R^{j,2}_{k,\ell,m} (y,t) 
                \Bigl(
                    \frac{\pt\widetilde{\theta}_\rho}{\pt y^m} (y)\widetilde{w}^\ell(y,t)
                    +
                    \widetilde{\theta}_\rho(y) \frac{\pt \widetilde{w}^\ell}{\pt y^m}(y,t)
                \Bigr)
            \biggr\}.
    \end{align*}
    Then, we have that
    \begin{equation*}
        \bigl\langle N[\widetilde{u},\widetilde{u}],
        \Rot(t_0)[\widetilde{\theta}_\rho,\widetilde{w}(t_0)]
        \bigr\rangle_{t_0}
        \leq 
        \sum_{i,j,k,\ell,m,n} \| \nabla_{n} \widetilde{u}^i\|_{L^2(\widetilde{\Omega})}
        \| \widetilde{u}^n W^{i,j}_{k,\ell,n}(t_0)\|_{L^2(\widetilde{\Omega})} 
        \sup_{0\leq t \leq T} J(t).
    \end{equation*}
    Firstly, we shall estimate $\| \widetilde{u}^n W^{i,j}_{k,\ell,n}(t_0)\|_{L^2(\widetilde{\Omega})} $.
    So, putting
    \begin{equation*}
        U:=
        \max_{i,j,k,\ell,m}
        \biggl(
            \sup_{\widetilde{\Omega}\times [0,T]}
            \bigl| g_{ij}(y,t) R^{j,1}_{k,\ell,m}(y,t)\bigr|
            +
            \sup_{\widetilde{\Omega}\times [0,T]}
            \bigl| g_{ij}(y,t) R^{j,2}_{k,\ell,m}(y,t)\bigr|
        \biggr),
    \end{equation*}
    we note that
    \begin{equation*}
        \begin{split}
            |W^{i,j}_{k,\ell,m}(y,t_0)|
            &\leq U
            \left(
                |\widetilde{\theta}_\rho(y)| |\widetilde{w}(y,t_0)|
                +
                |\nabla_y \widetilde{\theta}_\rho(y) |
                |\widetilde{w}(y,t_0)|
                +
                |\widetilde{\theta}_\rho(y)|
                |\nabla_y \widetilde{w}(y,t_0)|
            \right)
            \\
            & \leq
            U \left(
                |\widetilde{w}(y,t_0)| \mathbbm{1}_{\{d(y)<2e^{-1/\rho}\}}
                +
                \frac{2\sqrt{3}\rho}{d(y)}|\widetilde{w}(y,t_0)|
                +
                |\nabla_y \widetilde{w}(y,t_0)|\mathbbm{1}_{\{d(y)<2e^{-1/\rho}\}}
            \right),
        \end{split}
    \end{equation*}
    where $\mathbbm{1}_A$ is the usual characteristic function of the set $A$,
    i.e., $\mathbbm{1}_A(x)=1$ if $x \in A$ otherwise $\mathbbm{1}_A(x)=0$.
    Hence, by the Sobolev inequality and the Hardy inequality,
    we have that
    \begin{equation}\label{eq;W}
        \begin{split}
            \| \widetilde{u}^n W^{i,j}_{k,\ell,m}(t_0) \|_{L^2(\widetilde{\Omega})}
            \leq \,&C \biggl(
            \| \widetilde{u} \|_{L^6(\widetilde{\Omega})}
            \bigl(
            \|\widetilde{w}(t_0)\|_{L^3(\{d(y)<2e^{-1/\rho}\})}
            +
            \| \nabla_y \widetilde{w}(t_0)\|_{L^3(\{d(y)<2e^{-1/\rho}\})}
            \bigr)
            \\
            &\quad +
            \rho \Bigl\|\frac{\widetilde{u}}{d}\Bigr\|_{L^2(\widetilde{\Omega})}
            \| \widetilde{w}(t_0)\|_{L^\infty(\widetilde{\Omega})}
            \biggr)
            \\ 
            \leq\,& C
            \Bigl( 
                \|\widetilde{w}(t_0)\|_{L^3(\{d(y)<2e^{-1/\rho}\})}
            +
            \| \nabla_y \widetilde{w}(t_0)\|_{L^3(\{d(y)<2e^{-1/\rho}\})}
            \\&\quad +
            \rho \|\widetilde{w}(t_0)\|_{L^\infty(\widetilde{\Omega})}
            \Bigr)\,\|\nabla_y \widetilde{u}\|_{L^2(\widetilde{\Omega})},
        \end{split}
    \end{equation}
    where the constant $C>0$ depends on $\widetilde{\Omega}$ 
    but is independent of $t_0$, $\rho$ and 
    $\widetilde{u} \in H^1_{0,\sigma}(\widetilde{\Omega})$.

    Secondly, we shall estimate $\| \nabla_n \widetilde{u}^i \|_{L^2(\widetilde{\Omega})}$.
    We see that by the Poincar\'{e} inequality
    \begin{equation}\label{eq;Hardy}
        \begin{split}
            \| \nabla_n \widetilde{u}^i \|_{L^2(\widetilde{\Omega})}
            &=
            \Bigl\| 
                \frac{\pt \widetilde{u}^i}{\pt y^n} 
                +
                \sum_{\ell} \Gamma^{i}_{n\ell}(\cdot,t_0) \widetilde{u}^\ell
            \Bigr\|_{L^2(\widetilde{\Omega})}
            \\
            &\leq
            C \bigl(\|\nabla_y \widetilde{u}\|_{L^2(\widetilde{\Omega})}
            + 
            \sup_{\widetilde{\Omega}\times[0,T]} |\Gamma^{i}_{n\ell}(y,t)|\cdot
            \| \widetilde{u} \|_{L^2(\widetilde{\Omega})}
            \bigr)
            \\
            &\leq C \|\nabla_y \widetilde{u}\|_{L^2(\widetilde{\Omega})},
        \end{split}
    \end{equation}
    for $i=1,2,3$, 
    where the constant $C>0$ depends on $\widetilde{\Omega}$ but is independent of
    $t_0$ and $\widetilde{u} \in H^1_{0,\sigma}(\widetilde{\Omega})$.
    
    Combining \eqref{eq;W} and \eqref{eq;Hardy},
    for every fixed $t_0 \in [0,T]$, for $\ep>0$,
    taking sufficiently small $\rho=\rho(\ep,t_0)>0$, we have
    \begin{equation}\label{eq;Leray t0}
        \bigl\langle 
            N[\widetilde{u},\widetilde{u}], 
            \Rot(t_0)[\widetilde{\theta}_\rho,\widetilde{w}(t_0)]
        \bigr\rangle_{t_0}
        \leq {\ep} \| \nabla_y \widetilde{u} \|_{L^2(\widetilde{\Omega})}^2
        \quad \text{for } \widetilde{u} \in H^1_{0,\sigma}(\widetilde{\Omega}).
    \end{equation}

    Next, we extend \eqref{eq;Leray t0} for all $t \in [0,T]$ with a suitable $\rho=\rho(\ep,T)>0$.
    Here, we emphasize and recall that in order to obtain \eqref{eq;Leray t0} at $t_0$,
    it suffices to assume
    \begin{equation}\label{eq;key Leray}
        \|\widetilde{w}(t_0)\|_{L^3(\{d(y)<2e^{-1/\rho}\})}
            +
            \| \nabla_y \widetilde{w}(t_0)\|_{L^3(\{d(y)<2e^{-1/\rho}\})}
            +
            \rho \|\widetilde{w}(t_0)\|_{L^\infty(\widetilde{\Omega})}
            < \frac{\ep}{C_*},
    \end{equation}
    with some constant $C_*>0$ is independent of $t_0$.
    Since $\widetilde{w}$ is uniformly continuous on $[0,T]$ in $H^2(\widetilde{\Omega})$,
    with the aid of the Sobolev inequality $H^2(\widetilde{\Omega}) \hookrightarrow L^3(\widetilde{\Omega})$,
    $H^2(\widetilde{\Omega})\hookrightarrow L^\infty(\widetilde{\Omega})$
    we can choose some $\delta>0$ such that if $|t-s|<\delta$ then it holds that
    \begin{equation}\label{eq;Leray conti}
        \|\widetilde{w}(t)-\widetilde{w}(s)\|_{L^3(\widetilde{\Omega})}
        + \| \nabla_y \widetilde{w}(t)
        -\nabla_y \widetilde{w}(s) \|_{L^2(\widetilde{\Omega})}
        +
        \| \widetilde{w}(t)-\widetilde{w}(s) \|_{L^\infty(\widetilde{\Omega})}
        < \frac{\ep}{C_*}.
    \end{equation}
    Under this situation, we take a finite sequence $\{t_m\}_{m=0}^{m_0}$ with
    $t_m = m \delta$,
    where $0=t_0<t_1<\dots<t_{m_0}\leq T <(m_0+1)\delta$. 
    Then we can choose suitable $\rho=\rho(t_m)$ and $\rho=\rho(T)$ 
    so that \eqref{eq;key Leray}, hence, \eqref{eq;Leray t0} holds for each
    $t_0,\dots,t_{m_0},T$. Here, we put
    \begin{equation*}
        \rho_0 = \min\{\rho(t_0),\rho(t_1),\dots,\rho(t_{m_0}),\rho(T)\}\,(<1).
    \end{equation*}
    Then we observe that a cut off 
    function $\widetilde{\theta}_{\rho_0}$ yields 
    \eqref{eq;Leray} for every $t\in[0,T]$. 
    Indeed, for every $t\in [0,T]$, we have $t_m\leq t<t_{m+1}$ (or $t_{m_0}\leq t < T$),
    So we have by \eqref{eq;Leray conti}
    \begin{equation*}
        \begin{split}
            \|&\widetilde{w}(t)\|_{L^3(\{d(y)<2e^{-1/\rho_0}\})}
            +
            \| \nabla_y \widetilde{w}(t)\|_{L^3(\{d(y)<2e^{-1/\rho_0}\})}
            +
            \rho_0 \|\widetilde{w}(t)\|_{L^\infty(\widetilde{\Omega})}
            \\
            &\leq 
            \|\widetilde{w}(t)-\widetilde{w}(t_m)\|_{L^3(\{d(y)<2e^{-1/\rho_0}\})}
            +\|\widetilde{w}(t_m)\|_{L^3(\{d(y)<2e^{-1/\rho_0}\})}
            \\
            &\quad +
            \| \nabla_y \widetilde{w}(t)-\nabla_y \widetilde{w}(t_m)\|_{L^3(\{d(y)<2e^{-1/\rho_0}\})}
            +\| \nabla_y \widetilde{w}(t_m)\|_{L^3(\{d(y)<2e^{-1/\rho_0}\})}
            \\
            &\quad +
             \rho_0 \|\widetilde{w}(t) - \widetilde{w}(t_m)\|_{L^\infty(\widetilde{\Omega})}
             +
              \rho_0 \|\widetilde{w}(t_m)\|_{L^\infty(\widetilde{\Omega})}
            \\
            &\leq \frac{\ep}{C_*} + \frac{\ep}{C_*}=\frac{2\ep}{C_*}.
        \end{split}
    \end{equation*} 
    Hence we obtain that for $t\in [0,T]$
    \begin{equation}
        \bigl\langle 
            N[\widetilde{u},\widetilde{u}], 
            \Rot(t)[\widetilde{\theta}_{\rho_0},\widetilde{w}(t)]
        \bigr\rangle_{t}
        \leq {2\ep} \| \nabla_y \widetilde{u} \|_{L^2(\widetilde{\Omega})}^2
        \quad \text{for } \widetilde{u} \in H^1_{0,\sigma}(\widetilde{\Omega}).
    \end{equation}
    Since $\ep>0$ is arbitrary, the proof is completed.
\end{proof}

\section{Time dependence of the Helmholtz-Weyl decomposition}
Let us consider the Helmholtz-Weyl decomposition for $b(t) \in H^1\bigl(\Omega(t)\bigr)$
satisfying $\Div\,b(t)=0$ in $\Omega(t)$ such as
\begin{equation*}
    h(t) + \rot\, w(t)=b(t)\quad \text{in } \Omega(t),
\end{equation*}
where $h(t) \in V_{\mathrm{har}}\bigl(\Omega(t)\bigr)$ and $w(t) \in 
Z^2_\sigma\bigl(\Omega(t)\bigr)=X_{\mathrm{har}}\bigl(\Omega(t)\bigr)^{\perp}\cap X_\sigma^2\bigl(\Omega(t)\bigr)$ 
for $t\in\R$. Here, see Section \ref{subsec;HW} for the definitions of the above function spaces. 

In this section, our main interests are to reveal the domain dependence of $h(t)$ and $w(t)$.
Precisely, we investigate the time continuity and also the time differentiability 
of $\widetilde{h}(t)$ and $\widetilde{w}(t)$ on $\widetilde{\Omega}$.
We note that there are two aspects to provide a change in $\widetilde{h}(t)$ and $\widetilde{w}(t)$.
One comes from  the domain perturbation. The other is due to the change of the given data $b(t)$
with respect to $t\in\R$.

It is known that $V_{\mathrm{har}}\bigl(\Omega(t)\bigr)$ is spanned by a finite numbers of flows with 
some scalar potentials. Moreover, these potentials are governed by
the Laplace equations on $\Omega(t)$ with the Dirichlet boundary condition.
Since we are able to directly investigate the variation of these solutions
under the domain perturbation,
we can explicitly carry out the analysis of $\widetilde{h}(t)$ with 
the information of the basis of $V_{\mathrm{har}}\bigl(\Omega(t)\bigr)$ and $\widetilde{b}(t)$.

In contrast, the analysis of $\widetilde{w}(t)$ is rather nontrivial. 
The difficulty comes from a lack of strategy directly 
to compare vector potentials at two different times 
due to the domain deformations 
and also due to the time variation of the operator $\Rot(\cdot)$.
Then, to investigate time continuity we focus on an a priori estimate
for $w(t)$ of some strictly elliptic system in the sense of 
Agmon, Douglis and Nirenberg \cite{ADN}.
However, since the lower order term, i.e.,
$L^2$-norm of $\widetilde{w}(t)$ remains in the a priori estimate 
in general,
%
 we investigate the domain dependence of the
constant $C(\Omega)>0$ as in the decomposition
\begin{equation*}
    \| w \|_{H^2(\Omega)} 
    \leq C(\Omega) \| b \|_{H^1(\Omega)},
    \quad\text{where } h+\rot\,w=b \quad\text{for all } b \in H^1(\Omega).
\end{equation*}
Then, 
the uniformly boundedness of the constant $C\bigl(\Omega(t)\bigr)$ on $[0,T]$ plays an essential role 
to assure the continuity of the vector potential $\widetilde{w}(t)$.
For the time differentiability of $\widetilde{w}(t)$, 
we need more specific qualitative features of vector potentials.
Indeed, we investigate the time-invariance of properties of orthogonal compliment of
$X_{\mathrm{har}}(\Omega)$
under the transformation of vector fields, 
based on the argument in Foia\c{s} and Temam \cite{Foias Temam}.

\subsection{Properties of diffeomorphism from $\Omega(t)$ to $\Omega(t_0)$}
\label{subsec;unit normal}
%
To derive continuity of $\widetilde{w}(t)$ at each $t_0 \in [0,T]$,
it is convenient to consider transformed vector fields on $\Omega(t_0)$ from $\Omega(t)$
rather than on $\widetilde{\Omega}$. 

For this purpose, for $t\in \R$
we introduce a $C^\infty$ diffeomorphism $\vphi(\cdot,t)$ from $\overline{\Omega(t)}$ to $\overline{\Omega(t_0)}$, 
using the relation $\Omega(t) \xrightarrow{\phi(\cdot,t)} \widetilde{\Omega}
\xrightarrow[]{\phi^{-1}(\cdot,t_0)} \Omega(t_0)$,
defined by
\begin{equation}\label{eq;diffeo t0}
   \Omega(t_0) \ni \tilde{x} =\vphi(x,t) := \phi^{-1}\bigl(\phi(x,t), t_0\bigr)
    \quad \text{for } x \in \overline{\Omega(t)}.
\end{equation}
We often use the notation $\tilde{x}$ for an coordinate of $\Omega(t_0)$ 
for some fixed $t_0\in \R$.
Then $\vphi^{-1}(\tilde{x},t)=\phi^{-1}\bigl(\phi(\tilde{x},t_0),t\bigr)$. 
Here we note that we inherit the notations $J(t),g_{ij}$, $g^{ij}$, $\Gamma^{i}_{jk}$, 
\dots, $\Rot(t)$ replacing $\phi$ by $\vphi$ (if there is no possibility of confusion from the context).

Moreover, via a coordinate transformation with $\varphi(\cdot,t)$, 
there is a one-to-one relation between vector fields $u$ on $\Omega(t)$ and $\tilde{u}$ on $\Omega(t_0)$
such that
\begin{equation}\label{eq;identify tilde}
    \tilde{u}^i= \sum_{\ell} \frac{\pt \tilde{x}^i}{\pt x^\ell} u^\ell,
    \quad{i.e.,}\quad 
    \tilde{u}^i(\tilde{x})=\sum_{\ell} \frac{\pt \vphi^i}{\pt x^\ell}(\varphi^{-1}(\tilde{x},t),t)
    u^\ell\bigl(\vphi^{-1}(\tilde{x},t)\bigr), \quad i=1,2,3.
\end{equation}
So, we can identify $u$ on $\Omega(t)$ with $\tilde{u}$ on $\Omega(t_0)$ by the above transformation.
Hence, we may often write just $\tilde{u}$ when $u$ on $\Omega(t)$ and $\vphi(\cdot,t)$ are obvious
from the context.
Furthermore we also often use a notation of a vector (or function) with short-tilde, like $\tilde{u}$, 
in order to express an arbitrary vector field (or function) on $\Omega(t_0)$ as well.

Furthermore, we have the relation between 
the coordinate transformation $\widetilde{u}$ in \eqref{eq;identify} and 
$\tilde{u}$ in \eqref{eq;identify tilde}.
Indeed, we see, by a direct calculation of the transformation: $\tilde{u} \text{ on } \Omega(t_0) \longmapsto \widetilde{\tilde{u}} \text{ on }\widetilde{\Omega}$, 
that
\begin{equation}\label{eq;tildewidetilde}
    \widetilde{\tilde{u}}^i=\sum_{\ell} \frac{\pt y^i}{\pt \tilde{x}^\ell} \tilde{u}^\ell
    =\sum_{\ell,k} \frac{\pt y^i}{\pt \tilde{x}^\ell}\frac{\pt \tilde{x}^\ell}{\pt x^k} u^k 
    =\sum_{k} \frac{\pt y^i}{\pt x^k} u^k=\widetilde{u}^i,
    \quad i=1,2,3.
\end{equation}
Therefore, the analysis on $\tilde{u}(t)$ on $\Omega(t_0)$ is valid for that on $\widetilde{u}(t)$ on $\widetilde{\Omega}$.
\medskip

Let us investigate a time local behavior around $t_0$ of the diffeomorphism and the Riemann metrics.
Firstly, since $\vphi^{-1}=(\vphi^{-1}_1, \vphi^{-1}_2, \vphi^{-1}_3)$ is a $C^\infty$ function on 
$\overline{\Omega(t_0)\times \R}$ and since $\vphi^{-1}(\tilde{x},t_0)=\tilde{x}$
and since $\displaystyle \frac{\pt \vphi_i^{-1}}{\pt \tilde{x}^k}(\tilde{x},t_0)=\delta_{i,k}$,
$\displaystyle \frac{\pt^2 \vphi^{-1}_i}{\pt \tilde{x}^k\tilde{x}^\ell} (\tilde{x},t_0)\equiv 0$,
$\displaystyle  \frac{\pt^3 \vphi^{-1}_i}{\pt \tilde{x}^k\pt \tilde{x}^\ell \pt \tilde{x}^m} (\tilde{x},t_0)\equiv 0$
for $i,k,\ell,m=1,2,3$, 
we have 
\begin{alignat}{2}
    \label{eq;eq vphiId}
    \vphi^{-1}(\tilde{x}, t) 
    &\to \tilde{x} \quad&\text{uniformly in $\overline{\Omega(t_0)}$ as } t \to t_0,
    \\
    \label{eq;eq vphidelta}
    \frac{\pt \vphi_i^{-1}}{\pt x^k} (\tilde{x},t)
    &\to \delta_{i,k}
    \quad&\text{uniformly in $\overline{\Omega(t_0)}$ as } t \to t_0,
    \\
    \label{eq;convphi2}
    \frac{\pt^2 \vphi^{-1}_i}{\pt \tilde{x}^k \pt\tilde{x}^\ell} (\tilde{x},t),
    \,
    \frac{\pt^3 \vphi^{-1}_i}{\pt \tilde{x}^k\pt \tilde{x}^\ell \pt \tilde{x}^m} (\tilde{x},t)
    &\to 0
    \quad&\text{uniformly in $\overline{\Omega(t_0)}$ as } t \to t_0
\end{alignat}
for $i,k,\ell,m=1,2,3$.
Similarly, since $\vphi$ is a $C^\infty$ function on $\overline{Q}_\infty$ 
and since $\vphi(x,t_0)=x$, 
$\displaystyle\frac{\pt \vphi^i}{\pt x^k}(x,t_0)=\delta_{i,k}$, 
$\displaystyle\frac{\pt^2 \vphi^i}{\pt x^k\pt x^\ell}(x,t_0)\equiv 0$
and $\displaystyle\frac{\pt^3 \vphi^i}{\pt x^k \pt x^\ell \pt x^m}(x,t_0)\equiv 0$,
for $i,k,\ell,m=1,2,3$, we see that
\begin{gather}\label{eq;vphitoId}
    \frac{\pt \vphi^i}{\pt x^k}\bigl(\vphi^{-1}(\tilde{x},t),t\bigr)
    \to \delta_{i,k}
    \quad\text{uniformly in $\overline{\Omega(t_0)}$ as } t \to t_0,
    \\
    \label{eq;ptvphito0}
    \frac{\pt^2 \vphi^i}{\pt x^k\pt x^\ell}\bigl(\vphi^{-1}(\tilde{x},t),t\bigr),\,\,
    \frac{\pt^3 \vphi^i}{\pt x^k \pt x^\ell \pt x^m}\bigl(\vphi^{-1}(\tilde{x},t),t\bigr)
    \to 0
    \quad\text{uniformly in $\overline{\Omega(t_0)}$ as } t \to t_0,
\end{gather}
for $i,k,\ell,m=1,2,3$.

Now, we recall 
\begin{equation*}
    g^{ij}(\tilde{x},t) =
    \sum_{k} 
    \frac{\pt \vphi^i}{\pt x^k}\bigl(\vphi^{-1}(\tilde{x},t),t\bigr)
    \frac{\pt \vphi^j}{\pt x^k}\bigl(\vphi^{-1}(\tilde{x},t),t\bigr).
\end{equation*}
Due to \eqref{eq;convphi2}, \eqref{eq;vphitoId} and \eqref{eq;ptvphito0}, 
by direct calculation, we immediately obtain the following.
\begin{proposition} \label{prop;gij}
    it holds that for $i,j,\ell,m=1,2,3$ 
    \begin{gather*}
        g^{ij} (\tilde{x},t),\; g_{ij}(\tilde{x},t) \to \delta_{i,j} 
        \quad\text{uniformly in $\overline{\Omega(t_0)}$ as } t \to t_0,
        \\
        \frac{\pt g^{ij}}{\pt \tilde{x}^\ell}(\tilde{x},t),\,\,
        \frac{\pt^2 g^{ij}}{\pt \tilde{x}^\ell\pt \tilde{x}^m}(\tilde{x},t)
        \to 0
        \quad\text{uniformly in $\overline{\Omega(t_0)}$ as } t \to t_0,
        \\
        \Gamma^i_{j\ell} (\tilde{x},t),\,\,
        \frac{\pt \Gamma^i_{j\ell}}{\pt \tilde{x}^m}(\tilde{x},t) 
        \to 0
        \quad \text{uniformly in $\overline{\Omega(t_0)}$ as } t \to t_0.
    \end{gather*}
\end{proposition}
%
\subsection{Time dependence of the harmonic vector fields}
\label{sec;Vharmonic}
%
To consider time dependence of $\widetilde{h}(t)$, 
we shall investigate the time continuity and differentiability of the basis
of $V_{\mathrm{har}}\bigl(\Omega(t)\bigr)$.
A basis of $V_{\mathrm{har}}\bigl( \Omega(t)\bigr)$ is 
given by $\nabla q_k(t)$, $k=1,\dots,K$ whose scalar potential satisfies
the following Laplace equation with the Dirichlet boundary condition:
\begin{equation}\label{eq;Lap q_k}
    \begin{cases}
        \Delta q_k(t) =0  &\text{in } \Omega(t), 
        \\
        q_k(t)|_{\Gamma_\ell(t)}=\delta_{k,\ell} & \text{for }\,\ell=0,\dots,K 
\end{cases}
\end{equation}
for $k=1,\dots,K$. See, for instance, Kozono and Yanagisawa \cite{Kozono Yanagisawa IUMJ,Kozono Yanagisawa MathZ}.

To investigate the time continuity and differentiability of $\nabla q_k(t)$ at $t_0$,
we introduce the transformed function $\tilde{q}_k(t)$ on $\Omega(t_0)$ defined by
\begin{equation}\label{eq;defi tild q}
    \tilde{q}_k(\tilde{x},t)=q_k\bigl(\vphi^{-1}(\tilde{x},t),t\bigr)
    \quad \text{for } \tilde{x} \in \Omega(t_0).
\end{equation} 
Then we 
observe that \eqref{eq;Lap q_k} is equivalent to
\begin{equation}\label{eq;Ltq_k}
    \begin{cases}
        \mathcal{L}(t) \tilde{q}_k (t)= 0 & \text{in } \Omega(t_0),
        \\
        \tilde{q}_k(t)|_{\Gamma_\ell(t_0)}=\delta_{k,\ell} & \text{for } \ell=0,\dots,K
\end{cases}
\end{equation}
for $k=1,\dots,K$, where $\mathcal{L}(t)$ is defined by 
\begin{equation*}
    [\mathcal{L}(t) \tilde{q}] (\tilde{x})=
    \sum_{k,\ell} \frac{\pt }{\pt \tilde{x}^\ell} \left(
        g^{k\ell}(\tilde{x},t) \frac{\pt \tilde{q}}{\pt \tilde{x}^k}(\tilde{x})
    \right). 
\end{equation*}
%

For a while, we focus on the domain dependence of the solutions of \eqref{eq;Ltq_k}.
%
%
\subsubsection{Local uniform estimate of the solutions \eqref{eq;Ltq_k} for $t$}
Here, we prepare the following a priori estimates with uniform constant with repect to $t$.
\begin{proposition}\label{prop;uniform bound q_k}
    For each $t_0\in [0,T]$, there exist a constant $C=C(T,t_0)>0$  and $\delta>0$ 
    such that for all $t \in \R$ if $|t-t_0| < \delta$ then 
    \begin{equation}\label{eq;aprioriq}
        \| \tilde{v} \|_{H^2(\Omega(t_0))} \leq C \left(
            \| \mathcal{L}(t) \tilde{v} \|_{L^2(\Omega(t_0))} + 
            \| \tilde{v} \|_{H^{\frac{3}{2}}(\pt \Omega(t_0))}
        \right)
    \end{equation}
    for all $\tilde{v} \in H^2\bigl(\Omega(t_0)\bigr)$.
\end{proposition}
\begin{proof}
    Firstly, it is well known that for each $t\in \R$ 
    we can take $C=C(t)>0$ such that \eqref{eq;aprioriq} holds true.  
    Here, we prove the proposition by a contradiction argument. 
    Then for $m \in \N$ there exist $t_m \in \R$ and $\tilde{v}_m \in H^2\bigl(\Omega(t_0)\bigr)$ 
    such that
    $|t_m-t_0| < \frac{1}{m}$ and such that $\| \tilde{v}_m\|_{H^2(\Omega(t_0))} =1$ and
    \begin{equation*}
        \frac{1}{m} > \| \mathcal{L}(t_m) \tilde{v}_m \|_{L^2(\Omega(t_0))} 
        + \| \tilde{v}_m \|_{H^{\frac{3}{2}}(\pt \Omega(t_0))}.
    \end{equation*}
    Here, noting $\mathcal{L}(t_0)=\Delta_{\tilde{x}}$, we see that 
\begin{equation*}
    \begin{split}
        (\mathcal{L}(t_m) -\Delta_{\tilde{x}}) \tilde{v}_m
        &=\sum_{k,\ell} \frac{\pt }{\pt \tilde{x}^\ell} \left(
            \left(g^{k\ell}(\tilde{x},t_m) -\delta_{k,\ell}\right)
            \frac{\pt \tilde{v}_m}{\pt \tilde{x}^k}(\tilde{x}) 
            \right)
            \\
        &=
        \sum_{k,\ell} \left(
            \frac{\pt g^{k\ell}}{\pt \tilde{x}^\ell} 
            (\tilde{x},t_m)
            \frac{\pt \tilde{v}_m}{\pt \tilde{x}^k}(\tilde{x})
            +
            \left(g^{k\ell}(\tilde{x},t_m) -\delta_{k,\ell}\right)
            \frac{\pt^2 \tilde{v}_m}{\pt \tilde{x}^k \pt \tilde{x}^\ell}
            (\tilde{x})
        \right).
    \end{split}
\end{equation*}
So, by Proposition \ref{prop;gij}, we obtain that
\begin{equation}\label{eq;Lt-Delta}
    \begin{split}
    \bigl\|
        \bigl(\mathcal{L}(t_m) -\Delta_{\tilde{x}}\bigr) \tilde{v}_m
    \bigr\|_{L^2(\Omega(t_0))}  
    &\leq
    \sum_{k,\ell} \left(
        \sup_{\tilde{x}} \left|
            \frac{\pt g^{k\ell}}{\pt \tilde{x}^\ell} (\tilde{x},t_m)
        \right|
        +
        \sup_{\tilde{x}} 
    \left| 
        g^{k\ell}(\tilde{x},t_m) -\delta_{k,\ell} 
    \right|
    \right)\,\|\tilde{v}_m\|_{H^2(\Omega(t_0))}
    \\
    &\to 0 \quad \text{ as } t_m\to t_0.
    \end{split}
\end{equation}
    Then, by the a priori estimate of the Laplacian with the Dirichlet boundary condition 
    with some constant $C=C(t_0,\Delta_{\tilde{x}})>0$ 
    we have by \eqref{eq;Lt-Delta}
    \begin{equation*}
    \begin{split}
        1= \|\tilde{v}_m \|_{H^2(\Omega(t_0))} 
        &\leq C \left(
            \| \Delta_{\tilde{x}} \tilde{v}_m \|_{L^2(\Omega(t_0))} 
            + \| \tilde{v}_m\|_{H^{\frac{3}{2}}(\pt \Omega(t_0))}
        \right) 
        \\
        &\leq C \left(
            \bigl\|\bigl(\Delta_{\tilde{x}} - \mathcal{L}(t_m)\bigr)\tilde{v}_m \bigr\|_{L^2(\Omega(t_0))}
            + 
            \| \mathcal{L}(t_m)\tilde{v}_m \|_{L^2(\Omega(t_0))} 
            + \| \tilde{v}_m\|_{H^{\frac{3}{2}}(\pt \Omega(t_0))}
        \right)
        \\
        &\leq
        C \sum_{k,\ell} \left(
            \sup_{\tilde{x}} \left|
                \frac{\pt g^{k\ell}}{\pt \tilde{x}^\ell} (\tilde{x},t_m)
            \right|
            +
            \sup_{\tilde{x}} 
        \left| 
            g^{k\ell}(\tilde{x},t_m) -\delta_{k,\ell} 
        \right|
        \right)\,\|\tilde{v}_m\|_{H^2(\Omega(t_0))} 
        +
        \frac{C}{m} 
        \\
        &\to 0 \quad \text{as } m\to \infty.
    \end{split}
    \end{equation*}
This yields a contradiction. 
\end{proof}
%
\subsubsection{Time continuity of solutions to \eqref{eq;Ltq_k}}
Here, our aim is to consider the time continuity of the transformed function 
$\widetilde{q}_k(t)$ on $\widetilde{\Omega}$ from solutions $q_k(t)$ to \eqref{eq;Lap q_k}.
More precisely, for $k=1,\dots,K$ we put 
\begin{equation}\label{eq;qkkk}
    \widetilde{q}_k(y,t):= q_k \bigl(\phi^{-1}(y,t),t\bigr)
    \quad \text{for } y \in  \widetilde{\Omega}.
\end{equation}
Indeed, we have the following theorem.
\begin{theorem}\label{thm;rem1}
    Let $T>0$.
    It holds that $\widetilde{q}_k \in C\bigl([0,T];H^2(\widetilde{\Omega})\bigr)$
    for $k=1,\dots,K$.
\end{theorem}
In order to prove Theorem \ref{thm;rem1}, it suffices to consider the 
continuity at $t_0$ of the solutions $\tilde{q}_k(t)$ to \eqref{eq;Ltq_k}. 
Firstly, we shall prove the following lemma for $\tilde{q}_k(t)$ for $k=1,\dots,K$, 
noting that $\tilde{q}_k(t_0)=q_k(t_0)$.
\begin{lemma}\label{lem;contiq_k}
    Let $t_0 \in [0,T]$ and let $\tilde{q}_k(t)$ is the solution of \eqref{eq;Ltq_k} for $t\in\R$, 
    $k=1,\dots,K$. 
    For $\ep >0$ there exists $\delta>0$ such that
    for every $t\in \R$ if $|t-t_0|<\delta$ then
    \begin{equation*}
        \| \tilde{q}_k(t) - {q}_k(t_0)\|_{H^2(\Omega(t_0))} < \ep
    \end{equation*}
    for $k=1,\dots,K$.
\end{lemma}
\begin{proof}
    Noting $\mathcal{L}(t)\bigl(\tilde{q}_k(t)-q_k(t_0)\bigr)=
    \bigl(\mathcal{L}(t)-\Delta_{\tilde{x}}\bigr)q_k(t_0)$
    and $\tilde{q}_k(t)-q_k(t_0)|_{\pt \Omega(t_0)}=0$ for $k=1,\dots,K$,
    by the uniform a priori estimate \eqref{eq;aprioriq} for $\mathcal{L}(t)$, 
    if $|t-t_0|<\delta$ with some $\delta>0$,
    then we have 
    \begin{equation*}
        \begin{split}
            \| \tilde{q}_k (t) -q_k(t_0) \|_{H^2(\Omega(t_0))}
            &\leq 
            C \bigl\| \mathcal{L}(t)\bigl(\tilde{q}_k(t)-q_k(t_0)\bigr) \bigr\|_{L^2(\Omega(t_0))}
            \\
            &=
            C \bigl\| 
                \bigl(\mathcal{L}(t)-\Delta_{\tilde{x}}\bigr)q_k(t_0) 
            \bigr\|_{L^2(\Omega(t_0))}
            \\
        &\leq
        C \sum_{i,j} \left(
            \sup_{\tilde{x}} \left|
                \frac{\pt g^{ij}}{\pt \tilde{x}^j} (\tilde{x},t)
            \right|
            +
            \sup_{\tilde{x}} 
        \left| 
            g^{ij}(\tilde{x},t) -\delta_{i,j} 
        \right|
        \right)\,\|{q}_k(t_0)\|_{H^2(\Omega(t_0))} 
        \\
        &\to 0 \quad \text{as } t\to t_0,
        \end{split}
    \end{equation*}
    for all $k=1,\dots,K$. This completes the proof.
\end{proof}
Next, we shall prove Theorem \ref{thm;rem1}.
\begin{proof}[{Proof of Theorem \ref{thm;rem1}}]
    Firstly, we consider the transformation: $\tilde{q}_k(t)\text{ on } \Omega(t_0)
    \longmapsto
    \widetilde{\tilde{q}}_k(t) \text{ on } \widetilde{\Omega}$, defined by 
    $\widetilde{\tilde{q}}_k(y,t):=\tilde{q}_k\bigl(\phi^{-1}(y,t_0),t\bigr)$ for $y\in\widetilde{\Omega}$
     for $k=1,\dots,K$. 
     Since $\tilde{q}_k(t)$ is defined by \eqref{eq;defi tild q} and
      $\vphi^{-1}(\tilde{x},t)=\phi^{-1}\bigl(\phi(\tilde{x},t_0),t\bigr)$ for
     $\tilde{x}\in \Omega(t_0)$, by a direct calculation, we see that
     \begin{equation}\label{eq;equiv tt_0tilde q}
        \widetilde{\tilde{q}}_k (y,t)= \tilde{q}_k\bigl(\phi^{-1}(y,t_0),t\bigr)
        = q_k \bigl(\phi^{-1}(y,t),t\bigr)
        =\widetilde{q}_k(y,t),
     \end{equation}
     where $\widetilde{q}_k(t)$ is in \eqref{eq;qkkk}.
     Hence, by Proposition \ref{prop;funcsp} and Lemma \ref{lem;contiq_k},
     \begin{equation*}
        \| \widetilde{q}_k(t) - \widetilde{q}_k(t_0)\|_{H^2(\widetilde{\Omega})}
        =
        \| \widetilde{\tilde{q}}_k(t) - \widetilde{q}_k(t_0)\|_{H^2(\widetilde{\Omega})}
        \leq C 
        \| \tilde{q}_k(t) - q_k(t_0)\|_{H^2(\Omega(t_0))} \to 0 \quad \text{as } t \to t_0,
     \end{equation*}
     where the constant $C>0$ is independent of $t\in[0,T]$. This completes the proof.
\end{proof}

Noting that $V_{\mathrm{har}}\bigl(\Omega(t)\bigr)$ is spanned by $\nabla q_1(t),\dots,\nabla q_K(t)$,
we give the following corollary.
\begin{corollary}\label{cor;CqC}
    There exists a constant $C_1,\,C_2>0$ such that 
    for all $t \in [0,T]$ 
    \begin{equation*}
        C_1 \leq \| \nabla q_k(t) \|_{H^1(\Omega(t))} \leq C_2, \quad k=1,\dots,K. 
    \end{equation*}
\end{corollary}
\begin{proof}
     Using the transformation of a vector field $\nabla q_k(t)$ on $\Omega(t)$ to that on $\widetilde{\Omega}$, 
     we see that
     \begin{equation*}
        \sum_{\ell} \frac{\pt \phi^i}{\pt x^\ell}\bigl(\phi^{-1}(y,t),t\bigr) 
        \frac{\pt q_k}{\pt x^\ell}\bigl(\phi^{-1}(y,t),t\bigr)
        = \sum_{\ell} g^{i\ell}(y,t) \frac{\pt \widetilde{q}_k}{\pt y^\ell}(y,t)=[\nabla_g \widetilde{q}_k]^i(y,t),
     \end{equation*}
    for all $y \in \widetilde{\Omega}$. 
    Furthermore, we easily see that $\nabla_g \widetilde{q}_k\in C\bigl([0,T];H^1(\widetilde{\Omega})\bigr)$ 
    by Theorem \ref{thm;rem1}.
    Hence, by Proposition \ref{prop;funcsp} we obtain
    
    \begin{equation*}
        C \min_{0\leq t\leq T} \| \nabla_g \widetilde{q}_k (t) \|_{H^1(\widetilde{\Omega})}
        \leq \| \nabla q_k (t) \|_{H^1(\Omega(t))}
        \leq C \max_{0 \leq t \leq T} \| \nabla_g \widetilde{q}_k (t) \|_{H^1(\widetilde{\Omega})}
    \end{equation*}
    for $k=1,\dots,K$. Here, we note that $\|\nabla_g \widetilde{q}_k(t)\|_{H^1(\widetilde{\Omega})} \neq 0$
    for all $t\in [0,T]$.
    Indeed, if not, we see that $\nabla q_k(x,t)\equiv 0$, hence $q_k(x,t) \equiv \text{Const.}$
    in $\overline{\Omega(t)}$ with some $t\in[0,T]$. 
    However this contradicts the boundary condition of \eqref{eq;Lap q_k}.
    This completes the proof.
\end{proof}
\subsubsection{Time differentiability of solutions to \eqref{eq;Ltq_k}}
\label{subsec;diff qk}
The aim of this subsubsection is to show 
$\widetilde{q}_k \in C^1\bigl([0,T];H^2(\widetilde{\Omega})\bigr)$, $k=1,\dots,K$.
For this aim, we focus on $\tilde{q}_k(t)$ on $\Omega(t_0)$ at each fixed $t_0\in [0,T]$ as in the above.

We introduce an operator $\dot{\mathcal{L}}(t_0)$ 
(which is the first variation of $\mathcal{L}(t)$ at $t_0$)
defined by for $\tilde{q} \in H^2\bigl(\Omega(t_0)\bigr)$
\begin{equation*}
    \begin{split}
    \dot{\mathcal{L}}(t_0)\tilde{q} 
    &:= \lim_{\ep \to 0} \frac{\mathcal{L}(t_0+\ep) - \Delta_{\tilde{x}}}{\ep} \tilde{q} 
    \quad \text{in } L^2\bigl(\Omega(t_0)\bigr) 
    \\
    &= \sum_{i,j} \frac{\pt}{\pt \tilde{x}^j} \left(
        \left.\frac{\pt g^{ij}}{\pt t}\right|_{t=t_0} \frac{\pt \tilde{q}}{\pt \tilde{x}^i}
    \right).
    \end{split}
\end{equation*}
Then, we  introduce solutions $\dot{q}_k(t_0)$, $k=1,\dots,K$, of the following Poisson equation with the homogeneous Dirichlet boundary condition:
\begin{equation}\label{eq;Lap val q_k}
    \begin{cases}
        \Delta_{\tilde{x}} \dot{q}_k (t_0) = -\dot{\mathcal{L}}(t_0) q_k(t_0) &\text{in }\Omega(t_0),
        \\
        \dot{q}_k (t_0) =0 &\text{on } \pt \Omega(t_0),
    \end{cases}
\end{equation}
where $q_k(t_0)$ is a solution to \eqref{eq;Lap q_k} at $t_0\in [0,T]$.

Firstly, we discuss the existence of a time derivative $\tilde{q}_k(t)$ at $t_0$.
\begin{lemma}\label{lem;diff q_k}
    Let $t_0 \in [0,T]$ and let $q_k(t_0)$, $k=1,\dots,K$, be solutions to \eqref{eq;Lap q_k} at $t_0$.
    Let $\tilde{q}_k(t)$ is a solution to \eqref{eq;Ltq_k} for $t\in \R$. 
    Then, 
    \begin{equation*}
        \lim_{\ep\to 0} \frac{\tilde{q}_k(t_0+\ep)-q_k(t_0)}{\ep} = \dot{q}_k(t_0)
        \quad \text{in } H^2\bigl(\Omega(t_0)\bigr),
    \end{equation*}
    where $\dot{q}_k(t_0)$ is a solution to \eqref{eq;Lap val q_k}.
\end{lemma}
\begin{proof}
    In order to apply the a priori estimate \eqref{eq;aprioriq} for $\mathcal{L}(t_0+\ep)$, 
    we observe that by \eqref{eq;Ltq_k} and \eqref{eq;Lap val q_k},
    $\bigl(\tilde{q}_k(t)-q_k(t_0)\bigr)/\ep -\dot{q}_k(t)$ satisfies the following equation in $\Omega(t_0)$
    \begin{equation*}
        \begin{split}
            \mathcal{L}(t_0+\ep)&\left( 
                \frac{\tilde{q}_k(t_0+\ep) - q_k(t_0)}{\ep} - \dot{q}_k(t_0)
            \right)
            \\
            &=-\frac{\mathcal{L}(t_0+\ep) - \Delta_{\tilde{x}}}{\ep}q_k(t_0) 
            - \mathcal{L}(t_0+\ep)\dot{q}_k(t_0)
            \\
            &=
            -\frac{\mathcal{L}(t_0+\ep) - \Delta_{\tilde{x}}}{\ep}q_k(t_0) 
            +\dot{\mathcal{L}}(t)q_k(t_0) - \bigl(\mathcal{L}(t_0+\ep)-\Delta_{\tilde{x}}\bigr)\dot{q}_k(t_0).
        \end{split}
    \end{equation*}
    and  the boundary condition
    \begin{equation*}
        \frac{\tilde{q}_k(t_0+\ep) - q_k(t_0)}{\ep} - \dot{q}_k(t_0)=0 \quad \text{on } \pt \Omega(t_0).
    \end{equation*}

    Then, using \eqref{eq;aprioriq} we see that  by the definition of $\dot{\mathcal{L}}(t_0)$ 
    and by the same argument as in \eqref{eq;Lt-Delta},
    \begin{equation*}
        \begin{split}
            \left\| \frac{\tilde{q}_k(t_0+\ep) - q_k(t_0)}{\ep} - \dot{q}_k(t_0)\right\|_{H^2(\Omega(t_0))}
            &
            \leq
            C \left\|\mathcal{L}(t_0+\ep)\left( 
                \frac{\tilde{q}_k(t_0+\ep) - q_k(t_0)}{\ep} - \dot{q}_k(t_0)
            \right) \right\|_{L^2(\Omega(t_0))}
            \\
            &\to 0 \quad \text{as } \ep\to 0.
        \end{split}
    \end{equation*}
    This completes the proof.
\end{proof}
Next, we discuss the time continuity of the derivative $\dot{q}_k(t)$, $k=1,\dots,K$.
To begin with, to deal with $\dot{q}_k(t)$ for $t\in [0,T]$, we have to rewrite $\dot{\mathcal{L}}(t_0)$
replacing $t_0$ by $t$.
However, we remark that the Riemann metric $g^{ij}(\tilde{x},t)$, 
in the definition of $\dot{\mathcal{L}}(t_0)$ above, implicitly contains $t_0$ since it consists 
of the diffeomorphism $\varphi(\cdot,t)$ which depends on  $t_0 \in [0,T]$. 
So, we rewrite $\displaystyle \frac{\pt g^{ij}}{\pt t}$ using the diffeomorphism $\phi$, without $\varphi$.
Indeed, we define $\mathcal{G}^{ij}(x,t)$
on $\overline{Q}_\infty$ by 
\begin{equation*}
    \begin{split}
    \mathcal{G}^{ij}(x,t)
    =\,&
    -\sum_{k} \left(
        \frac{\pt^2 \phi^{-1}_i}{\pt y^k \pt s}\bigl(\phi(x,t),t\bigr) 
        \frac{\pt \phi^k}{\pt x^j}(x,t)
        +
        \frac{\pt^2 \phi^{-1}_j}{\pt y^k \pt s}\bigl(\phi(x,t),t\bigr) 
        \frac{\pt \phi^k}{\pt x^i}(x,t)
    \right)
    \quad \text{for } (x,t)\in \overline{Q}_\infty.
    \end{split}
\end{equation*}
Here, we note that $\mathcal{G}^{ij}\in C^\infty(\overline{Q}_\infty)$.
Noting the identity
\begin{multline*}
    0\equiv
    \sum_{k,\ell} \frac{\pt^2 \phi^{-1}_i}{\pt y^k \pt y^\ell}\bigl(\phi(x,t),t\bigr)
    \frac{\pt \phi^k}{\pt x^j}(x,t)\frac{\pt \phi^\ell}{\pt t}(x,t)
    \\+
    \sum_{k} \frac{\pt^2 \phi^{-1}_i}{\pt y^k \pt s}\bigl( \phi(x,t),t\bigr) 
    \frac{\pt \phi^k}{\pt x^j}(x,t)
    +
    \sum_{k} \frac{\pt \phi^{-1}_i}{\pt y^k} \bigl(\phi(x,t),t\bigr) 
    \frac{\pt^2 \phi^k}{\pt x^j \pt t}(x,t)
\end{multline*}
for all $(x,t) \in \overline{Q}_\infty$ from \eqref{eq;Kronecker},
 by an elementary calculation, we see that
$ \displaystyle \frac{\pt g^{ij}}{\pt t}(\tilde{x},t_0)=\mathcal{G}^{ij}(\tilde{x},t_0)$\footnote{The representation of $\mathcal{G}^{ij}$ is not unique in general, and has an equivalent one.}. 
Therefore, we can write $\dot{\mathcal{L}}(t)$ on $\Omega(t)$ such as 
\begin{equation*}
    \dot{\mathcal{L}}(t)q= \sum_{i,j} \frac{\pt }{\pt x^j} \left(
        \mathcal{G}^{ij}(t) \frac{\pt q}{\pt x^i}
    \right) \quad \text{for } q \in H^2\bigl(\Omega(t)\bigr),\; t\in [0,T].
\end{equation*}

Hence, we observe that $\dot{q}_k(t)$, $k=1,\dots,K$, satisfies the following equation on $\Omega(t)$:
\begin{equation*}
    \begin{cases}
        \Delta \dot{q}_k(t) = -\dot{\mathcal{L}}(t)q_k(t) &\text{in } \Omega(t),\\
        \dot{q}_k(t) = 0 & \text{on } \pt \Omega(t),
    \end{cases}
\end{equation*}
for each $t \in [0,T]$.

In order to investigate the time continuity at $t_0 \in [0,T]$ of 
the transformed function $\widetilde{\dot{q}}_k(t)$ on $\widetilde{\Omega}$, 
it is  effective to introduce the transformation:
$ \dot{q}_k(t) \text{ on } \Omega(t) \longmapsto \tilde{\dot{q}}_k(t) \text{ on } \Omega(t_0)$.
Here, $\tilde{\dot{q}}_k(\tilde{x},t)=\dot{q}_k\bigl(\vphi^{-1}(\tilde{x},t),t\bigr)$
for $\tilde{x} \in \Omega(t_0)$.
Moreover, we introduce
$\tilde{\mathcal{G}}^{ij}(\tilde{x},t)=\mathcal{G}^{ij}\bigl(\vphi^{-1}(\tilde{x},t),t\bigr)$.
Then, we see that $\tilde{\dot{q}}_k(t)$ satisfies the following (transformed) elliptic equation:
\begin{equation}\label{eq;Ltdotq_k}
    \begin{cases}
        \mathcal{L}(t)\, \tilde{\dot{q}}_k(t) 
        = -\tilde{\dot{\mathcal{L}}}(t)\,\tilde{q}_k(t)
        &\text{in }\Omega(t_0),
        \\
        \tilde{\dot{q}}_k(t)=0
        &\text{on } \pt\Omega(t_0),
    \end{cases}
\end{equation}
for all $t\in [0,T]$, where the operator $\tilde{\dot{\mathcal{L}}}(t)$ on $H^2\bigl(\Omega(t_0)\bigr)$
is defined by for $\tilde{q} \in H^2\bigl(\Omega(t_0)\bigr)$
\begin{equation}\label{eq;defi tildotLtqk}
    \begin{split}
    \bigl[\tilde{\dot{\mathcal{L}}}(t)\,\tilde{q}\bigr](\tilde{x})
    :=\,&
    \sum_{i,j,m,n} 
        \frac{\pt \vphi^m}{\pt x^i}\bigl(
        \vphi^{-1}(\tilde{x},t),t
    \bigr)
    \frac{\pt \vphi^n}{\pt x^j}\bigl(
        \vphi^{-1}(\tilde{x},t),t
    \bigr)
    \frac{\pt \tilde{ \mathcal{G}}^{ij}}{\pt \tilde{x}^n}(\tilde{x},t)
    \frac{\pt \tilde{q}}{\pt \tilde{x}^m}(\tilde{x})
    \\
    &+
    \sum_{i,j,m,n} 
    \frac{\pt \vphi^m}{\pt x^i}\bigl(
        \vphi^{-1}(\tilde{x},t),t
    \bigr)
    \frac{\pt \vphi^n}{\pt x^j}\bigl(
        \vphi^{-1}(\tilde{x},t),t
    \bigr)
    \tilde{\mathcal{G}}^{ij}(\tilde{x},t)
    \frac{\pt^2 \tilde{q}}{\pt \tilde{x}^m \pt \tilde{x}^n}(\tilde{x})
    \\
    &+
    \sum_{i,j,m} 
    \frac{\pt^2 \vphi^m}{\pt x^i \pt x^j} \bigl(\vphi^{-1} (\tilde{x},t),t\bigr)
    \tilde{\mathcal{G}}^{ij}(\tilde{x},t)\frac{\pt \tilde{q}}{\pt \tilde{x}^m}(\tilde{x}).
    \end{split}
\end{equation}
Here, we note that we easily see $\tilde{\dot{\mathcal{L}}}(t_0)=\dot{\mathcal{L}}(t_0)$ 
on $H^2\bigl(\Omega(t_0)\bigr)$
by \eqref{eq;vphitoId} and \eqref{eq;ptvphito0}. 
Indeed, 
\begin{equation}\label{eq;tildotLt0}
    \begin{split}
    \bigl[\tilde{\dot{\mathcal{L}}}(t_0)\,\tilde{q}\bigr](\tilde{x})
    =\,&
    \sum_{i,j,m,n} 
        \delta_{m,i} \delta_{n,j}
    \frac{\pt { \mathcal{G}}^{ij}}{\pt \tilde{x}^n}(\tilde{x},t_0)
    \frac{\pt \tilde{q}}{\pt \tilde{x}^m}(\tilde{x})
    \\
    &+
    \sum_{i,j,m,n} 
    \delta_{m,i} \delta_{n,j}
    {\mathcal{G}}^{ij}(\tilde{x},t_0)
    \frac{\pt^2 \tilde{q}}{\pt \tilde{x}^m \pt \tilde{x}^n}(\tilde{x},)
    \\
    =\,&
    \sum_{i,j} 
    \frac{\pt { \mathcal{G}}^{ij}}{\pt \tilde{x}^j}(\tilde{x},t_0)
    \frac{\pt \tilde{q}}{\pt \tilde{x}^i}(\tilde{x})
    +
    \sum_{i,j} 
    {\mathcal{G}}^{ij}(\tilde{x},t_0)
    \frac{\pt^2 \tilde{q}}{\pt \tilde{x}^i \pt \tilde{x}^j}(\tilde{x})
    \\
    =\,& \bigl[\dot{\mathcal{L}}(t_0)\tilde{q}\bigr](\tilde{x}).
    \end{split}
\end{equation}
Then, we obtain the time continuity of $\tilde{\dot{q}}_k(t)$ at $t_0$.
\begin{lemma}\label{lem;diff conti q_k}
    Let $t_0\in [0,T]$ and let $\tilde{\dot{q}}_k(t)$ is the solution of \eqref{eq;Ltdotq_k} for $t \in [0,T]$,
    $k=1,\dots,K$. 
    For $\ep>0$ there exists $\delta>0$ such that for every $t \in [0,T]$
    if $|t-t_0|<\delta$ then 
    \begin{equation*}
        \| \tilde{\dot{q}}_k(t) - \dot{q}_k (t_0) \|_{H^2(\Omega(t_0))} <\ep,
    \end{equation*}
    for $k=1,\dots,K$.
\end{lemma}
\begin{proof}
    In order to apply the a priori estimate \eqref{eq;aprioriq},
    from \eqref{eq;Lap val q_k} and \eqref{eq;Ltdotq_k}
    we confirm that
    \begin{equation}\label{eq;tilddotqt-dotqt0}
        \begin{split}
        \mathcal{L}\bigl(\tilde{\dot{q}}_k(t) -\dot{q}_k(t_0)\bigr) 
        &= -\bigl(\tilde{\dot{\mathcal{L}}}(t)-\dot{\mathcal{L}}(t_0)\bigr) \tilde{q}(t)
        - \dot{\mathcal{L}}(t_0)\bigl(\tilde{q}(t)-q(t_0)\bigr)
        -\bigl(\mathcal{L}(t)-\Delta_{\tilde{x}}\bigr) \dot{q}_k(t_0),
        \end{split}
    \end{equation}
    and $\tilde{\dot{q}}_k(t) - \dot{q}_k(t_0)=0$ on $\pt \Omega(t_0)$.
    
    Firstly, by \eqref{eq;vphitoId} and \eqref{eq;ptvphito0} noting that for $i,j,m,n=1,2,3$,
    \begin{gather*}
        \sup_{\tilde{x} \in \overline{\Omega(t_0)}}
        \left|
            \frac{\pt \vphi^m}{\pt x^i}\bigl(
            \vphi^{-1}(\tilde{x},t),t
        \bigr)
        \frac{\pt \vphi^n}{\pt x^j}\bigl(
            \vphi^{-1}(\tilde{x},t),t
        \bigr)
        \frac{\pt \tilde{ \mathcal{G}}^{ij}}{\pt \tilde{x}^n}(\tilde{x},t)
        -
        \delta_{m,i} \delta_{n,j}
        \frac{\pt { \mathcal{G}}^{ij}}{\pt \tilde{x}^n}(\tilde{x},t_0)
        \right|\to 0,
        \\
        \sup_{\tilde{x}\in \overline{\Omega(t_0)}}
        \left| 
            \frac{\pt \vphi^m}{\pt x^i}\bigl(
            \vphi^{-1}(\tilde{x},t),t
        \bigr)
        \frac{\pt \vphi^n}{\pt x^j}\bigl(
            \vphi^{-1}(\tilde{x},t),t
        \bigr)
        \tilde{\mathcal{G}}^{ij}(\tilde{x},t)
        -
        \delta_{m,i} \delta_{n,j}
        {\mathcal{G}}^{ij}(\tilde{x},t_0)
        \right|\to 0,
    \end{gather*}
    as $t \to t_0$ and noting that $\limsup\limits_{t\to t_0} \| \tilde{q}_k(t) \|_{H^2(\Omega(t_0))}<\infty$ by
    Lemma \ref{lem;contiq_k}, we obtain
   \begin{equation*}
    \bigl\| \bigl(\tilde{\dot{\mathcal{L}}}(t)-\dot{\mathcal{L}}(t_0)\bigr) \tilde{q}_k(t)\bigr\|_{L^2(\Omega(t_0))}
        \to 0
      \quad \text{as } t\to t_0.
   \end{equation*}
   Secondly, by Lemma \ref{lem;contiq_k}, we have 
   $\bigl\| \dot{\mathcal{L}}(t_0)\bigl(\tilde{q}_k(t) -q_k(t_0)\bigr)\bigr\|_{L^2(\Omega(t_0))} \to 0$
   as $t \to t_0$.
    The third term of the R.H.S. of \eqref{eq;tilddotqt-dotqt0} tends to 0 as $t\to t_0$ 
    by the same manner as in \eqref{eq;Lt-Delta}.
    This completes the proof.
\end{proof}
As a summary, we obtain the following theorem for $\widetilde{q}_k(t)$ on $\widetilde{\Omega}$ as in 
\eqref{eq;qkkk}.
\begin{theorem}\label{thm;contidiff q}
    Let $T>0$. It holds that $\widetilde{q}_k \in C^1\bigl([0,T];H^2(\widetilde{\Omega})\bigr)$
    for $k=1,\dots,K$.
\end{theorem}
\begin{proof} Let $t_0\in [0,T]$ be fixed.
    For the transformation: $\tilde{q}_k(t) \text{ on } \Omega(t_0)
    \longmapsto \widetilde{\tilde{q}}_k(t) \text{ on } \widetilde{\Omega}$, 
    we recall that $\widetilde{\tilde{q}}_k(t)=\widetilde{q}_k(t)$ in $\widetilde{\Omega}$ 
    as in \eqref{eq;equiv tt_0tilde q}.
    Moreover, we introduce the transformation: 
    $\dot{q}_k(t_0) \text{ on } \Omega(t_0) 
    \longmapsto \widetilde{\dot{q}}_k(t_0) \text{ on } \widetilde{\Omega}$.

    Firstly, the existence of time derivative of $\widetilde{q}_k(t)$ at $t_0$ 
    follows from Lemma \ref{lem;diff q_k}. Indeed, by Proposition \ref{prop;funcsp},
    we see that
    \begin{equation*}
        \begin{split}
            \left\| \frac{\widetilde{q}_k(t_0+\ep) - \widetilde{q}_k(t_0)}{\ep} 
            - \widetilde{\dot{q}}_k(t_0)\right\|_{H^2(\widetilde{\Omega})}
            &=
            \left\| \frac{\widetilde{\tilde{q}}_k(t_0+\ep) - \widetilde{q}_k(t_0)}{\ep} 
            - \widetilde{\dot{q}}_k(t_0)\right\|_{H^2(\widetilde{\Omega})}
            \\
            &\leq 
            C\left\| \frac{\tilde{q}_k(t_0+\ep) - {q}_k(t_0)}{\ep} 
            - {\dot{q}}_k(t_0)\right\|_{H^2({\Omega(t_0)})}
            \\
            &\to 0 \quad\text{as } \ep \to 0,
        \end{split}
    \end{equation*}
    where we can choose the constant $C>0$ is independent of $\ep$ near the origin.
    Therefore, we conclude that $\pt_s\widetilde{q}_k(t_0)=\widetilde{\dot{q}}_k(t_0)$
    in $\widetilde{\Omega}$.

    Next, we cosider the time continuity of $\pt_s \widetilde{q}_k(t)$ at $t_0$.
    Since $\widetilde{\dot{q}}_k(t) = \widetilde{\tilde{\dot{q}}}_k(t)$ in $\widetilde{\Omega}$ 
    as the above,
    by Lemma \ref{lem;diff conti q_k} and Proposition \ref{prop;funcsp}, we see that
    \begin{equation*}
        \begin{split}
            \| \pt_s \widetilde{q}_k(t) - \pt_s \widetilde{q}_k(t_0) \|_{H^2(\widetilde{\Omega})}
            &=
            \| \widetilde{\dot{q}}_k(t) - \widetilde{\dot{q}}_k(t_0) \|_{H^2(\widetilde{\Omega})}
            \\
            &=
            \|\widetilde{\tilde{\dot{q}}}_k(t) - \widetilde{\dot{q}}_k(t_0) \|_{H^2(\widetilde{\Omega})}
            \\
            &\leq
            C\| \tilde{\dot{q}}_k(t) - \dot{q}_k(t_0) \|_{H^2(\Omega(t_0))}
            \\
            &\to 0 \quad \text{as } t \to t_0,
        \end{split}
    \end{equation*}
    where the constant is independent of $t \in [0,T]$. 
    This completes the proof.
\end{proof}
\subsubsection{Orthogonal basis of the harmonic vector fields}
We consider the Gram-Schmidt orthonormalization
of the basis $\nabla q_1(t),\dots,\nabla q_K(t)$ of $V_{\mathrm{har}}\bigl(\Omega(t)\bigr)$
in the sense of $L^2\bigl(\Omega(t)\bigr)$. 
Let $\eta_1(t), \dots, \eta_K(t)$ be an orthonormal system constructed by the following procedure.
\begin{equation}\label{eq;cons Vhar}
    \begin{split}
    \eta_1(t) &= \frac{\nabla q_1(t)}{\|\nabla q_1(t) \|_{L^2(\Omega(t))}} 
    \\&
    =: \alpha_{11}(t)\nabla q_1(t),
    \\
    \eta_2(t) &= \frac{\nabla q_2(t) -\bigl(\nabla q_2(t),\eta_1(t)\bigr)\eta_1(t)}{\|
    \nabla q_2(t) -\bigl(\nabla q_2(t),\eta_1(t)\bigr)\eta_1(t)\|_{L^2(\Omega(t))}}
    \\&
    =: \alpha_{21}(t)\nabla q_1(t) + \alpha_{22}(t) \nabla q_2(t),
    \\
    &\vdots
    \\
    \eta_K(t) &=
    \frac{\nabla q_K(t) -\sum\limits_{k=1}^{K-1}\bigl(\nabla q_K(t),\eta_k(t) \bigr)\eta_k(t)}{
        \left\|\nabla q_K(t) -\sum\limits_{k=1}^{K-1}\bigl(\nabla q_K(t),\eta_k(t) \bigr)\eta_k(t)\right\|_{L^2(\Omega(t))}
    }
    \\&
    =:\alpha_{K1}(t)\nabla q_1(t) + \dots+ \alpha_{KK}(t)\nabla q_K(t).
    \end{split}
\end{equation}
Here, we put coefficients $\alpha_{jk}(t)$ for $j,k=1,\dots,K$ in the above, by setting $\alpha_{jk}(t)=0$ for $k>j$.

Moreover, we introduce transformed vector fields 
$\widetilde{\eta}_1(t),\dots,\widetilde{\eta}_K(t)$ on $\widetilde{\Omega}$
defined by,    for $i=1,2,3$, 
\begin{equation}\label{eq;cons Vhartilde}
    \widetilde{\eta}^{i}_k(t)=\sum_{\ell}\frac{\pt y^i}{\pt x^\ell}\eta^\ell_k(t),
    \quad\text{i.e.,}\quad
    \widetilde{\eta}_k^i(y,t) := \sum_{\ell} \frac{\pt \phi^i}{\pt x^\ell} \bigl(\phi^{-1}(y,t),t\bigr)
    \eta_k^\ell\bigl(\phi^{-1}(y,t),t\bigr)
\end{equation}
for all $y\in \widetilde{\Omega}$, for all $t\in[0,T]$, and for $k=1,\dots,K$.
\begin{lemma}\label{lem;alpha}
    Let $T>0$. It holds that
    \begin{equation*}
        \alpha_{jk} \in C^1\bigl([0,T]\bigr)
        \quad \text{and}\quad
        \widetilde{\eta}_k \in C^1\bigl([0,T];H^1(\widetilde{\Omega})\bigr).
    \end{equation*}
    for $j,k=1,\dots,K$.
\end{lemma}
\begin{proof}
    We prove this lemma by induction argument.
    Firstly, noting Corollary \ref{cor;CqC}, we see that
    \begin{equation}\label{eq;norm nq}
        0\neq \| \nabla q_1(t) \|^2_{L^2(\Omega(t))}
        = \int_{\widetilde{\Omega}} 
        \sum_{k,\ell} g^{k\ell}(y,t) \frac{\pt \widetilde{q}_1}{\pt y^k} (y,t)
        \frac{\pt \widetilde{q}_1}{\pt y^\ell}(y,t) J(t)\,dy.
    \end{equation}
    Since $g^{k\ell} \in C^\infty(\overline{\widetilde{Q}}_\infty)$, 
    $J \in C^\infty\bigl([0,T]\bigr)$
    and $\widetilde{q}_1 \in C^1\bigl([0,T];H^2(\widetilde{\Omega})\bigr)$ by Theorem \ref{thm;contidiff q},
    the right hand side of \eqref{eq;norm nq} is differentiable with respect to $t$. 
    Hence, $\alpha_{11}\in C^1\bigl([0,T]\bigr)$.
    Therefore,  $\widetilde{\eta}_1 \in C^1\bigl([0,T];H^1(\widetilde{\Omega})\bigr)$,
    since $\widetilde{\eta}_1(t)=\alpha_{11}(t) \nabla_g \widetilde{q}_1(t)$.

    Next, let us assume
    $\widetilde{\eta}_{k} \in C^1\bigl([0,T];H^1(\widetilde{\Omega})\bigr)$ 
    and $\alpha_{k,\ell}\in C^1\bigl([0,T]\bigr)$ for $k=1,\dots,j$ and $\ell=1,\dots,K$.
    Suppose 
    \begin{equation*}
        \eta_{j+1}(t) = 
        \frac{\nabla q_{j+1}(t) -\sum\limits_{k=1}^{j}\bigl(\nabla q_{j+1}(t),\eta_k(t) \bigr)\eta_k(t)}{
        \left\|\nabla q_{j+1}(t) -\sum\limits_{k=1}^{j}
        \bigl(\nabla q_{j+1}(t),\eta_k(t) \bigr)\eta_k(t)\right\|_{L^2(\Omega(t))}}.
    \end{equation*}

    Here, we shall show that the denominator satisfies
    \begin{equation}\label{eq;alpha deno}
        \left\| \nabla q_{j+1}(t) -\sum_{k=1}^j \bigl(\nabla q_{j+1}(t),\eta_k(t)\bigr)\eta_k(t)\right\|_{L^2(\Omega(t))} \neq 0
    \end{equation}
    and is in $C^1\bigl([0,T]\bigr)$ as a function of variable $t$.
     Indeed, if it takes zero at some $t_0 \in [0,T]$ we have
    $\displaystyle \nabla q_{j+1}(t_0) -\sum_{k}^{j} \bigl(\nabla q_{j+1}(t_0),\eta_k(t_0)\bigr)\eta_k(t_0)=0$.
    However, this means $\nabla q_{1}(t_0),\dots, \nabla q_{j+1}(t_0)$ is linearly dependent, which makes a contradiction.
    Since
    \begin{equation}\label{eq;alpha j1}
        \bigl( \nabla q_{j+1}(t),\eta_k(t)\bigr)
        =
        \bigl\langle \nabla_g\widetilde{q}_{j+1}(t), \widetilde{\eta}_k(t) \bigr\rangle_t
        =\int_{\widetilde{\Omega}} \sum_{i} \frac{\pt \widetilde{q}_{j+1}}{\pt y^i} (y,t)
        \widetilde{\eta}_k^i(y,t) J(t)\,dy,
    \end{equation}
    for $k=1,\dots j$, we note that the right hand side of \eqref{eq;alpha j1} is in $C^1\bigl([0,T]\bigr)$
    as a function of $t$.
    Then, we see that the transformed vector
    \begin{equation*}
        \nabla_g \widetilde{q}_{j+1} -\sum_{k=1}^j \bigl(\nabla q_{j+1}(\cdot),\eta_k{(\cdot)}\bigr)\widetilde{\eta}_k
        \in C^1\bigl([0,T];H^1(\widetilde{\Omega})\bigr).
    \end{equation*}
    By the same argument on \eqref{eq;norm nq}, 
    we see that the left hand side of \eqref{eq;alpha deno} is in $C^{1}\bigl([0,T]\bigr)$
    as a function of $t$.

    Therefore, combining the assumption of induction, 
    we conclude that $\alpha_{j+1\,k}\in C^1\bigl([0,T]\bigr)$ for $k=1,\dots,j+1$ are well defined
    and that $\widetilde{\eta}_{j+1}\in C^1\bigl([0,T];H^1(\widetilde{\Omega})\bigr)$. This completes the proof.
\end{proof}
%
\subsubsection{Time dependence of the harmonic vector fields of the decomposition}
At the end of this subsection, we shall establish the time dependence of the 
$V_{\mathrm{har}}\bigl(\Omega(t)\bigr)$-component $h(t)$ in 
the Helmholtz-Weyl decomposition $h(t) + \rot\, w(t) = b(t)$ on $\Omega(t)$.

Since $V_{\mathrm{har}}\bigl(\Omega(t)\bigr)$ is spanned by
the orthonormal basis $\eta_1(t),\dots,\eta_k(t)$ constructed by \eqref{eq;cons Vhar} in the $L^2$-sense,
$h(t) \in V_{\mathrm{har}}\bigl(\Omega(t)\bigr)$ can be uniquely expressed by
\begin{equation}\label{eq;represent h}
    h(t) = \sum_{k=1}^K \bigl(b(t), \eta_k(t)\bigr)\eta_k(t)
    \quad\text{in } \Omega(t).
\end{equation}

Here, we recall the notation of the coordinate transformation for a vector field
$u$ on $Q_T$ as in  \eqref{eq;transform u} , i.e., for $i=1,2,3$,
\begin{equation*}
    \widetilde{u}^i = \sum_{\ell}\frac{\pt y^i}{\pt x^\ell} u^\ell,
    \quad \text{i.e.,} \quad
    \widetilde{u}^i(y,t) = \sum_{\ell} \frac{\pt \phi^i}{\pt x^\ell}
    \bigl(\phi^{-1}(y,t),t\bigr) u^\ell\bigl(\phi^{-1}(y,t),t\bigr)
    \quad \text{for } y \in \widetilde{\Omega}.
\end{equation*}

\begin{theorem}\label{thm;contideffi har}
    Let $T>0$. Let $b(t) \in H^1\bigl(\Omega(t)\bigr)$ satisfy $\Div\, b(t)=0$ in $\Omega(t)$ for 
    $t\in [0,T]$ and let $\displaystyle h(t)=\sum_{k=1}^K \bigl(b(t),\eta_k(t)\bigr)\eta_k(t)$
    for $t \in [0,T]$. If $\widetilde{b} \in C\bigl([0,T];H^1(\widetilde{\Omega})\bigr)$,
    then it holds that
    \begin{equation*}
        \widetilde{h} \in C\bigl([0,T];H^1(\widetilde{\Omega})\bigr).
    \end{equation*}
    Furthermore, if $\widetilde{b}\in C^1\bigl([0,T];H^1(\widetilde{\Omega})\bigr)$,
    then it holds that
    \begin{equation*}
        \widetilde{h} \in C^1 \bigl([0,T];H^1(\widetilde{\Omega})\bigr).
    \end{equation*}
\end{theorem}
\begin{proof}
    To begin with, we see that
    \begin{equation}\label{eq;coeff th}
        \bigl(b(t),\eta_k(t)\bigr)
        =\bigl\langle \widetilde{b}(t), \widetilde{\eta}_k(t)\bigr\rangle_t
        =\int_{\widetilde{\Omega}} 
        \sum_{i,j} g_{ij}(y,t) \widetilde{b}^i(y,t)\widetilde{\eta}^j_k(y,t)J(t)\,dy.
    \end{equation}
Since $\widetilde{\eta}_k \in C^1\bigl([0,T];H^1(\widetilde{\Omega})\bigr)$ 
for $k=1,\dots,K$ by Lemma \ref{lem;alpha}, 
$g_{ij}\in C^\infty(\overline{\widetilde{Q}}_\infty)$ for $i,j=1,2,3$ and
 $J\in C^\infty\bigl([0,T]\bigr)$, 
the right hand side is in $C^m \bigl([0,T]\bigr)$
as a function of variable $t$, provided $\widetilde{b} \in C^m\bigl([0,T];H^1(\widetilde{\Omega})\bigr)$
with some $m \in \{0,1\}$. 
Therefore, since 
$\displaystyle\widetilde{h}(t) = \sum_{k=1}^K \bigl(b(t),\eta_k(t)\bigr)\widetilde{\eta}_k(t)$,
we conclude that $\widetilde{h}\in C^m\bigl([0,T];H^1(\widetilde{\Omega})\bigr)$
if $\widetilde{b} \in C^m\bigl([0,T];H^1(\widetilde{\Omega})\bigr)$ with some $m\in \{0,1\}$. 
This completes the proof.
\end{proof}
\subsection{Time dependence of the vector potentials}
\label{subsec;rot}
%
For the analysis of the time (or domain) dependence of the vector potential 
$w(t) \in H^2\bigl(\Omega(t)\bigr)$ in the Helmholtz-Weyl decomposition
for $b(t) \in H^1\bigl(\Omega(t)\bigr)$ on each $\Omega(t)$ for $t\in\R$,
we focus on a fact that $w(t)$ satisfies the following strictly elliptic system;
\begin{equation}\label{eq;Lap w}
\begin{cases}
    -\Delta w(t) = \rot\, b(t) &\text{in } \Omega(t),
    \\
    \rot\, w(t) \times \nu(t) = b(t) \times \nu(t)  &\text{on } \pt \Omega(t),
    \\
    w(t) \cdot \nu(t) =0   &\text{on } \pt \Omega(t),
\end{cases}
\end{equation}
where $\nu(x,t)=\bigl(\nu^1(x,t),\nu^2(x,t),\nu^2(x,t)\bigr)$
is the outward unit normal vector at $x \in \pt\Omega(t)$. 

To consider the time continuity and differentiability at a fixed $t_0 \in [0,T]$
of the transformed vector $\widetilde{w}(t)$ on $\widetilde{\Omega}$, 
it is convenient to introduce the transformation: $w(t) \text{ on } \Omega(t)
\longmapsto \tilde{w}(t) \text{ on } \Omega(t_0)$ defined by \eqref{eq;identify tilde} 
in Section \ref{subsec;unit normal}.
Moreover, in such a situation, we use the diffeomorphism 
$\varphi(\cdot,t):\overline{\Omega(t)} \to \overline{\Omega(t_0)}$
defined in \eqref{eq;diffeo t0} and inherit the notations $J(t)$, $g_{ij}$, $g^{ij}$, $\Gamma^i_{jk}$,
\dots, $\Rot(t)$ replacing $\phi(\cdot,t)$ by $\varphi(\cdot,t)$.
Especially, we remark that the kernel functions $R^{i,1}_{j,k,\ell}$, 
$R^{i,2}_{j,k,\ell}$ of $\Rot(t)$ in \eqref{eq;def rot}
are modified as 
\begin{equation}\label{eq;modified R}
    \begin{split}
    R^{i,1}_{j,k,\ell}(\tilde{x},t)
    &:=
    \frac{\pt \vphi^i}{\pt x^j} (x,t) \biggl( 
        \frac{\pt \vphi^k}{\pt x^{\sigma(j+1)}}(x,t) \frac{\pt^2 \vphi^{-1}_{\sigma(j+2)}}{\pt \tilde{x}^k\pt \tilde{x}^\ell} (\tilde{x},t)
        -
        \frac{\pt \vphi^k}{\pt x^{\sigma(j+2)}} (x,t)\frac{\pt^2 \vphi^{-1}_{\sigma(j+1)}}{\pt \tilde{x}^k\pt \tilde{x}^\ell} (\tilde{x},t) 
        \biggr),\\
    R^{i,2}_{j,k,\ell}(\tilde{x},t)
    &:=
    \frac{\pt \vphi^i}{\pt x^j} (x,t) \biggl( 
        \frac{\pt \vphi^k}{\pt x^{\sigma(j+1)}}(x,t) 
        \frac{\pt \vphi^{-1}_{\sigma(j+2)}}{\pt \tilde{x}^\ell} (\tilde{x},t)
        -
        \frac{\pt \vphi^k}{\pt x^{\sigma(j+2)}} (x,t)
        \frac{\pt \vphi^{-1}_{\sigma(j+1)}}{\pt \tilde{x}^\ell} (\tilde{x},t) 
        \biggr)
    \end{split}
\end{equation}
with $x=\vphi^{-1}(\tilde{x},t)$ for $\tilde{x}\in \Omega(t_0)$, which plays an essential role in this subsection.
Then, from \eqref{eq;eq vphiId} to \eqref{eq;ptvphito0}, we note that $\Rot(t_0)=\rot_{\tilde{x}}$ on $\Omega(t_0)$.

Let us transform \eqref{eq;Lap w} to an equivalent one on $\Omega(t_0)$ for $\tilde{w}(t)$ as follows.
\begin{equation}\label{eq;L(t)w(t)}
    \begin{cases}
        -L(t)\, \tilde{w}(t) = \Rot(t)\, \tilde{b}(t), & \text{in } \Omega(t_0),
        \\
        B_1[\Rot(t)\,\tilde{w}(t),\tilde{\nu}(t)]=B_1[\tilde{b}(t),\tilde{\nu}(t)] & \text{on } \pt\Omega(t_0),
        \\
        B_2[\tilde{w}(t),\tilde{\nu}(t)] = 0 & \text{on } \pt \Omega(t_0).
    \end{cases}
\end{equation}
Here, $L(t)$ as in Section \ref{subsec;formulation}, 
$\Rot(t)$ as in \eqref{eq;def rot} modified by \eqref{eq;modified R}, 
and 
\begin{align*}
    B_1[\tilde{u},\tilde{v}]^i &:= 
    \sum_{j,k,\ell} S^i_{j,k,\ell} \left(
        \tilde{u}^k \tilde{v}^\ell
        -
        \tilde{u}^\ell \tilde{v}^k
    \right), \quad i=1,2,3,
    \\
    B_2[\tilde{u},\tilde{v}]
    &:= \sum_{k,\ell} g_{k\ell} \tilde{u}^k \tilde{v}^\ell,
\end{align*}
for $\tilde{u}, \tilde{v}$ on $\pt \Omega(t_0)$ where we put 
\begin{equation*}
    S^i_{j,k,\ell}(\tilde{x},t)
    =
    \frac{\pt \vphi^i}{\pt x^j}\bigl(\vphi^{-1}(\tilde{x},t),t\bigr)
    \frac{\pt \vphi^{-1}_{\sigma(j+1)}}{\pt \tilde{x}^k} (\tilde{x},t)
    \frac{\pt \vphi^{-1}_{\sigma(j+2)}}{\pt \tilde{x}^\ell}(\tilde{x},t)
    \quad \text{for } \tilde{x} \in \Omega(t_0),
\end{equation*}
which is derived from the transformation
\begin{equation*}
    \begin{split}
        \sum_{j} \frac{\pt \tilde{x}^i}{\pt x^j}[u\times v]^j
        &=
        \sum_j 
        \frac{\pt \tilde{x}^i}{\pt x^j} 
        \bigl(
    u^{\sigma(j+1)}v^{\sigma(j+2)} - u^{\sigma(j+2)} v^{\sigma(j+1)}
    \bigr)
    \\
    &=
    \sum_{j} \frac{\pt \tilde{x}^i}{\pt x^j} 
    \left(
        \sum_{k} \frac{\pt x^{\sigma(j+1)}}{\pt \tilde{x}^k} 
        \tilde{u}^k
        \sum_{\ell} \frac{\pt x^{\sigma(j+2)}}{\pt \tilde{x}^\ell} 
        \tilde{v}^\ell 
        -
        \sum_{\ell} \frac{\pt x^{\sigma(j+2)}}{\pt \tilde{x}^\ell} 
        \tilde{u}^\ell
        \sum_{k} \frac{\pt x^{\sigma(j+1)}}{\pt \tilde{x}^k} 
        \tilde{v}^k
    \right)
    \\
    &=
    \sum_{j,k,\ell} 
    \frac{\pt \tilde{x}^i}{\pt x^j}
    \frac{\pt x^{\sigma(j+1)}}{\pt \tilde{x}^k}
    \frac{\pt x^{\sigma(j+2)}}{\pt \tilde{x}^\ell}
    \left(
        \tilde{u}^k \tilde{v}^\ell - \tilde{u}^\ell \tilde{v}^k
    \right),
    \end{split}
\end{equation*}
where the index $\sigma(i)$ is defined by \eqref{eq;defi sigma}.
%
%
\subsubsection{Expression of outward unit normal vectors on $\pt \Omega(t)$}
In order to deal with the transformed problem \eqref{eq;L(t)w(t)} on $\Omega(t_0)$ for a fixed $t_0\in [0,T]$, 
we shall prepare an explicit expression of $\nu(t)$ on $\pt\Omega(t)$ 
and its transformed vector $\tilde{\nu}(t)$ on $\pt\Omega(t_0)$. 

For this purpose, let us introduce a finite open covering $\{ O_\kappa \}$ in $\R^3$ of $\pt \Omega(t_0)$.
We can take $\tilde{G}_\kappa \in C^\infty(O_\kappa)$ which satisfies 
$\tilde{G}_\kappa(\tilde{x})=0$ if and only if $\tilde{x} \in \pt \Omega(t_0)\cap O_\kappa$
and satisfies 
\begin{equation*}
    \nu(\tilde{x},t_0)= \frac{\nabla_{\tilde{x}} \tilde{G}(\tilde{x})}{|\nabla_{\tilde{x}}\tilde{G}(\tilde{x})|}
    \quad
    \text{ for } \tilde{x} \in \pt \Omega(t_0) \cap O_\kappa.
\end{equation*}
Hereafter, we abbreviate $\tilde{G}_\kappa$ as $\tilde{G}$ for simplicity.
Now, we put $G(x,t):=\tilde{G}\bigl(\vphi(x,t)\bigr)$ for $x \in \Omega(t)$ 
with a suitable extension near a neighborhood of some boundary portion of $\pt\Omega(t)$ which contains $\vphi^{-1}(\pt\Omega(t_0)\cap O_\kappa)$ 
with the associated $\kappa$, if necessary. Then we may express that
\begin{equation*}
    \nu(x,t) = \frac{\nabla G (x,t)}{|\nabla G(x,t)|}
    \quad \text{for } x \in \pt\Omega(t).
\end{equation*}
We put  for $\tilde{x} \in \Omega(t_0)$,
\begin{equation*}
    \begin{split}
    D(\tilde{x},t)&:=\sum_{i}\left(\sum_{\ell} 
    \frac{\pt {\vphi}^\ell}{\pt x^i} \bigl(\vphi^{-1}(\tilde{x},t),t\bigr)
    \frac{\pt \tilde{G}}{\pt \tilde{x}^\ell}(\tilde{x})\right)^2,
    \\
    E(\tilde{x}) &:= |\nabla_{\tilde{x}} \tilde{G}(\tilde{x})|^2
    =\sum_i \left(\frac{\pt \tilde{G}}{\pt \tilde{x}^i}(\tilde{x})\right)^2.
    \end{split}
\end{equation*}
Then noting that $D(\tilde{x},t)=|\nabla G(x,t)|^2$ with $x=\vphi^{-1}(\tilde{x},t)$,  
we see that for $i=1,2,3$, 
\begin{equation}\label{eq;def nu}
    \begin{split}
    \tilde{\nu}^i (\tilde{x},t) 
    &=\sum_{k} \frac{\pt \tilde{x}^i}{\pt x^k} \nu^k(x,t) =\sum_{k} \frac{\pt \vphi^i}{\pt x^k} (x,t) 
    \frac{\displaystyle\frac{\pt G}{\pt x^k}(x,t) }{{|\nabla G(x,t)|}}
    =
    \sum_{k} g^{ik} (\tilde{x},t)
    \frac{\displaystyle\,\frac{\pt \tilde{G}}{\pt \tilde{x}^k}(\tilde{x})\,}{D(\tilde{x},t)}.
    \end{split}
\end{equation}
Here we note that since, by \eqref{eq;vphitoId}, 
$D(\tilde{x},t) \to E(\tilde{x})$ 
uniformly in some neighborhood of $\pt \Omega(t_0)$ as $t\to t_0$, 
we see that for $i=1,2,3$,
\begin{equation}\label{eq;nu0}
    \tilde{\nu}^i (\tilde{x},t) \to \nu^i (\tilde{x},t_0)
    \quad\text{uniformly in $\pt\Omega(t_0)$ as } t\to t_0.
\end{equation}
Moreover, by properties 
from \eqref{eq;eq vphiId} to \eqref{eq;ptvphito0} 
and by Proposition \ref{prop;gij}
we can directly derive 
$\displaystyle \frac{\pt D}{\pt \tilde{x}^j}(\tilde{x},t) \to 
\frac{\pt E}{\pt \tilde{x}^j}(\tilde{x})$ 
and
$\displaystyle \frac{\pt^2 D}{\pt \tilde{x}^j\pt \tilde{x}^k}(\tilde{x},t) \to
\frac{\pt^2 E}{\pt \tilde{x}^j\pt \tilde{x}^k}(\tilde{x})$
uniformly in $\tilde{x} \in \pt \Omega(t_0)$ as $t\to t_0$ for $j,k = 1,2,3$.
From these observation, we remark that for $i,j,k=1,2,3$,
\begin{gather}
    \label{eq;nu1}
    \sup_{\tilde{x}\in \pt \Omega(t_0)}
    \left| \frac{\pt \tilde{\nu}^i}{\pt \tilde{x}^j} (\tilde{x},t) 
    - \frac{\pt \nu^i}{\pt \tilde{x}^j}(\tilde{x},t_0)\right|
    \to 0 \quad \text{as } t\to t_0,
    \\
    \label{eq;nu2}
    \sup_{\tilde{x}\in \pt \Omega(t_0)}
    \left| \frac{\pt^2 \tilde{\nu}^i}{\pt \tilde{x}^j\pt \tilde{x}^k} (\tilde{x},t) 
    - \frac{\pt^2 \nu^i}{\pt \tilde{x}^j \pt \tilde{x}^k}(\tilde{x},t_0)\right|
    \to 0 \quad \text{as } t\to t_0,
\end{gather}
which play an important role in controlling the trace norm in the above a priori estimate.

Finally, due to \eqref{eq;nu0}, \eqref{eq;nu1} and \eqref{eq;nu2}, 
we have the following.
\begin{proposition}\label{prop;trace}
    Let $s \in [0,2]$. Then there exists a constant $C>0$ such that
    for every function $g \in H^{s}\bigl(\pt\Omega(t_0)\bigr)$ it holds that
    \begin{multline}\label{eq;traceconv}
        \bigl\|  g \bigl(\tilde{\nu}^i(t)-\nu^i(t_0)\bigr)  \bigr\|_{H^s(\pt\Omega(t_0))}
        \leq C
        \Bigg( 
            \sup_{\tilde{x}\in \pt\Omega(t_0)} 
            |\tilde{\nu}^i(\tilde{x},t)
            -
            \nu^i(\tilde{x},t_0)|
            \\+
            \sup_{\tilde{x}\in \pt\Omega(t_0)} 
            \bigl|\nabla_{\tilde{x}}\bigl(\tilde{\nu}^i(\tilde{x},t)
            -
            \nu^i(\tilde{x},t_0)\bigr)\bigr|
            +
            \sup_{\tilde{x}\in \pt\Omega(t_0)} 
            \left| \nabla_{\tilde{x}}^2\bigl(\tilde{\nu}^i(\tilde{x},t)
            -
            \nu^i(\tilde{x},t_0)\bigr)\right|
        \Bigg)
        \| g \|_{H^s(\pt \Omega(t_0))}
    \end{multline}
    for all $i=1,2,3$ and $t\in\R$.
\end{proposition}
\begin{proof}
    We shall calculate its trace norm according to a definition as in Lions and Magenes \cite[Chapter 1, Section 7]{Lions Magenes}.
    Let $\{ O_\kappa\}$ be a finite family of open sets in $\R^3$, covering $\pt\Omega(t_0)$ and assume that
    for each $O_\kappa$ there exists a $C^\infty$ diffeomorphism $\gamma_\kappa$ from $O_\kappa$ to
    $Q:=\bigl\{(\xi^\prime,\xi^3) \in \R^3\,;\, |\xi^\prime|<1,\,-1 < \xi^3 <1\bigr\}$ with
    $\gamma_\kappa\bigl(\pt\Omega(t_0)\cap O_\kappa\bigr)=Q\cap \{\xi^3=0\}$.
    Furthermore, let $\{\varrho_\kappa\}$ be an associated partition of unity on $\pt\Omega(t_0)$ 
    with $\varrho_\kappa \in C_0^\infty(O_\kappa)$.
    So putting $\widetilde{\varrho_\kappa g} (\xi^\prime):= \varrho_\kappa\bigl(\gamma^{-1}_\kappa(\xi^\prime,0)\bigr)
    g\bigl(\gamma^{-1}_\kappa(\xi^\prime,0)\bigr)$ 
    for $g \in H^s\bigl(\pt\Omega(t_0)\bigr)$, the trace norm of $g$ is defined by
    \begin{equation*}
        \|g \|_{H^s(\pt \Omega(t_0))} = \left[
            \sum_{\kappa} \| \widetilde{\varrho_\kappa g} \|^2_{H^s(\R^2_{\xi^\prime})} 
        \right]^{\frac{1}{2}}.
    \end{equation*}

    Firstly, we consider $s=0$. 
    Further, we put $\Lambda(t)=\tilde{\nu}^i(t) - \nu^i(t_0)$ for simplicity.
    We note that $\widetilde{\varrho_\kappa \bigl(g \Lambda(t)\bigr)}(\xi^\prime) 
    = \widetilde{\varrho_\kappa g} (\xi^\prime) \Lambda\bigl(\gamma^{-1}_\kappa(\xi^\prime,0),t\bigr)$.
    Then, we see that
    \begin{equation*}
        \begin{split}
        \| \widetilde{\varrho_\kappa g \Lambda(t)} \|_{L^2(\R^2)}
        &\leq 
        \sup_{\pt\Omega(t_0)}|\Lambda(\tilde{x},t)| \cdot \|\widetilde{\varrho_\kappa g} \|_{L^2(\R^2)}
        \\
        &\leq 
        \left(
            \sup_{\pt\Omega(t_0)}|\Lambda(\tilde{x},t)|
            +
            \sup_{\pt\Omega(t_0)}|\nabla_{\tilde{x}} \Lambda(\tilde{x},t)|
            + 
            \sup_{\pt\Omega(t_0)}|\nabla_{\tilde{x}}^2 \Lambda(\tilde{x},t)|
        \right)\|\widetilde{\varrho_\kappa g} \|_{L^2(\R^2)}.
        \end{split}
    \end{equation*}
Next, we consider the case $s=2$.
By an elementary calculation, we see that
\begin{equation}
    \frac{\pt \bigl(\widetilde{\varrho_\kappa g}\Lambda(t)\bigr)}{\pt \xi^i}
    =
    \frac{\pt \widetilde{\varrho_\kappa g}}{\pt \xi^i} \Lambda(t) 
    + \widetilde{\varrho_\kappa g} \sum_{k} 
    \frac{\pt \gamma_{\kappa,k}^{-1}}{\pt\xi^i}\frac{\pt \Lambda}{\pt \tilde{x}^k}(t),
\end{equation}
\begin{equation}
    \begin{split}
    \frac{\pt^2 \bigl(\widetilde{\varrho_\kappa g}\Lambda(t)\bigr)}{\pt \xi^i \pt \xi^j}
    =&
    \frac{\pt^2 \widetilde{\varrho_\kappa g}}{\pt \xi^i \pt \xi^j} \Lambda(t)
    +
    \sum_{k} 
    \left(\frac{\pt \widetilde{\varrho_\kappa g}}{\pt \xi^i} 
    \frac{\pt \gamma_{\kappa,k}^{-1}}{\pt \xi^j}
    +\frac{\pt \widetilde{\varrho_\kappa g}}{\pt \xi^j}
    \frac{\pt \gamma_{\kappa,k}^{-1}}{\pt \xi^i}
    \right) \frac{\pt \Lambda}{\pt \tilde{x}^k}(t)
    \\
    &+
    \sum_{k} \widetilde{\varrho_\kappa g}
    \frac{\pt^2 \gamma_{\kappa,k}^{-1}}{\pt \xi^i \xi^j} \frac{\pt \Lambda}{\pt \tilde{x}^k}(t)
    +
    \sum_{k,\ell}
    \widetilde{\varrho_\kappa g}
    \frac{\pt \gamma_{\kappa,k}^{-1}}{\pt \xi^i}
    \frac{\pt \gamma_{\kappa,\ell}^{-1}}{\pt \xi^j}
    \frac{\pt^2 \Lambda}{\pt \tilde{x}^k \pt \tilde{x}^\ell}(t)
    \end{split}
\end{equation}
for $i,j=1,2$. 
Therefore, taking a $L^2(\R^2)$-norm, we obtain that
\begin{multline}
    \| \nabla_{\xi^\prime}^2 \bigl(
    \widetilde{\varrho_\kappa g\Lambda(t)}\bigr) \|_{L^2(\R^2)}
    \\ 
    \leq
    C 
    \left(
        \sup_{\pt\Omega(t_0)}|\Lambda(\tilde{x},t)|
        +
        \sup_{\pt\Omega(t_0)}|\nabla_{\tilde{x}} \Lambda(\tilde{x},t)|
        + 
        \sup_{\pt\Omega(t_0)}|\nabla_{\tilde{x}}^2 \Lambda(\tilde{x},t)|
    \right)
    \|  \widetilde{\varrho_\kappa g} \|_{H^2(\R^2)},
\end{multline}
where the constant $C>0$ may depend on $\|\gamma_\kappa^{-1}\|_{W^{2,\infty}(Q)} $.
This yields \eqref{eq;traceconv} for $s=2$.

Since $H^s\bigl(\pt \Omega(t_0)\bigr)$, $s \in [0,2]$ is 
the real interpolation space between $L^2\bigl(\pt\Omega(t_0)\bigr)$
and 
$H^2\bigl(\pt \Omega(t_0)\bigr)$
we obtain \eqref{eq;traceconv} with a constant independent of $\Lambda(t)$.
This completes the proof.
\end{proof}

\subsubsection{Local uniform estimate of solutions to \eqref{eq;L(t)w(t)} for $t$}
\label{sec;Apriori w}
%
Let us consider the a priori estimate for the solution of \eqref{eq;L(t)w(t)}.

To begin with, as is mentioned in \cite[Lemma 4.4]{Kozono Yanagisawa IUMJ}, 
since \eqref{eq;Lap w} takes a form of the strictly elliptic operator with 
the complementing boundary condition in the sense of Agmon, Douglis and Nirenberg
\cite{ADN}, we have an a priori estimate for \eqref{eq;Lap w}:
\begin{equation}\label{eq;apriori w}
        \| w \|_{H^2(\Omega(t))} \leq 
        C \left(
            \| \Delta w \|_{L^2(\Omega(t))} +
            \| \rot\,w \times \nu(t) \|_{H^{\frac{1}{2}}(\pt \Omega(t))}
            +
            \| w \cdot \nu(t) \|_{H^{\frac{3}{2}}(\pt \Omega(t))}
            +
            \| w \|_{L^2(\Omega(t))}
        \right)
    \end{equation}
    for all $w \in H^2\bigl(\Omega(t)\bigr)$ with some constant $C>0$ which may depend $\Omega(t)$.

    Then, by the direct calculation, we can obtain the  the equivalence between
a trace norm on $\pt\Omega(t)$ and on $\pt\Omega(t_0)$.
Therefore, from \eqref{eq;apriori w}, we have the following
\begin{multline}\label{eq;aprioriLtwCt}
    \| \tilde{w} \|_{H^2(\Omega(t_0))}
    \leq 
    C(t)
    \Big(
        \| L(t) \tilde{w} \|_{L^2(\Omega(t_0))}
       \\ +
        \| B_1[\Rot(t)\,\tilde{w}, \tilde{\nu}(t)] \|_{H^{\frac{1}{2}}(\pt \Omega(t_0))}
        +
        \| B_2[\tilde{w},\tilde{\nu}(t)] \|_{H^{\frac{3}{2}}(\pt\Omega(t_0))}
        +
        \|\tilde{w}\|_{L^2(\Omega(t_0))}
    \Big),
\end{multline}
for all $\tilde{w} \in H^2\bigl(\Omega(t_0)\bigr)$, for $t\in[0,T]$.

Hence, we have the following locally uniform estimate of $C(t)>0$ in \eqref{eq;aprioriLtwCt}:
\begin{lemma}\label{lem;apriori w}
    Let $t_0 \in [0,T]$ be fixed. There exist a constant $C=C(T,t_0)>0$ and 
    $\delta>0$ such that for every $t\in \R$ if $|t-t_0| <\delta$ then
    it holds that 
        \begin{multline}\label{eq;aprioriLtw}
            \| \tilde{w} \|_{H^2(\Omega(t_0))}
            \leq 
            C
            \Big(
                \| L(t) \tilde{w} \|_{L^2(\Omega(t_0))}
               \\ +
                \| B_1[\Rot(t)\,\tilde{w}, \tilde{\nu}(t)] \|_{H^{\frac{1}{2}}(\pt \Omega(t_0))}
                +
                \| B_2[\tilde{w},\tilde{\nu}(t)] \|_{H^{\frac{3}{2}}(\pt\Omega(t_0))}
                +
                \|\tilde{w}\|_{L^2(\Omega(t_0))}
            \Big),
        \end{multline}
    for all $\tilde{w} \in H^2\bigl(\Omega(t_0)\bigr)$.
\end{lemma}
\begin{proof}
    We prove the lemma by contradiction argument.
    Let us assume for every $m \in \N$, 
    there exist $t_m \in \R$ and $ \tilde{w}_n \in H^2\bigl(\Omega(t_0)\bigr)$ 
    with
    $\| \tilde{w}_m\|_{H^2(\Omega(t_0))} \equiv 1$ 
    such that $|t -t_0|< \frac{1}{m}$
    and such that
    \begin{multline*}
        \frac{1}{m}
        >
        \| L(t_m) \tilde{w}_m \|_{L^2(\Omega(t_0))}
        \\ +
         \| B_1[\Rot(t_m)\,\tilde{w}_m, \tilde{\nu}(t_m)] \|_{H^{\frac{1}{2}}(\pt \Omega(t_0))}
         +
         \| B_2[\tilde{w}_m,\tilde{\nu}(t_m)] \|_{H^{\frac{3}{2}}(\pt\Omega(t_0))}
         +
         \|\tilde{w}_m\|_{L^2(\Omega(t_0))}
    \end{multline*} 
   On the other hand, by the a priori estimate of \eqref{eq;apriori w} at $t_0$, 
   we have 
   \begin{equation}\label{eq;w contradict}
    \begin{split}
    1 =\,&\| \tilde{w}_m \|_{H^2(\Omega(t_0))} 
    \\
    \leq\,  &C_* \Bigl(
        \|\Delta_{\tilde{x}} \tilde{w}_m \|_{L^2(\Omega(t_0))}
        +
        \|\rot_{\tilde{x}}\, \tilde{w}_m \times \nu(t_0)\|_{H^{\frac{1}{2}}(\pt \Omega(t_0))}
       \\
       & +
        \| \tilde{w}_m \cdot \nu(t_0)\|_{H^{\frac{2}{3}}(\pt \Omega(t_0))}
        +
        \| \tilde{w}_m \|_{L^2(\Omega(t_0))}
        \Bigr)
        \\
        \leq \,&
        C_*\Bigl(
        \bigl\|\bigl( \Delta_{\tilde{x}} -L(t_m)\bigr)\tilde{w}_m\bigr\|_{L^2(\Omega(t_0))}
        + \|L(t_m) \tilde{w}_m \|_{L^2(\Omega(t_0))}
        \\
        &+
        \bigl\| \rot_{\tilde{x}}\, \tilde{w}_m \times \nu(t_0) 
        -B_1[\Rot(t_n)\,\tilde{w}_m,\tilde{\nu}(t_m)]\bigr\|_{H^{\frac{1}{2}}(\pt \Omega(t_0))}
        +
        \bigl\| B_1[\Rot(t_m)\,\tilde{w}_m,\tilde{\nu}(t_m)]\bigr\|_{H^{\frac{1}{2}}(\pt \Omega(t_0))}
        \\
        &+
        \bigl\| \tilde{w}_m \cdot \nu(t_0) - B_2[\tilde{w}_m, \tilde{\nu}(t_m)] \bigr\|_{H^{\frac{3}{2}}(\pt \Omega(t_0))}
        +
        \bigl\| B_2[\tilde{w}_m, \tilde{\nu}(t_m)] \bigr\|_{H^{\frac{3}{2}}(\pt \Omega(t_0))}
        \\
        &+ \|\tilde{w}_m \|_{L^2(\Omega(t_0))}
        \Bigr)
        \\
        \leq \,& 
        C_* \Bigl(
        \bigl\|\bigl( \Delta_{\tilde{x}} -L(t_m)\bigr)\tilde{w}_m\bigr\|_{L^2(\Omega(t_0))}
        +
        \bigl\| \rot_{\tilde{x}}\, \tilde{w}_m \times \nu(t_0) 
        -B_1[\Rot(t_m)\,\tilde{w}_m,\tilde{\nu}(t)]\bigr\|_{H^{\frac{1}{2}}(\pt \Omega(t_0))}
        \\
        & +
        \bigl\| \tilde{w}_m \cdot \nu(t_0) - B_2[\tilde{w}_m, \tilde{\nu}(t_m)] \bigr\|_{H^{\frac{3}{2}}(\pt \Omega(t_0))}
        \Bigr)
        + 
        \frac{C_*}{m}.
    \end{split}
   \end{equation}
Here, we recall that for $\tilde{w}\in H^2\bigl(\Omega(t_0)\bigr)$
\begin{equation*}
    \begin{split}
    [L(t) \tilde{w}]^i =\,& 
    \sum_{k,\ell} \frac{\pt }{\pt \tilde{x}^\ell} \left(
        g^{k\ell}
        \frac{\pt \tilde{w}^i}{\pt \tilde{x}^k}
        \right)
       \\& +
        2 \sum_{j,k,\ell} 
        g^{k\ell} \Gamma^{i}_{jk} \frac{\pt \tilde{w}^j}{\pt \tilde{x}^\ell}
        +
        \sum_{j,k,\ell}
        \left(
            \frac{\pt}{\pt \tilde{x}^k} \bigl( g^{k\ell} \Gamma^{i}_{j\ell}\bigr)
            +
            \sum_{n} g^{k\ell} \Gamma^{n}_{j\ell} \Gamma^{i}_{kn} 
        \right) \tilde{w}^j.
    \end{split}
\end{equation*}
Therefore, we have that
\begin{multline*}
    [L(t_m) \tilde{w}]^i(\tilde{x})  -\Delta_{\tilde{x}}\tilde{w}_m^i(\tilde{x})=
    \sum_{k,\ell} \frac{\pt }{\pt \tilde{x}^\ell} \left(
        \bigl(g^{k\ell} (\tilde{x},t_m) - \delta_{k,\ell}\bigr)
        \frac{\pt \tilde{w}^i}{\pt \tilde{x}^k}(\tilde{x})
        \right)
       \\ +
        2 \sum_{j,k,\ell} 
        g^{k\ell}(\tilde{x},t_m) \Gamma^{i}_{jk}(\tilde{x},t_m)
         \frac{\pt \tilde{w}^j}{\pt \tilde{x}^\ell}(\tilde{x},t_m)
        +
        \sum_{j,k,\ell}
        \bigg(
            \frac{\pt}{\pt \tilde{x}^k} \bigl( g^{k\ell} (\tilde{x},t_m)\Gamma^{i}_{j\ell}(\tilde{x},t_m)\bigr)
           \\ +
            \sum_{n} g^{k\ell}(\tilde{x},t_m) \Gamma^{n}_{j\ell} (\tilde{x},t_m)\Gamma^{i}_{kn} (\tilde{x},t_m)
        \bigg) \tilde{w}^j(\tilde{x},t_m).
\end{multline*}
Since $\| \tilde{w}_m \|_{H^2(\Omega(t_0))}\equiv 1$, by the similar argument of \eqref{eq;Lt-Delta} 
and by Proposition \ref{prop;gij}, we observe that
\begin{equation}\label{eq;L(t)-Delta}
    \bigl\| \bigl(L(t_m) -\Delta_{\tilde{x}}\bigr) \tilde{w}_m \bigr\|_{L^2(\Omega(t_0))} \to 0 \quad \text{as }{t_m \to t_0}.
\end{equation}
Next, we see that
\begin{equation}\label{eq;dec B_1 times}
    \begin{split}
        B_1&[\Rot(t_m)\,\tilde{w}_n,\tilde{\nu}(t_m)]^i -[\rot_{\tilde{x}}\,\tilde{w}_m \times \nu(t_0)]^i 
        \\
        =\,&
        \sum_{j,k,\ell}
        S^{i}_{j,k,\ell} (t_m)\Bigl(
            [\Rot(t_m)\tilde{w}_m]^k \bigl(\tilde{\nu}^\ell(t_m)-\nu^\ell(t_0)\bigr)
            -
            [\Rot(t_m)\tilde{w}_m]^\ell\bigl(\tilde{\nu}^k(t_m)-\nu^k(t_0)\bigr)
        \Bigr)
        \\
        &+
        \sum_{j,k,\ell} 
        S^{i}_{j,k,\ell} (t_m)\Bigl(
            [\Rot(t_m)\tilde{w}_m - \rot_{\tilde{x}}\,\tilde{w}_m]^k \nu^\ell(t_0)
            -
            [\Rot(t_m)\tilde{w}_m - \rot_{\tilde{x}}\,\tilde{w}_m]^\ell \nu^k(t_0)
        \Bigr)
        \\
        &+
        \sum_{j,k,\ell} 
        \Bigl(S^{i}_{j,k,\ell}(t_m)-\delta_{i,j}\delta_{\sigma(j+1),k} \delta_{\sigma(j+2),\ell}\Bigr)
        \Bigl(
            [\rot_{\tilde{x}}\,\tilde{w}_m]^k \nu^\ell(t_0) - [\rot_{\tilde{x}}\,\tilde{w}_m]^\ell \nu^k(t_0)
        \Bigr) 
        \\
        =:\,& I_1+I_2+I_3.
    \end{split}
\end{equation}
We shall estimate $I_1$. We note that $S^i_{j,k,\ell},\, R^{i,1}_{j,k,\ell},\, R^{i,2}_{j,k,\ell}
\in C^\infty \bigl(\overline{\Omega(t_0)\times (t_0-1,t_0+1)}\bigr)$
for $i,j,k,\ell=1,2,3$.
 So, it holds that
\begin{equation*}
    \sup_{|t-t_0|\leq 1}\| S^i_{j,k,\ell} (\cdot,t)\|_{C^2(\overline{\Omega(t_0)})}
    <M,\quad i,j,k,\ell=1,2,3,
\end{equation*}
with some $M>0$ and so does for $R^{i,1}_{j,k,\ell}$ and $R^{i,2}_{j,k,\ell}$.
Hence, by Proposition \ref{prop;trace} and the trace theorem, we have
\begin{equation*}
    \begin{split}
    \|I_1\|_{H^{\frac{1}{2}}(\pt\Omega(t_0))} 
    &\leq 
    C M\sum_{k} 
    \sum_{j=0}^2\sup_{\pt\Omega(t_0)}
    \left|\nabla_{\tilde{x}}^j \bigl(\tilde{\nu}^k (\tilde{x},t_m)-\nu^k(\tilde{x},t)\bigr)\right|
    \|\Rot(t_m) \tilde{w}_m\|_{H^{\frac{1}{2}}(\pt \Omega(t_0))} 
    \\
    &\leq
    C M\sum_{k} 
    \sum_{j=0}^2\sup_{\pt\Omega(t_0)}
    \left|\nabla_{\tilde{x}}^j \bigl(\tilde{\nu}^k (\tilde{x},t_m)-\nu^k(\tilde{x},t)\bigr)\right|
    \|\Rot(t_m) \tilde{w}_m\|_{H^{1}( \Omega(t_0))} 
    \\
    &\leq
    C M^2\sum_{k} 
    \sum_{j=0}^2\sup_{\pt\Omega(t_0)}
    \left|\nabla_{\tilde{x}}^j \bigl(\tilde{\nu}^k (\tilde{x},t_m)-\nu^k(\tilde{x},t)\bigr)\right|
    \|\tilde{w}_m\|_{H^{2}( \Omega(t_0))}. 
    \end{split} 
\end{equation*}
Since $ \|\tilde{w}_m \|_{H^2(\Omega(t_0))} \equiv 1$, 
\eqref{eq;nu0}, \eqref{eq;nu1} and \eqref{eq;nu2} yield that $\|I_1\|_{H^{\frac{1}{2}}(\pt\Omega(t_0))} \to 0$
as $t_m \to t_0$.

Next, we shall estimate $I_2$. by the definition of $\Rot(t)$ with $\vphi$, 
we note that 
\begin{equation}\label{eq;rot Rot}
    [\rot_{\tilde{x}}\, \tilde{w}]^i = [\Rot(t_0)\,\tilde{w} ]^i
    =
    \sum_{j,k,\ell} R^{i,2}_{j,k,\ell} (\tilde{x},t_0) \frac{\pt \tilde{w}^k}{\pt \tilde{x}^\ell}
    \quad\text{for } \tilde{w} \in H^1\bigl(\Omega(t_0)\bigr).
\end{equation}
Hence, we have
\begin{equation*}
    [\Rot(t_m)\, \tilde{w}_m - \rot_{\tilde{x}}\,w_m]^i
    =
    \sum_{j,k,\ell} R^{i,1}_{j,k,\ell} (\tilde{x},t_m) \tilde{w}_m^k
    +
    \sum_{j,k,\ell} \bigl( R^{i,2}_{j,k,\ell}(\tilde{x},t_m)
    -R^{i,2}_{j,k,\ell}(\tilde{x},t_0)\bigr) \frac{\pt \tilde{w}_m^k}{\pt \tilde{x}^\ell}. 
\end{equation*}
Moreover we note that by \eqref{eq;eq vphidelta}, \eqref{eq;convphi2}, \eqref{eq;vphitoId} 
it holds that for $i,j,k,\ell=1,2,3$,
\begin{gather}
    \label{eq;R^1}
    \| R^{i,1}_{j,k,\ell} (\cdot,t_m) \|_{C^2(\overline{\Omega(t_0)})} \to 0
    \quad \text{as } t_m \to t_0,
    \\
    \label{eq;R^2}
    \| R^{i,2}_{j,k,\ell} (\cdot,t_m) - R^{i,2}_{j,k,\ell} (\cdot,t_0)\|_{C^2(\overline{\Omega(t_0)})}
    \to 0
    \quad \text{as } t_m \to t_0.
\end{gather}
Then by an analogy of the proof of Proposition \ref{prop;trace} with taking $\Lambda=\nu(t_0)$,
we directly derive
\begin{equation*}
    \begin{split}
        \bigl\| S^i_{j,k,\ell}(t_m)& [\Rot(t_m)\,\tilde{w}_m - \rot_{x}\,\tilde{w}_m]^k \nu^\ell(t_0)
        \bigr\|_{H^{\frac{1}{2}}(\pt \Omega(t_0))}
        \\
        &\leq 
        C \sum_{n=0}^2 \sup_{\pt\Omega(t_0)} 
        \bigl|\nabla_{\tilde{x}}^n\nu^\ell(\tilde{x},t_0)\bigr|\,
        \bigl\|  S^i_{j,k,\ell} [\Rot(t_m)\,\tilde{w}_m - \rot_{x}\,\tilde{w}_m]^k
        \bigr\|_{H^{\frac{1}{2}}(\pt \Omega(t_0))}
        \\
        &\leq
        C \sum_{n=0}^2 \sup_{\pt\Omega(t_0)} 
        \bigl|\nabla_{\tilde{x}}^n\nu^\ell(\tilde{x},t_0)\bigr|\,
        \bigl\|  S^i_{j,k,\ell} [\Rot(t_m)\,\tilde{w}_m - \rot_{x}\,\tilde{w}_m]^k
        \bigr\|_{H^{1}( \Omega(t_0))}
        \\
        &\leq
        CM  \sum_{n=0}^2 \sup_{\pt\Omega(t_0)} 
        \bigl|\nabla_{\tilde{x}}^n\nu^\ell(\tilde{x},t_0)\bigr|\,
        \sum_{j^\prime,k^\prime,\ell^\prime}\Bigl(
        \| R^{k,1}_{j\prime,k^\prime,\ell^\prime} (\cdot,t_m) \|_{C^2(\overline{\Omega(t_0)})}
        \\
        &\quad +
        \| R^{k,2}_{j^\prime,k^\prime,\ell^\prime} (\cdot,t_m) 
        - R^{k,2}_{j^\prime,k^\prime,\ell^\prime} (\cdot,t_0)\|_{C^2(\overline{\Omega(t_0)})}
       \Bigr) \|\tilde{w}_m\|_{H^2(\Omega(t_0))}.
    \end{split}
\end{equation*}
Therefore, we have $\|I_2\|_{H^{\frac{1}{2}}(\pt\Omega(t_0))} \to 0$ as $t_m \to t_0$.

Furthermore, we see that for $i,j,k,\ell=1,2,3$,
\begin{equation*}
    \bigl\|S^{i}_{j,k,\ell}(t_m)-\delta_{i,j}\delta_{\sigma(j+1),k} \delta_{\sigma(j+2),\ell}
    \bigr\|_{C^2(\overline{\Omega(t_0)})}
    \to 0 
    \quad \text{as } t_m \to t_0
\end{equation*}
Hence, similarly we have
\begin{equation*}
    \begin{split}
        \bigl\|\bigl(S^{i}_{j,k,\ell}&(t_m)-\delta_{i,j}\delta_{\sigma(j+1),k} \delta_{\sigma(j+2),\ell}
        \bigr)
            [\rot_{\tilde{x}}\,\tilde{w}_m]^k \nu^\ell(t_0)\bigr\|_{H^{\frac{1}{2}}(\pt \Omega(t_0))}
        \\
        &
        \leq 
        C \sum_{n=0}^2 \sup_{\pt \Omega(t_0)} \bigl|\nabla^n_{\tilde{x}} \nu^\ell(\tilde{x},t_0)\bigr|
        \bigl\|
            \bigl(S^{i}_{j,k,\ell}(t_m)-\delta_{i,j}\delta_{\sigma(j+1),k} \delta_{\sigma(j+2),\ell}
        \bigr)
            [\rot_{\tilde{x}}\,\tilde{w}_m]^k
        \bigr\|_{H^{\frac{1}{2}}(\pt\Omega(t_0))}
        \\
        &
        \leq 
        C \sum_{n=0}^2 \sup_{\pt \Omega(t_0)} \bigl|\nabla^n_{\tilde{x}} \nu^\ell(\tilde{x},t_0)\bigr|
        \bigl\|
            \bigl(S^{i}_{j,k,\ell}(t_m)-\delta_{i,j}\delta_{\sigma(j+1),k} \delta_{\sigma(j+2),\ell}
        \bigr)
            [\rot_{\tilde{x}}\,\tilde{w}_m]^k
        \bigr\|_{H^{1}(\Omega(t_0))}
        \\
        &
        \leq 
        C \sum_{n=0}^2 \sup_{\pt \Omega(t_0)} \bigl|\nabla^n_{\tilde{x}} \nu^\ell(\tilde{x},t_0)\bigr|
        \bigl\| S^{i}_{j,k,\ell}(t_m)-\delta_{i,j}\delta_{\sigma(j+1),k} \delta_{\sigma(j+2),\ell}
        \bigr\|_{C^2(\overline{\Omega(t_0)})}
        \| \tilde{w}_m \|_{H^2(\Omega(t_0))}.
    \end{split}
\end{equation*}
Therefore we obtain $\|I_3\|_{H^{\frac{1}{2}}(\pt \Omega(t_0))} \to 0$
as $t_m\to t_0$.

Finally, since $\displaystyle \tilde{w}_m\cdot \nu(t_0)
=\sum_{k,\ell} \delta_{k,\ell}\tilde{w}^k_m \nu^\ell(t_0) $, we see that
\begin{equation}\label{eq;dec B2}
    \begin{split}
        B_2[\tilde{w}_n,\tilde{\nu}(t_m)] - &\tilde{w}_m \cdot \nu(t_0)
        \\
        &=
        \sum_{k,\ell} g_{k\ell}(t_m) \tilde{w}_m^k \bigl(\tilde{\nu}^\ell(t_m) - \nu^\ell(t_0)\bigr)
        +
        \sum_{k,\ell} \bigl(g_{k\ell}(t_m) - \delta_{k,\ell}\bigr) \tilde{w}_m^k \nu^\ell(t_0)
        \\
        &=:J_1+J_2.
    \end{split}
\end{equation}
Since $\displaystyle \sup_{|t-t_0|\leq 1} \|g_{k\ell} (t)\|_{C^2(\overline{\Omega(t_0)})}<M$ for $k,\ell=1,2,3$,
with some $M>0$, by Proposition \ref{prop;trace},
\eqref{eq;nu0}, \eqref{eq;nu1} and \eqref{eq;nu2}, we have
\begin{equation*}
    \begin{split}
        \|J_1\|_{H^{\frac{3}{2}}(\pt \Omega(t_0))}
        &\leq
        C \sum_{k,\ell}\sum_{j=0}^2 \sup_{\pt\Omega(t_0)} 
        \bigl| \nabla_{\tilde{x}}^j\bigl( \tilde{\nu}^\ell(\tilde{x},t_m) - \nu^\ell(\tilde{x},t_0)\bigr)\bigr|
        \| g_{k\ell} (t_m) \tilde{w}^k_m \|_{H^{\frac{3}{2}}\pt \Omega(t_0)} 
        \\
        &\leq
        CM \sum_{\ell}\sum_{j=0}^2 \sup_{\pt\Omega(t_0)} 
        \bigl| \nabla_{\tilde{x}}^j\bigl( \tilde{\nu}^\ell(\tilde{x},t_m) - \nu^\ell(\tilde{x},t_0)\bigr)\bigr|
        \| \tilde{w}_m \|_{H^2(\Omega(t_0))}
        \\
        &\to 0 \quad \text{as } t_m \to t_0.
        \end{split}
\end{equation*}
Moreover, we have
\begin{equation*}
    \begin{split}
        \|J_2\|_{H^{\frac{3}{2}}(\pt \Omega(t_0))} 
        &\leq
        C\sum_{k,\ell} \sum_{j=0}^2
        \bigl| \nabla_{\tilde{x}}^j\nu^\ell(\tilde{x},t_0)\bigr|
        \bigl\| \bigl(
            g_{k\ell}(t_m) - \delta_{k,\ell}
        \bigr) \tilde{w}^k_m\bigr\|_{H^{\frac{3}{2}}(\pt \Omega(t_0))}
        \\
        &\leq
        C\sum_{k,\ell} \sum_{j=0}^2
        \bigl| \nabla_{\tilde{x}}^j\nu^\ell(\tilde{x},t_0)\bigr|
        \| 
            g_{k\ell}(t_m) - \delta_{k,\ell}
        \|_{C^2(\overline{\Omega(t_0)})}
        \| \tilde{w}_m \|_{H^{2}( \Omega(t_0))}
        \\
        &\to 0 \quad \text{as } t_m\to t_0.
    \end{split}
\end{equation*}
Therefore, by \eqref{eq;w contradict} we see that $\|w_m\|_{H^2(\Omega(t_0))} \to 0$ as $t_m\to t_0$.
However, this contradicts to the assumption.
The proof is completed.
\end{proof}
\subsubsection{Uniform estimate of the constant for $t$ in the Helmholtz-Weyl decomposition}
We discuss the (domain) dependence of the constant of the estimates as in
the Helmholtz-Weyl decomposition. 
Similar to the argument for harmonic vector fields, 
the a priori estimate \eqref{eq;aprioriLtw} plays an important role in the analysis of the 
time dependence of $w(t)$.
However, it contains the lower order term of the solution of \eqref{eq;L(t)w(t)},
in contrast to the case of harmonic vector fields.
Therefore, we may not investigate the time continuity and differentiability of ${w}(t)$ 
only by $L(t)$, $B_1$ and $B_2$ as in \eqref{eq;aprioriLtw}.
Then, we additionally consider the uniform estimate of 
the constant for $t$ in the Helmholtz-Weyl decomposition as in Proposition \ref{prop;KY}.

Introducing a Banach space $H_{\mathrm{div}}^1(\Omega) :=\{ u \in H^1(\Omega)\,;\, \Div\,u =0 \text{ in }\Omega\}$,
due to Proposition \ref{prop;KY},
we define a bounded linear operator 
\begin{equation*}
    f_\Omega : H^1_{\mathrm{div}} (\Omega) \ni b \mapsto f_{\Omega}[b]=w \in H^2(\Omega) \cap Z^2_\sigma(\Omega)
    \quad\text{with $w$ is in \eqref{eq;KY dec}.}
\end{equation*}
Here, we put 
\begin{equation*}
    C(\Omega) := \sup_{b \in H^1_{\Div}(\Omega)\setminus \{0\}}
    \frac{\| f_\Omega[b]\|_{H^2(\Omega)}}{\| b \|_{H^1(\Omega)}} < \infty.
\end{equation*}

Now, we investigate the boundedness of $C\bigl(\Omega(t)\bigr)$ for $t\in [0,T]$.
\begin{theorem}\label{thm;const KY dec}
    Let $T>0$. Then it holds that
    \begin{equation}\label{eq;bound COmega}
        \sup_{0\leq t \leq T} C\bigl( \Omega(t)\bigr) < \infty.
    \end{equation}
\end{theorem}
\begin{proof}
    We prove this theorem by a contradiction argument.
    Let us assume $\sup\limits_{0\leq t \leq T} C\bigl(\Omega(t)\bigr)=\infty$.
    Hence, for each $m \in \N$ there exits $t_m \in [0,T]$
    such that $C\bigl(\Omega(t_m)\bigr)>m$. 
    By  the Bolzano-Weierstrass theorem, taking a subsequence of $\{t_m\}_{m\in\N}$ if necessary,
    we may also assume $t_m \to t_0$ as $m\to\infty$ with some $t_0 \in [0,T$].
    Moreover, for each $t_m$ we can choose $b(t_m) \in H_{\Div}^1\bigl(\Omega(t_m)\bigr)$
    such that
    \begin{gather}\label{eq;bm}
        \|b(t_m) \|_{H^1(\Omega(t_m))} =\frac{1}{C\bigl(\Omega(t_m)\bigr)} <\frac{1}{m},
        \\
        \label{eq;CwC}
        \frac{1}{2} < 
        \bigl\| f_{\Omega(t_m)} \bigl[b(t_m)\bigr] \bigr\|_{H^2(\Omega(t_m))} \leq 1
    \end{gather}
    for all $m\in \N$.

    For $t_0\in [0,T]$, introducing the transformation $u\text{ on }\Omega(t)
    \longmapsto \tilde{u} \text{ on }\Omega(t_0)$ as in \eqref{eq;identify tilde}
    with the diffeomorphism $\varphi(\cdot,t)$ defined by \eqref{eq;diffeo t0}, 
    we see that by Proposition \ref{prop;funcsp} and \eqref{eq;bm}
    \begin{equation*}
        \|\tilde{b}(t_m) \|_{H^1(\Omega(t_0)) }
        \leq
        C \| b(t_m) \|_{H^1(\Omega(t_m))} \to 0 \quad \text{as } t_m \to t_0.
    \end{equation*}
    Moreover, as in \eqref{eq;represent h}, the harmonic vector filed $h(t_m)$ on $\Omega(t_m)$ as 
    in the Helmholtz-Weyl decomposition of $b(t_m)$ is expressed as
    \begin{equation*}
        h(t_m)=\sum_{k=1}^K \bigl( b(t_m), \eta_k(t_m)\bigr) \eta_k(t_m)
        \quad \text{on } \Omega(t_m),
    \end{equation*}
    where $\eta_1(t_m),\dots,\eta_K(t_m)$ are orthonormal basis of $V_{\mathrm{har}}\bigl(\Omega(t_m)\bigr)$
    defined by \eqref{eq;cons Vhar}.
    So, by Lemma \ref{lem;alpha} we see that
    \begin{equation*}
        \begin{split}
        \| \tilde{h}(t_m) \|_{H^1(\Omega(t_0))}
        &\leq 
        C \| h (t_m) \|_{H^1(\Omega(t_m))}
        \leq
        C\sum_{k=1}
         \bigl|\bigl(b(t_m),\eta_k(t_m)\bigr)\bigr|
        \| \eta_k(t_m) \|_{H^1(\Omega(t_m))}
        \\
       & \leq C
        \|b(t_m)\|_{L^2(\Omega(t_m))} 
        \sum_{k} \sup_{0\leq t \leq T} \| \widetilde{\eta}_k(t) \|_{H^1(\widetilde{\Omega})}
        \\
        & \to 0 \quad \text{as } t_m \to t_0.
        \end{split}
    \end{equation*}
    Then, let $w(t_m)=f_{\Omega(t_m)}[b(t_m)] \in H^2\bigl(\Omega(t_m)\bigr) \cap Z_\sigma^2\bigl(\Omega(t_m)\bigr)$.
    Since $\tilde{h}(t_m) + \Rot(t_m)\,\tilde{w}(t_m) = \tilde{b}(t_m)$ on $\Omega(t_0)$ for $m \in \N$,
     we have
    \begin{equation}\label{eq;Rot tw t0}
        \| \Rot(t_m)\,\tilde{w}(t_m) \|_{H^{1}(\Omega(t_0))} 
        =
        \| \tilde{b}(t_m) - \tilde{h}(t_m) \|_{H^1(\Omega(t_0))}
        \to 0
        \quad \text{as } t_m \to t_0.
    \end{equation}

    On the other hand, from \eqref{eq;CwC} and Proposition \ref{prop;funcsp}, there exist constant $C_1, C_2>0$ such that
    \begin{equation}\label{eq;CtwC}
        C_1 < \| \tilde{w}(t_m) \|_{H^2(\Omega(t_0))} < C_2 \quad \text{for }m\in\N.
    \end{equation}
    Since $H^2\bigl(\Omega(t_0)\bigr)$ is compactly embedded in $H^1(\Omega(t_0))$, we may assume, taking a subsequence if necessary,
    that there exists $w_0 \in H^2\bigl(\Omega(t_0)\bigr)$ such that 
    \begin{equation}\label{eq;tw conv w0}
        \begin{cases}
            \tilde{w}(t_m) \rightharpoonup w_0  &\text{weakly in } H^2\bigl(\Omega(t_0)\bigr) \text{ as } t_m \to t_0,
            \\
            \tilde{w}(t_m) \to w_0 &\text{strongly in } H^1\bigl(\Omega(t_0)\bigr) \text{ as } t_m \to t_0. 
        \end{cases}
    \end{equation}
    To begin with, we shall show $w_0 \in X^2_\sigma\bigl(\Omega(t_0)\bigr) \cap H^2\bigl(\Omega(t_0)\bigr)$.
    Firstly by Proposition \ref{prop;IW div}, we have $0=\Div\, w(t_m)=\Div_{\tilde{x}} \,\tilde{w}(t_m)$.
    Hence, $\Div_{\tilde{x}} w_0=0$.
    To confirm the boundary condition $w_0\cdot \nu(t_0)=0$, where $\nu(t_0)$ is the outward unit normal vector on $\pt\Omega(t_0)$,
    taking an arbitrary test function $\tilde{\psi} \in C^\infty(\overline{\Omega(t_0)})$, we see that
    \begin{equation*}
        \begin{split}
            \bigl\langle w_0 \cdot \nu(t_0),\tilde{\psi}\bigr\rangle_{\pt\Omega(t_0)}
            &=
            (w_0,\nabla_{\tilde{x}}\tilde{\psi}) + (\Div_{\tilde{x}}\,w_0,\tilde{\psi})
            \\
            &=\bigl(w_0 - \tilde{w}(t_m),\nabla_{\tilde{x}}\tilde{\psi}\bigr)
            +
            \bigl(\tilde{w}(t_m), \nabla_{\tilde{x}}\tilde{\psi}\bigr).
        \end{split}
    \end{equation*}
    Here, we have
    \begin{equation*}
        \bigl|\bigl(w_0 - \tilde{w}(t_m),\nabla_{\tilde{x}}\tilde{\psi}\bigr)
        \bigr|
        \leq \| w_0 - \tilde{w}(t_m)\|_{L^2(\Omega(t_0))}\|\nabla_{\tilde{x}} \tilde{\psi}\|_{L^2(\Omega(t_0))} \to 0 
        \quad \text{ as } t_m \to t_0,
    \end{equation*}
    and noting $\displaystyle
    \psi(x)=\tilde{\psi}(\vphi(x,t_m))$ for $x \in \Omega(t_m)$, since $w(t_m) \in X_\sigma^2\bigl(\Omega(t_m)\bigr)
    \subset L^2_\sigma\bigl(\Omega(t_m)\bigr)$ we have by changing variables $x=\varphi^{-1} (\tilde{x},t_m)$,
    \begin{equation}\label{eq;twntpsi}
        \begin{split}
            0=\int_{\Omega(t_m)} w(x,t_m)\cdot \nabla \psi(x)\,dx
            =
            \int_{\Omega(t_0)} \tilde{w}(\tilde{x},t_m)
             \cdot \nabla_{\tilde{x}} \tilde{\psi} (\tilde{x})\,d\tilde{x}\, J(t_m)
             \quad\text{for } m\in \N.
        \end{split}
    \end{equation}
    Therefore, we obtain that $w_0\cdot \nu(t_0)=0$ in the weak sense, 
    hence $w_0 \in X_\sigma^2\bigl(\Omega(t_0)\bigr)$.

    Next, we observe that $\Rot(t_m) \tilde{w}(t_m) \to \rot_{\tilde{x}}\,w_0$ in $L^2\bigl(\Omega(t_0)\bigr)$ as
    $t_m \to t_0$.
    Recalling \eqref{eq;rot Rot}, we see that for $i=1,2,3$,
    \begin{equation}\label{eq;Rottm-rot}
        \begin{split}
            [\Rot(t_m)\, \tilde{w}(t_m) - \rot_{\tilde{x}}\,w_0]^i
            =\,&
            \sum_{j,k,\ell} R^{i,1}_{j,k,\ell}(t_m) \,\tilde{w}(t_m) 
            +
            \sum_{j,k,\ell} 
            \bigl(
            R^{i,2}_{j,k,\ell}(t_m) - R^{i,2}_{j,k,\ell}(t_0)
            \bigr)
            \frac{\pt\tilde{w}^k}{\pt \tilde{x}^\ell}(t_m)
            \\
            &+ \bigl[
                \rot_{\tilde{x}}\,\bigl(\tilde{w}(t_m) - w_0\bigr)
            \bigr]^i.
        \end{split}
    \end{equation}
    By \eqref{eq;R^1}, \eqref{eq;R^2}, \eqref{eq;CtwC} and \eqref{eq;tw conv w0}, 
    we have the desired convergence and combining \eqref{eq;Rot tw t0} we have
    \begin{equation*}
        \rot_{\tilde{x}}\, w_0= \lim_{t_m\to t_0} \Rot(t_m) \, \tilde{w}(t_m) =0.
    \end{equation*}
    So, by \cite[(1) of Theorem 2.1]{Kozono Yanagisawa IUMJ} we obtain $w_0 \in X_{\mathrm{har}}\bigl(\Omega(t_0)\bigr)$.

    However, we claim that $w_0 \perp X_{\textrm{har}}\bigl(\Omega(t_0)\bigr)$. 
    To confirm this claim, we introduce the orthogonal basis of $X_{\mathrm{har}}\bigl(\Omega(t)\bigr)$, $\zeta_1(t),\dots,\zeta_L(t)$ with
    $\| \zeta_\ell (t) \|_{H^2(\Omega(t))}=1$, $\ell=1,\dots, L$, 
    in $L^2$-sense,
    for each $t \in \R$, since $\dim X_{\mathrm{har}}\bigl(\Omega(t)\bigr)=L$ 
    for all $t\in\R$ 
    under our assumption \eqref{eq;KL}, 
    for more details, see \cite{Kozono Yanagisawa IUMJ}.
    Moreover, we introduce the transformed vector field 
    $\tilde{\zeta}_\ell(t)$ on $\Omega(t_0)$.
    Then, since 
    $C\leq \| \tilde{\zeta}_\ell(t_m)\|_{H^2(\Omega(t_0))} \leq C^\prime$ 
    for $\ell=1,\dots,L$, with some constant $C,C^\prime>0$
    by Proposition \ref{prop;funcsp}, 
    for each $\ell=1,\dots,L$, there exists a nontrivial $\zeta_{\ell,0} \in H^2\bigl(\Omega(t_0)\bigr)$
    such that 
    \begin{equation}\label{eq;tzeta conv zeta0}
        \begin{cases}
            \tilde{\zeta}_\ell(t_m) \rightharpoonup \zeta_{\ell, 0}  &\text{weakly in } H^2(\Omega(t_0)) \text{ as } t_m \to t_0,
            \\
            \tilde{\zeta}_\ell(t_m) \to \zeta_{\ell, 0} &\text{strongly in } H^1(\Omega(t_0)) \text{ as } t_m \to t_0. 
        \end{cases}
    \end{equation}
    Here, $\zeta_{\ell,0}\neq 0$ follows from the fact that
    $\Div_{\tilde{x}}\,\tilde{\zeta}_{\ell}(t_m)=0$, 
    $\rot_{\tilde{x}}\, \tilde{\zeta}_{\ell}(t_m)
    =\bigl(\rot_{\tilde{x}} -\Rot(t_m)\bigr)\tilde{\zeta}_{\ell}(t_m)$ in $\Omega(t_0)$
    and $\tilde{\zeta}_{\ell}(t_m)\cdot\nu(t_0)=0$ in $\pt \Omega(t_0)$
    and the inequality in \cite[(2.8) in Theorem 2.4]{Kozono Yanagisawa IUMJ}:
    \begin{equation*}
        \begin{split}
       0< C &\leq \| \tilde{\zeta}_{\ell}(t_m)\|_{H^2(\Omega(t_0))}\\
        &\leq C_{\Omega(t_0)}\Bigl(
            \bigl\|\bigl(\rot_{\tilde{x}}-\Rot(t_m)\bigr)\tilde{\zeta}_\ell(t_m)\bigr\|_{H^1(\Omega(t_0))}
            + \sum_{\ell^\prime=1}^L 
            \bigl|\bigl(\tilde{\zeta}_{\ell}(t_m), \zeta_{\ell^\prime}(t_0)\bigr)\bigr|
        \Bigr).
        \end{split}
    \end{equation*}
So, letting $t_m\to t_0$, we observe that 
$\sum_{\ell^\prime=1}^L \bigl|\bigl(\zeta_{\ell,0},\zeta_{\ell^\prime}(t_0)\bigr)\bigr|\neq 0$.

    Noting that $\Rot(t_m)\,\tilde{\zeta}_\ell(t_m)=0$ from $\rot\, \zeta_\ell(t_m)=0$, by the same argument as above,
    we have for $\ell=1,\dots,L$,
    \begin{equation*}
        \begin{cases}
            \Div_{\tilde{x}}\, \zeta_{\ell,0}=0 &\text{in }\Omega(t_0),
            \\
            \rot_{\tilde{x}}\, \zeta_{\ell,0}=0 & \text{in } \Omega(t_0),
            \\
            \zeta_{\ell,0} \cdot \nu(t_0)=0 & \text{on } \pt \Omega(t_0).
        \end{cases}
    \end{equation*}
    Therefore, we obtain $\zeta_{\ell,0} \in X_{\textrm{har}}\bigl(\Omega(t_0)\bigr)$, $\ell=1,\dots,L$.
    Furthermore, we shall show $\zeta_{\ell,0} \perp \zeta_{\ell^\prime,0}$ for $\ell \neq \ell^\prime$.
    Indeed, 
    \begin{equation}\label{eq;zeta perp}
        (\zeta_{\ell,0},\zeta_{\ell^\prime,0})
        =
        \bigl(\zeta_{\ell,0} -\tilde{\zeta}_\ell(t_m), \zeta_{\ell^\prime,0}\bigr)
        +
        \bigl(\tilde{\zeta}_{\ell}(t_m), \zeta_{\ell^\prime,0} -\tilde{\zeta}_{\ell^\prime}(t_m)\bigr)
        +
        \bigl(\tilde{\zeta}_\ell(t_m), \tilde{\zeta}_{\ell^\prime}(t_m)\bigr).
    \end{equation}
    Here, to estimate the third term of the right hand side of \eqref{eq;zeta perp}, we note that
    \begin{equation}
        \begin{split}
            0= \bigl( \zeta_\ell (t_m),\zeta_{\ell^\prime}(t_m)\bigr)
            &=\int_{\Omega(t_m)} \sum_{i} \zeta_{\ell}^i(x,t_m)\zeta_{\ell^\prime}^i(x,t_m)\,dx
            \\
            &=\int_{\Omega(t_0)} \sum_{j,k} g_{jk}(\tilde{x},t_m)
            \tilde{\zeta}_\ell^j(\tilde{x},t_m) 
            \tilde{\zeta}_{\ell^\prime}^k (\tilde{x},t_m) J(t_m)\,d\tilde{x}.
        \end{split}
    \end{equation}
    Therefore, we have
    \begin{equation}\label{eq;zetalzetalp}
        \int_{\Omega(t_0)} \sum_{j,k} g_{jk}(\tilde{x},t_m)
            \tilde{\zeta}_\ell^j(\tilde{x},t_m) 
            \tilde{\zeta}_{\ell^\prime}^k (\tilde{x},t_m) \,d\tilde{x}=0.
    \end{equation}
    Hence, using \eqref{eq;zetalzetalp}, we see that
    \begin{equation}
        \begin{split}
        \bigl(\tilde{\zeta}_{\ell}(t_m),\tilde{\zeta}_{\ell^\prime}(t_m)\bigr)
        &=
        \int_{\Omega(t_0)} \sum_{j} \tilde{\zeta}_\ell^j(\tilde{x},t_m) 
        \tilde{\zeta}_{\ell^\prime}^j (\tilde{x},t_m) \,d\tilde{x}
        \\
        &=
        \int_{\Omega(t_0)} \sum_{j,k} 
        \bigl(\delta_{j,k} - g_{jk}(\tilde{x},t_m) \bigr)
        \tilde{\zeta}_\ell^j(\tilde{x},t_m) 
        \tilde{\zeta}_{\ell^\prime}^k (\tilde{x},t_m) \,d\tilde{x}
        \\
        \end{split}
    \end{equation}
    Therefore, by Proposition \ref{prop;gij}, we have
    \begin{equation}\label{eq;tzetaltzetalp}
        \begin{split}
        \bigl| \bigl(\tilde{\zeta}_{\ell}(t_m),\tilde{\zeta}_{\ell^\prime}(t_m)\bigr)\bigr|
        &\leq
        C \sum_{j,k}\sup_{\Omega(t_0)}
        \bigl|\delta_{j,k} - g_{jk}(\tilde{x},t_m)\bigr|
        \| \tilde{\zeta}_\ell(t_m) \|_{L^2(\Omega(t_0))}
        \| \tilde{\zeta}_{\ell^\prime}(t_m)\|_{L^2(\Omega(t_0))}
        \\
        & \to 0 \quad \text{as } t_m \to t_0.
        \end{split}
    \end{equation}
    Thus, by \eqref{eq;zeta perp}, we obtain $\zeta_{\ell,0} \perp \zeta_{\ell^\prime,0}$.
    Hence, $\zeta_{1,0},\dots,\zeta_{L,0}$ is a orthogonal basis of $X_{\textrm{har}}\bigl(\Omega(t_0)\bigr)$.
    
    Replacing $\zeta_{\ell,0}$ and $\tilde{\zeta}_{\ell}(t_m)$ from \eqref{eq;zeta perp}
    to \eqref{eq;tzetaltzetalp} by
    $w_0$ and $\tilde{w}(t_m) \in Z_\sigma^2\bigl(\Omega(t_m)\bigr)$, respectively, we also obtain that
    \begin{equation}\label{eq;w0perpXhar}
        (w_0,\zeta_{\ell,0})=0, \quad \ell=1,\dots,L.
    \end{equation}
   As a conclusion, $w_0 \in X_{\mathrm{har}}\bigl(\Omega(t_0)\bigr)^{\perp} \cap X_{\mathrm{har}}\bigl(\Omega(t_0)\bigr)$,
   hence, $w_0=0$. 
   Now, by \eqref{eq;tw conv w0} we emphasize that 
    \begin{equation*}
        \| \tilde{w}(t_m) \|_{H^1(\Omega(t_0))} =\| \tilde{w}(t_m) - w_0 \|_{H^1(\Omega(t_0))}
        \to 0 \quad \text{as } t_m \to t_0.
    \end{equation*}

    We recall that the vector potential $\tilde{w}(t_m)$ satisfies \eqref{eq;L(t)w(t)} for $t_m$. 
    So by Lemma \ref{lem;apriori w} we see that
    \begin{equation*}
        \| \tilde{w}(t_m) \|_{H^2(\Omega(t_0))}
        \leq
        C\left(
            \|\Rot (t_m)\, \tilde{b}(t_m)\|_{L^2(\Omega(t_0))}
            +
            \bigl\| B_1\bigl[\tilde{b}(t_m),\tilde{\nu}(t_m)\bigr] \bigr\|_{H^\frac{1}{2}(\Omega(t_0))}
            +
            \|\tilde{w}(t_m)\|_{L^2(\Omega(t_0))}
        \right)
    \end{equation*}
    for sufficiently large $m\in\N$.
    By the similar manner in the previous Section \ref{sec;Apriori w}, 
    we recall that
    \begin{gather*}
        B_1[\tilde{b}(t_m),\tilde{\nu}(t_m)]^i=
        \sum_{j,k,\ell} S^{i}_{j,k,\ell}(t_m) \bigl(
            \tilde{b}^k(t_m) \tilde{\nu}^\ell(t_m) - \tilde{b}^\ell(t_m) \tilde{\nu}^k(t_m)
        \bigr),
        \\
        \sup_{t\in[0,T]} \| S^i_{j,k,\ell}(t)\|_{C^2(\overline{\Omega(t_0)})} < M,
        \quad i,j,k,\ell=1,2,3,
        \\
       \sum_{n=0}^2  \sup_{\pt \Omega(t_0)} |\nabla_{\tilde{x}}^n \tilde{\nu}^\ell(\tilde{x},t_m) |<M,
       \quad m\in\N,\;\ell=1,2,3.
    \end{gather*}
    Like a proof of Proposition \ref{prop;trace}, we can directly derive that for $i,j,k,\ell=1,2,3$,
    \begin{equation*}
        \begin{split}
        \| S^i_{j,k,\ell} (t_m) \tilde{b}^k(t_m) \tilde{\nu}^\ell(t_m)
        \|_{H^\frac{1}{2}(\pt\Omega(t_0))}
        &\leq
        C \sum_{n=0}^2 \sup_{\pt\Omega(t_0)}|\nabla_{\tilde{x}}^n \tilde{\nu}^\ell(t_m)|
        \| S^{i}_{j,k,\ell}(t_m) \tilde{b}^k(t_m) \|_{H^\frac{1}{2}(\pt \Omega(t_0))}
        \\
        &\leq
        C \sum_{n=0}^2 \sup_{\pt\Omega(t_0)}|\nabla_{\tilde{x}}^n \tilde{\nu}^\ell(t_m)|
        \| S^{i}_{j,k,\ell}(t_m) \tilde{b}^k(t_m) \|_{H^1( \Omega(t_0))}
        \\
        &\leq
        CM^2 
        \| \tilde{b}(t_m) \|_{H^1( \Omega(t_0))} 
        \\
        &\to 0 \quad\text{as } t_m \to t_0.
        \end{split}
    \end{equation*}
    Therefore, $\|B_1[\tilde{b}(t_m),\tilde{\nu}(t_m)]\|_{H^{\frac{1}{2}}(\pt\Omega(t_0))} \to 0$ 
    as $t_m \to t_0$. Since it is easy to see that 
    \begin{equation*} \| \Rot(t_m)\,\tilde{b}(t_m) \|_{L^2(\Omega(t_0))} \to 0
        \quad \text{as } t_m \to t_0.
    \end{equation*}
    we obtain that
    \begin{equation*}
        \| \tilde{w}(t_m) \|_{H^2(\Omega(t_m))} \to 0 \quad \text{as } t_m \to t_0.
    \end{equation*}
    However, this contradicts \eqref{eq;CtwC}. The proof is completed.
\end{proof}
%
\subsubsection{Time continuity of the vector potentials}
By the virtue of the a priori estimate \eqref{eq;apriori w} and 
the uniform estimate of the constant of the Helmholtz-Weyl decomposition \eqref{eq;bound COmega},
we establish the time continuity of the vector potential as follows.


%
\begin{theorem}\label{thm;conti w}
    Let $T>0$ and let $b(t) \in H^1\bigl(\Omega(t)\bigr)$ with $\Div\, b(t)=0$ in $\Omega(t)$ for $t\in \R$
    and let $w(t) \in Z_\sigma^{2}\bigl(\Omega(t)\bigr)\cap H^2\bigl(\Omega(t)\bigr)$
    and $h(t) \in V_{\mathrm{har}}\bigl(\Omega(t)\bigr)$ satisfy
    $h(t) + \rot\,w(t)=b(t)$ in $\Omega(t)$.
    if $\widetilde{b} \in C\bigl([0,T];H^1(\widetilde{\Omega})\bigr)$ then
    \begin{equation}
        \widetilde{w} \in C\bigl([0,T]; H^2(\widetilde{\Omega})\bigr).
    \end{equation}
\end{theorem}
For the analysis of the time continuity at $t_0\in[0,T]$, 
we consider the transformed vector $\tilde{w}(t)$ on $\Omega(t_0)$
defined by \eqref{eq;identify tilde} with the diffeomorphism 
$\varphi(\cdot,t):\overline{\Omega(t)}\to \overline{\Omega(t_0)}$ 
defined by \eqref{eq;diffeo t0}.
The following lemma plays an essential role to prove Theorem \ref{thm;conti w}.
\begin{lemma}\label{lem;conti w local}
    Under the assumption of Theorem \ref{thm;conti w}, for each $t_0\in [0,T]$,
    for every $\ep>0$ there exists $\delta>0$ such that
    for $t\in [0,T]$ if $|t-t_0|<\delta$ then
    \begin{equation*}
        \|\tilde{w}(t) -w(t_0) \|_{H^2(\Omega(t_0))} <\ep.
    \end{equation*}
\end{lemma}
\begin{proof}
    We prove this lemma by contradiction argument.
    We assume that there exist $t_0 \in [0,T]$ and $\ep_0$ such that
    for every $m\in \N$ there exists $t_m\in[0,T]$ such that
    $|t_m-t_0|<\frac{1}{m}$ and $\| \tilde{w}(t_m) -w(t_0) \|_{H^2(\Omega(t_0))}\geq \ep_0$.

    Here, we note that $\widetilde{b} \in C\bigl([0,T];H^1(\widetilde{\Omega})\bigr)$ 
    implies $\tilde{b}(t_m) \to b(t_0)$ in $H^1\bigl(\Omega(t_0)\bigr)$ as $t_m\to t_0$.
    Indeed, by the relation $\widetilde{\tilde{b}}(t)=\widetilde{b}(t)$ in $\widetilde{\Omega}$
    as in \eqref{eq;tildewidetilde} and Proposition \ref{prop;funcsp}, we see that
    \begin{equation*}
        \begin{split}
            \| \tilde{b}(t_m) - b(t_0) \|_{H^1(\Omega(t_0))}
            &\leq
            C\| \widetilde{\tilde{b}}(t_m) - \widetilde{b}(t) \|_{H^1(\widetilde{\Omega})}
            =C \| \widetilde{b}(t_m) - \widetilde{b}(t_0) \|_{H^1(\widetilde{\Omega})}
            \\
            &\to 0 \quad \text{as } t_m \to t_0,
        \end{split}
    \end{equation*}
    where the constant $C>0$ is independent of $t\in [0,T]$.
Since $\widetilde{h} \in C\bigl([0,T];H^1(\widetilde{\Omega})\bigr)$ from Theorem \ref{thm;contideffi har}, 
by the same way we note that $\tilde{h}(t_m)\to h(t_0)$ in $H^1\bigl(\Omega(t_0)\bigr)$ as $t_m\to t_0$.

    Therefore,  letting $t_m \to t_0$, we see that
    \begin{equation}\label{eq;Rottw rot wt0}
        \begin{split}
        \Rot(t_m)\,\tilde{w}(t_m)&=\tilde{b}(t_m)-\tilde{h}(t_m)
            \\    
            & \to b(t_0) - h(t_0)
            = \rot_{\tilde{x}}\, w(t_0).
        \end{split}
    \end{equation}

    On the other hand, by Theorem \ref{thm;const KY dec} and Proposition \ref{prop;funcsp},
    we see that
    \begin{equation*}
        \| \tilde{w}(t_m)\|_{H^2(\Omega(t_0))}
        \leq
        C\| w(t_m) \|_{H^2(\Omega(t_m))}
        \leq 
        C \|b(t_m)\|_{H^1(\Omega(t_m))} 
        \leq
        C \sup_{0\leq t \leq T} \| \widetilde{b}(t) \|_{H^1(\widetilde{\Omega})}
    \end{equation*}
    for sufficiently large $m\in\N$, where the constant $C>0$ is independent of $m$.
    Therefore, taking a subsequence of $\{t_m\}$ if necessary, 
    we may take $w_\infty \in H^2\bigl(\Omega(t_0)\bigr)$
    such that
    \begin{equation}
        \begin{cases}
            \tilde{w}(t_m) \rightharpoonup w_\infty 
            &\text{weakly in } H^2(\Omega(t_0)) \quad\text{as } t_m \to t_0,
            \\
            \tilde{w}(t_m) \to w_\infty
            &\text{strongly in } H^1\bigl(\Omega(t_0)\bigr)
            \quad \text{as } t_m \to t_0.
        \end{cases}
    \end{equation}
    Then, by the same argument as \eqref{eq;Rottm-rot}, we see that
    \begin{equation}\label{eq;Rottw rot winf}
        \Rot (t_m)\, \tilde{w}(t_m) \to \rot_{\tilde{x}}\,w_\infty 
        \; \text{in } L^2\bigl(\Omega(t_0)\bigr)
        \quad\text{as } t_m \to t_0.
    \end{equation}
    Moreover, since $\Div_{\tilde{x}}\,\tilde{w}(t_m) =0$, 
    we have $\Div_{\tilde{x}}\,w_\infty=0$. 
    Since $\bigl(\tilde{w}(t_m),\nabla_{\tilde{x}} \tilde{\psi}\bigr)=0$ for all 
    test function $\tilde{\psi}\in C^\infty(\overline{\Omega(t_0)})$ 
    by the argument in \eqref{eq;twntpsi}, we also see that 
    \begin{equation*}
        \langle w_\infty \cdot \nu(t_0), \tilde{\psi} \rangle_{\pt \Omega(t_0)}
        = \bigl(w_\infty - \tilde{w}(t_m), \nabla_{\tilde{x}} \tilde{\psi} \bigr) \to 0 \quad \text{as } t_m \to t_0.
    \end{equation*}
    Hence $w_\infty \in X_\sigma^2\bigl(\Omega(t_0)\bigr)$. Furthermore, by the same argument to obtain \eqref{eq;w0perpXhar},
    we see that $w_\infty \perp X_{\mathrm{har}}\bigl(\Omega(t_0)\bigr)$. 
    Therefore, $w_\infty \in Z^2_{\sigma}\bigl(\Omega(t_0)\bigr)$.
    Hence, by \eqref{eq;Rottw rot wt0}, \eqref{eq;Rottw rot winf} and by the uniqueness of the decomposition, we note that 
    \begin{equation*}
        w(t_0) = w_\infty. 
    \end{equation*}
    Therefore, we obtain that 
    \begin{equation}\label{eq;wtmtowt0}
        \| \tilde{w}(t_m) - w(t_0) \|_{H^1(\Omega(t_0))} \to 0
        \quad\text{as }
    t_m \to t_0.
    \end{equation}

    Then, by \eqref{eq;apriori w}, we have that
    \begin{multline}\label{eq;apriori w at t0}
            \| \tilde{w}(t_m) -w(t_0) \|_{H^2(\Omega(t_0))}
            \leq 
            C_*
            \Bigl(
                \bigl\| \Delta_{\tilde{x}} \bigl( 
                    \tilde{w}(t_m) - w (t_0)
                \bigr)
                \bigr\|_{L^2(\Omega(t_0))}
                \\
                +
                \bigl\| \rot_{\tilde{x}}\,\bigl( \tilde{w}(t_m) - w(t_0)\bigr) \times \nu(t_0)
                \bigr\|_{H^{\frac{1}{2}}(\pt\Omega(t_0))}
                +
                \bigl\|
                    \bigl( \tilde{w}(t_m) - w(t_0)\bigr) \cdot \nu(t_0)
                \bigr\|_{H^{\frac{3}{2}}(\pt\Omega(t_0))}
               \\ +
                \|  \tilde{w}(t_m) - w(t_0) \|_{L^2(\Omega(t_0))}
            \Bigr).
    \end{multline}
    Here, we note that, recalling $\tilde{w}(t_m)$ satisfies \eqref{eq;L(t)w(t)} for $t_m$,
    \begin{equation*}
        \begin{split}
            \Delta_{\tilde{x}} \bigl( 
                \tilde{w}(t_m) - w (t_0)
            \bigr)
            &=
            \Delta_{\tilde{x}} \tilde{w}(t_m) + \rot_{\tilde{x}}\, b(t_0) 
            \\
            &=
            \bigl( \Delta_{\tilde{x}} -L(t_m)\bigr) \tilde{w}(t_m)
            - \Rot(t_m)\,\tilde{b}(t_m) +\rot_{\tilde{x}}\, b(t_0).
        \end{split}
    \end{equation*}
    Since $\| \tilde{w}(t_m)\|_{H^2(\Omega(t_0))}<C$ for sufficiently large $m\in\N$,
    by the same argument as \eqref{eq;L(t)-Delta},
    \begin{equation*}
        \bigl\| \bigl( \Delta_{\tilde{x}} -L(t_m)\bigr) \tilde{w}(t_m)
        \bigr\|_{L^2(\Omega(t_0))}\to 0
        \quad \text{as } t_m \to t_0.
    \end{equation*}
    And, by the same way as \eqref{eq;Rottm-rot}, we have that
    \begin{equation*}
        \| \Rot(t_m)\,\tilde{b}(t_m) - \rot_{\tilde{x}}\, b(t_0)
        \|_{L^2(\Omega(t_0))} \to 0
        \quad \text{as } t_m \to t_0.
    \end{equation*}
    Next, we note that
    \begin{equation*}
        \begin{split}
            \rot_{\tilde{x}}\,&\bigl( \tilde{w}(t_m) - w(t_0)\bigr) \times \nu(t_0)
            =
            \rot_{\tilde{x}}\, \tilde{w}(t_m)\times \nu(t_0) - b(t_0) \times \nu(t_0)
            \\
            &=
            \rot_{\tilde{x}}\, \tilde{w} (t_m) \times \nu(t_0) -
            B_1[\Rot(t_m)\,\tilde{w}(t_m),\tilde{\nu}(t_m)]
            + B_1[\tilde{b}(t_m),\tilde{\nu}(t_m)] - b(t_0)\times \nu(t_0).
        \end{split}
    \end{equation*}
    By the same argument of the convergence of each term  in \eqref{eq;dec B_1 times},
    we get
    \begin{equation*}
        \bigl\| \rot_{\tilde{x}}\, \tilde{w} (t_m) \times \nu(t_0) -
        B_1[\Rot(t_m)\,\tilde{w}(t_m),\tilde{\nu}(t_m)]
        \bigr\|_{H^{\frac{1}{2}}(\pt \Omega(t_0))} 
        \to 0
        \quad \text{as } t_m \to t_0.
    \end{equation*}
    Moreover, similar to \eqref{eq;dec B_1 times}, we see that for $i=1,2,3$
    \begin{equation*}
        \begin{split}
            B_1&[\tilde{b}(t_m),\tilde{\nu}(t_m)]^i -[b(t_0) \times \nu(t_0)]^i 
            \\
            =\,&
            \sum_{j,k,\ell}
            S^{i}_{j,k,\ell} (t_m)\Bigl(
                \tilde{b}^k(t_m) \bigl(\tilde{\nu}^\ell(t_m)-\nu^\ell(t_0)\bigr)
                -
                \tilde{b}^\ell(t_m)\bigl(\tilde{\nu}^k(t_m)-\nu^k(t_0)\bigr)
            \Bigr)
            \\
            &+
            \sum_{j,k,\ell} 
            S^{i}_{j,k,\ell} (t_m)\Bigl(
                \bigl(\tilde{b}^k(t_m) - b^k(t_0)\bigr) \nu^\ell(t_0)
                -
                \bigl(\tilde{b}^\ell(t_m) - b^\ell(t_0)\bigr) \nu^k(t_0)
            \Bigr)
            \\
            &+
            \sum_{j,k,\ell} 
            \Bigl(S^{i}_{j,k,\ell}(t_m)-\delta_{i,j}\delta_{\sigma(j+1),k} \delta_{\sigma(j+2),\ell}\Bigr)
            \Bigl(
                b^k(t_0) \nu^\ell(t_0) - b^\ell(t_0) \nu^k(t_0)
            \Bigr) 
            \\
            =:\, & I_1^\prime + I_2^\prime + I_3^\prime.
        \end{split}
    \end{equation*}
    The convergences of $I_1^\prime$, $I_2^\prime$ and $I_3^\prime$ are derived by an analogy of 
    those of $I_1$, $I_2$ and $I_3$ in \eqref{eq;dec B_1 times}.
    Therefore, we see that
    \begin{equation*}
        \bigl\|\rot_{\tilde{x}}\,\bigl( \tilde{w}(t_m) - w(t_0)\bigr) \times \nu(t_0)
        \bigr\|_{H^{\frac{1}{2}}(\pt\Omega)} \to 0
        \quad \text{as } t_m \to t_0.
    \end{equation*}

    Finally, we note that 
    \begin{equation*}
        \begin{split}
        \bigl( \tilde{w}(t_m) - w(t_0)\bigr) \cdot \nu(t_0)
        = \tilde{w}(t_m) \cdot \nu(t_0)
        = \tilde{w}(t_m) \cdot \nu(t_0) -B_2[\tilde{w}(t_m),\tilde{\nu}(t_m)].
        \end{split}
    \end{equation*}
    So by the same argument of  the convergence of each term  in \eqref{eq;dec B2},
    we have 
    \begin{equation*}
        \bigl\|\bigl( \tilde{w}(t_m) - w(t_0)\bigr) \cdot \nu(t_0) 
        \bigr\|_{H^{\frac{3}{2}}(\pt \Omega(t_0))}
        \to 0
        \quad\text{as } t_m \to t_0.
    \end{equation*}

    Since we have \eqref{eq;wtmtowt0}, by \eqref{eq;apriori w at t0} we have
    \begin{equation*}
        \| \tilde{w}(t_m) - w(t_0) \|_{H^2(\Omega(t_0))}
        \to 0
        \quad \text{as } t_m \to t_0.
    \end{equation*}
    However, this contradicts the assumption $\|\tilde{w}(t_m) -w(t_0)\|_{H^2(\Omega(t_0))}\geq \ep_0$.
    The proof is completed.
\end{proof}
Finally, we give the proof of Theorem \ref{thm;conti w}.
\begin{proof}[Proof of Theorem \ref{thm;conti w}]
    At first, we recall the relation $\widetilde{\tilde{w}}(t) = \widetilde{w}(t)$ on $\widetilde{\Omega}$ as in
    \eqref{eq;tildewidetilde}.
    Therefore, by Lemma \ref{lem;conti w local} and by Proposition \ref{prop;funcsp} we see that
    \begin{equation*}
        \| \widetilde{w}(t) - \widetilde{w}(t_0) \|_{H^2(\widetilde{\Omega})}
        =
        \| \widetilde{\tilde{w}}(t) - \widetilde{w}(t_0)\|_{H^2(\widetilde{\Omega})}
        \leq C
        \| \tilde{w}(t) - w(t_0) \|_{H^2(\Omega(t_0))}
        \to 0,
    \end{equation*}
    as $t\to t_0$, where the constnat $C>$ is independent of $t \in [0,T]$.
\end{proof}

\subsubsection{Time differentiability of the vector potentials}

To investigate the time differentiability of $\widetilde{w}(t)$ at $t_0 \in [0,T]$,
we also make use of the transformation: 
$w(t) \text{ on }\Omega(t) \longmapsto \tilde{w}(t) \text{ on } \Omega(t_0)$.
Different from the case of time continuity, the difficulty arises that
Theorem \ref{thm;const KY dec} is not enough to dominate the lower order term
of $\bigl(\tilde{w}(t)-w(t_0)\bigr)/(t-t_0)$ in the a priori estimate.
Hence, we need more qualitative information of the vector potentials in 
$w(t) \in Z^2_{\sigma}\bigl(\Omega(t)\bigr)$.
Indeed, the important question is  whether the coordinate transform \eqref{eq;identify tilde}
preserves the orthogonality to $X_{\mathrm{har}}\bigl(\Omega(t)\bigr)$ or not.


%


To consider the question, let us introduce a specific basis of $X_{\mathrm{har}}\bigl(\Omega(t)\bigr)$
constructed by Foia\c{s} and Temam \cite{Foias Temam}.
Let $p_\ell(t) \in C^\infty\bigl(\dot{\Omega}(t)\bigr)$, $\ell=1,\dots,L$ satisfy
\begin{equation}\label{eq;basis of Xhar}
    \begin{cases}
        \Delta p_\ell(t) =0 &\text{in }\dot{\Omega}(t),
        \smallskip\\
        \displaystyle
        \frac{\pt p_\ell(t)}{\pt \nu(t)} =0 &\text{on }\pt \Omega(t), 
        \smallskip\\
        \displaystyle
        \left[ \frac{\pt p_\ell (t)}{\pt \nu_j(t)}\right]_{\Sigma_{j}(t)}
        =0, \quad [p_\ell(t)]_{\Sigma_j(t)}=\delta_{\ell,j} 
        &\text{for } j=1,\dots,L,
    \end{cases}
\end{equation}
where $[f]_{\Sigma_j(t)}$ denotes the jump of the value $f$ on $\Sigma_j(t)$ defined by
\begin{equation*}
    [f]_{\Sigma_j(t)} :=f|_{\Sigma^+_j(t)} - f|_{\Sigma_j^-(t)},
\end{equation*}
 $\Sigma_j^+(t)$ and $\Sigma_j^-(t)$ denote two sides of $\Sigma_j(t)$,
and where $\nu_j(t)$ is the unit normal vector to $\Sigma_j(t)$ oriented 
from $\Sigma_j^+(t)$ toward $\Sigma_j^-(t)$.
We note that $p_\ell(t)$ has multivalued on $\Omega(t)$ 
but single valued in $\dot{\Omega}(t)$
and  $\nabla p_\ell \in C^\infty(\overline{\Omega(t)})$, 
for all $ \ell = 1,\dots,L$. Then 
$\nabla p_1(t), \dots, \nabla p_L(t)$ is the basis of 
$X_{\mathrm{har}}\bigl(\Omega(t)\bigr)$. 
See also Temam \cite[Appendix I]{Temam}.

\begin{lemma}\label{lem;Xhar perp}
    Let $w(t) \in Z^2_{\sigma}\bigl(\Omega(t)\bigr)\cap H^2\bigl(\Omega(t)\bigr)$ for $t \in \R$. 
    Then, $\tilde{w}(t) \in Z^2_{\sigma}\bigl(\Omega(t_0)\bigr)\cap H^2\bigl(\Omega(t_0)\bigr)$.
\end{lemma}
\begin{proof}
Since $Z^2_{\sigma}\bigl(\Omega(t)\bigr)
=X_{\mathrm{har}}\bigl(\Omega(t)\bigr)^\perp \cap X_{\sigma}^2\bigl(\Omega(t)\bigr)$,
we note that $w(t) \perp X_{\mathrm{har}}\bigl(\Omega(t)\bigr)$.
So, 
since $\Div\, w(t)=0$ a.e. in $\Omega(t)$ and $w(t)\cdot\nu(t)=0$ on $\pt \Omega(t)$, by a direct calculation, we see that
\begin{equation*}
    \begin{split}
        0&=\bigl(w(t),\nabla_x p_\ell(t)\bigr)
        =\int_{\dot{\Omega}(t)} w(t)\cdot \nabla_x p_\ell(t)\,dx
        \\
        &=
        \int_{\pt \dot{\Omega}(t)} w(t)\cdot \nu(t) \, p_\ell(t)\,dS
        -
        \int_{\dot{\Omega}(t)} \Div\, w(t) \,p_\ell(t)\,dx
        \\
        &=
        \sum_{j=1}^L \int_{\Sigma_j(t)} w(t)\cdot \nu_j(t)\,
        \bigl[p_\ell(t)\bigr]_{\Sigma_j(t)}\,dS
        \\
        &=\int_{\Sigma_\ell(t)} w(t)\cdot \nu_\ell(t)\,dS,
    \end{split}
\end{equation*}
for all $\ell=1,\dots,L$. On the other hand,
we put $\tilde{p}_\ell(\tilde{x},t):= p_\ell \bigl(\varphi^{-1}(\tilde{x},t),t\bigr)$
for $\tilde{x}\in \Omega(t_0)$, for $\ell=1,\dots,L$.
Then we easily see that 
\begin{equation*}
    \bigl[\tilde{p}_\ell(t)\bigr]_{\Sigma_j(t_0)}=\delta_{\ell,j}
    \quad \text{for } \ell,j =1,\dots,L,
\end{equation*}
since the diffeomorphism $\varphi(\cdot,t)$ maps $\Sigma_j(t)$ to $\Sigma_j(t_0)$.
Moreover, we note that 
$\nabla_{\tilde{x}}\tilde{p_\ell}(t) \in C^{\infty}\bigl(\overline{\Omega(t_0)}\bigr)$.
Then, 
since $\tilde{w}(t) \in L^2_{\sigma}\bigl(\Omega(t_0)\bigr) 
\cap H^2\bigl(\Omega(t_0)\bigr)$ with $\Div\,\tilde{w}(t)=0$
a.e. in $\Omega(t_0)$ and $\tilde{w}(t)\cdot \nu(t_0)=0$ on $\pt \Omega(t_0)$,
we have by the similar manner as the above, for $\ell=1,\dots,L$,
\begin{equation*}
    \begin{split}
0&=\int_{\Omega(t)} w(t)\cdot \nabla_x p_\ell (t)\,dx
=\int_{\Omega(t_0)} \tilde{w}(t) \cdot \nabla_{\tilde{x}}\tilde{p}_\ell(t)J(t)\,d\tilde{x}
= J(t) \int_{\dot{\Omega}(t_0)} \tilde{w}(t) \cdot \nabla_{\tilde{x}}\tilde{p}_\ell(t)\,d\tilde{x}
\\
&=J(t) \int_{\Sigma_\ell(t_0)} \tilde{w}(t)\cdot \nu_\ell(t_0)\,dS.
    \end{split}
\end{equation*}
Hence, we obtain that
\begin{equation*}
    \int_{\Sigma_\ell(t_0)} \tilde{w}(t)\cdot \nu_\ell(t_0)\,dS=0
    \quad\text{for } \ell=1,\dots,L.
\end{equation*}
This yields that $\tilde{w}(t) \perp X_{\mathrm{har}}\bigl(\Omega(t_0)\bigr)$.
Indeed, 
let $\nabla_{\tilde{x}} p_1(t_0),\dots, \nabla_{\tilde{x}}p_L(t_0)$ be
the basis of $X_{\mathrm{har}}\bigl(\Omega(t_0)\bigr)$ defined from \eqref{eq;basis of Xhar}.
Then we immediately obtain that
\begin{equation*}
    \bigl(\tilde{w}(t),\nabla_{\tilde{x}}p_\ell(t_0)\bigr)
    =\int_{\Sigma_\ell(t_0)} \tilde{w}(t)\cdot \nu_\ell(t_0)\,dS=0
    \quad \text{for } \ell=1,\dots,L.
\end{equation*}
This completes the proof.
\end{proof}
%
%

Next, we discuss the differentiability of $\tilde{w}(t)$ at $t_0 \in [0,T]$.
The following lemma gives the existence of the derivative.
\begin{lemma}\label{lem;diff local w}
    Let $T>0$.
    Let $b(t) \in H^1\bigl(\Omega(t)\bigr)$ with $\Div\,b(t)=0$ in $\Omega(t)$ for $t\in \R$
    and let $w(t) \in Z^2_{\sigma}\bigl(\Omega(t)\bigr)\cap H^2\bigl(\Omega(t)\bigr)$ 
    and $h(t) \in V_{\mathrm{har}}\bigl(\Omega(t)\bigr)$ satisfy
    $h(t) + \rot\, w(t) =b(t)$ in $\Omega(t)$.
    If $\widetilde{b} \in C^1\bigl([0,T];H^1(\widetilde{\Omega})\bigr)$,
    then, for every $t_0 \in [0,T]$, there exists $\dot{w}(t_0) \in H^2\bigl(\Omega(t_0)\bigr)$ such that
    \begin{equation*}
        \lim_{\ep \to 0} 
        \frac{\tilde{w}(t_0+\ep) - w(t_0)}{\ep}=\dot{w}(t_0) 
        \quad \text{in } H^2\bigl(\Omega(t_0)\bigr).
    \end{equation*}
\end{lemma}
\begin{proof}
To begin with, as a preparation, we note that for $t_0 \in [0,T]$
\begin{equation*}
    \dot{b}(t_0):= \lim_{\ep \to 0} 
    \frac{\tilde{b}(t_0+\ep)-b(t_0)}{\ep}
    \quad\text{and}\quad
    \dot{h}(t_0):= \lim_{\ep \to 0} 
    \frac{\tilde{h}(t_0+\ep)-h(t_0)}{\ep}
\end{equation*}
exist in $H^1\bigl(\Omega(t_0)\bigr)$, since $b,h\in C^1\bigl([0,T];H^1(\widetilde{\Omega})\bigr)$
from the assumption and Theorem \ref{thm;contideffi har}, respectively.
Actually, $\dot{b}^i(t_0)=\displaystyle\sum_{k} \frac{\pt \tilde{x}^i}{\pt y^k} \pt_s \widetilde{b}^k(t_0)$
and
$\dot{h}^i(t_0)=\displaystyle\sum_{k} \frac{\pt \tilde{x}^i}{\pt y^k} \pt_s \widetilde{h}^k(t_0)$
for $i=1,2,3$.
Indeed, from the transformation: $\tilde{b}(t)\text{ on }\Omega(t_0) 
\stackrel{\phi(\cdot,t_0)}{\longmapsto}\widetilde{\tilde{b}}(t) \text{ on }\widetilde{\Omega}$
and the relation $\widetilde{\tilde{b}}(t)=\widetilde{b}(t)$ 
in $\widetilde{\Omega}$ as in \eqref{eq;tildewidetilde},
we have that for $i=1,2,3$,
\begin{equation*}
    \tilde{b}^i(t) =\sum_{k}\frac{\pt \tilde{x}^i}{\pt y^k} \widetilde{b}^k(t),
    \quad\text{i.e.,}\quad
    \tilde{b}^i(\tilde{x},t)= \sum_{k} \frac{\pt \phi^{-1}_{i}}{\pt y^k} (y,t_0)\widetilde{b}^k(y,t)
    \quad\text{with } y=\phi(\tilde{x},t_0).
\end{equation*}
Hence, by a direct calculation, we see that 
\begin{equation*}
    \begin{split}
    \left\| \frac{\tilde{b}^i(t_0+\ep) -b^i(t_0)}{\ep} 
    -\sum_{k}\frac{\pt \tilde{x}^i}{\pt y^k}\pt_s\widetilde{b}^k(t_0)\right\|_{H^1(\Omega(t_0))} 
    &\leq 
    C
    \left\| \frac{\widetilde{b}(t_0+\ep)-\widetilde{b}(t_0)}{\ep}-\pt_s\widetilde{b}(t_0)\right\|_{H^1(\widetilde{\Omega})}
    \\
    &\to 0 \quad \text{as } \ep \to 0,
    \end{split}
\end{equation*}
for $i=1,2,3$, where we can take the constant $C>0$ is independent of $\ep$ and $t_0$.
So does for $h(t)$.

Noting that 
$ \tilde{h}(t_0+\ep) + \Rot (t_0+\ep)\,\tilde{w}(t_0+\ep)=\tilde{b}(t_0+\ep)$ 
in $\Omega(t_0)$
for $\ep\in \R\setminus \{0\}$, we see that
\begin{equation*}
    \begin{split}
    \tilde{b}(t_0+\ep) -b(t_0) - &\tilde{h}(t_0+\ep) + h(t_0)
    \\ 
    &= 
    \Rot (t_0+\ep)\, \tilde{w}(t_0+\ep) - \rot\, w(t_0)
    \\
    &= \bigl(\Rot(t_0+\ep) - \rot_{\tilde{x}} \bigr) \tilde{w}(t_0+\ep) + \rot_{\tilde{x}}\bigl(\tilde{w}(t_0+\ep)-w(t_0)\bigr).
    \end{split}
\end{equation*}
More precisely, we have
\begin{multline}\label{eq;rot ww}
    \rot_{\tilde{x}}\,
        \frac{\tilde{w}(t_0+\ep) -w(t_0)}{\ep}
    \\=
    \frac{\tilde{b}(t_0+\ep) -b(t_0)}{\ep} 
    - \frac{\tilde{h}(t_0+\ep) - h(t_0)}{\ep}
    -\frac{\Rot(t_0+\ep) - \rot_{\tilde{x}}}{\ep} \tilde{w}(t_0+\ep). 
\end{multline}
Here, taking the third term of \eqref{eq;rot ww} into account, we put for $\tilde{v} \in H^1\bigl(\Omega(t_0)\bigr)$
\begin{equation*}
    \begin{split}
    \dot{\Rot}(t_0)\, \tilde{v} 
    &:= \lim_{\ep\to 0} \frac{\Rot(t_0 +\ep ) -\rot_{\tilde{x}}}{\ep} \tilde{v}
    \quad \text{in } L^2\bigl(\Omega(t_0)\bigr)
    \\
    &= \left(
        \sum_{j,k,\ell} \frac{\pt R^{i,1}_{j,k,\ell}}{\pt t}(t_0)\tilde{v}^\ell
        +
        \sum_{j,k,\ell} \frac{\pt R^{i,2}_{j,k,\ell}}{\pt t}(t_0) \frac{\pt \tilde{v}^\ell}{\pt \tilde{x}^k}
    \right)_{i=1,2,3}.
    \end{split}
\end{equation*}
Moreover, we see that $\dot{\Rot}(t_0)\,\tilde{v}$ is also well defined in $H^1\bigl(\Omega(t_0)\bigr)$
for $\tilde{v} \in H^2\bigl(\Omega(t_0)\bigr)$.
Since $\tilde{w}(t_0+\ep) \to w(t_0)$ in $H^2\bigl(\Omega(t_0)\bigr)$ 
as $\ep \to 0$
by Lemma \ref{lem;conti w local}, 
so $\| \tilde{w}(t_0+\ep)\|_{H^2(\Omega(t_0))}$
is bounded for $\ep$ sufficiently close to $0$.
Therefore, we can see that  each term in the right hand side of \eqref{eq;rot ww}
is differentiable at $t_0$.

On the other hand, 
since $\bigl(\tilde{w}(t_0+\ep)-w(t_0)\bigr)/\ep  \in 
Z_\sigma^2\bigl(\Omega(t_0)\bigr) \cap H^2\bigl(\Omega(t_0)\bigr)$ by Lemma \ref{lem;Xhar perp},
we see that \eqref{eq;rot ww} gives the Helmholtz-Weyl decomposition on $\Omega(t_0)$ 
of the right hand side of \eqref{eq;rot ww},
since the R.H.S. of \eqref{eq;rot ww} is divergence free.
Therefore, by Proposition \ref{prop;KY} it holds that
\begin{equation*}
    \left\| \frac{\tilde{w}(t_0+\ep)-\tilde{w}(t_0)}{\ep}\right\|_{H^2(\Omega(t_0))}
    \leq C\| \text{R.H.S. of \eqref{eq;rot ww}}\|_{H^1(\Omega(t_0))},
\end{equation*}
for $\ep\neq 0$, where we can take the constant $C>0$ is independent of $\ep$ and $t_0$ by
Theorem \ref{thm;const KY dec}. The linearity of the above estimate yields the existence of
$\displaystyle \dot{w}(t_0):= \lim\limits_{\ep\to0}\frac{\tilde{w}(t_0+\ep)-w(t_0)}{\ep}$ 
in $H^2\bigl(\Omega(t_0)\bigr)$.
Indeed, $\displaystyle \left\{\frac{\tilde{w}(t_0+\ep)-w(t_0)}{\ep}\right\}$ forms a Cauchy sequence as $\ep\to0$.
This completes the proof.
\end{proof}

It should be noted that, from the proof of Lemma \ref{lem;diff local w}, we also see that 
\begin{equation*}
    \dot{w}(t_0) \perp X_{\mathrm{har}}\bigl(\Omega(t_0)\bigr),
\end{equation*}
and 
\begin{equation*}
    \lim_{\ep\to 0} \left\|\rot_{\tilde{x}}\,\frac{\tilde{w}(t_0+\ep)-\tilde{w}(t_0)}{\ep} -\rot_{\tilde{x}}\,\dot{w}(t_0)
    \right\|_{H^1(\Omega(t_0))}=0.
\end{equation*}
Hence, we obtain
\begin{equation}\label{eq;rot dotw}
    \rot_{\tilde{x}}\, \dot{w}(t_0) =  \dot{b}(t_0) - \dot{h}(t_0) -\dot{\Rot}(t_0)\, w(t_0)
    \quad \text{in } \Omega(t_0),
\end{equation}
and, noting that $\Div_{\tilde{x}}\, \dot{\Rot}(t_0)\,w(t_0)=0$
in $\Omega(t_0)$, its estimate
\begin{equation}\label{eq;bounded dotw}
    \begin{split}
    \| \dot{w}(t_0) \|_{H^2(\Omega(t_0))} 
    &\leq 
    C \bigl( \| \dot{b}(t_0) \|_{H^1(\Omega(t_0))}
    + \| \dot{h}(t_0) \|_{H^1(\Omega(t_0))} + \| \dot{\Rot}(t_0) \,w(t_0) \|_{H^1(\Omega(t_0))}
    \bigr)
    \\
    &\leq 
    C \left(
        \sup_{0\leq t \leq T} \| \pt_s \widetilde{b}(t) \|_{H^1(\widetilde{\Omega})}
        +
        \sup_{0\leq t \leq T} \| \pt_s \widetilde{h}(t) \|_{H^1(\widetilde{\Omega})}
        +
        \sup_{0\leq t \leq T} \| \widetilde{w} (t) \|_{H^2(\widetilde{\Omega})}
    \right),
    \end{split}
\end{equation}
where we can choose the constant $C>0$ is independent of $t_0\in [0,T]$ by Theorem \ref{thm;const KY dec},
Proposition \ref{prop;funcsp}, and by the estimate of kernel functions in $\dot{\Rot}(t_0)$.
\medskip

Next, we discuss time continuity of the derivative  $\dot{w}(t)$ at $t_0$.
Hence, similar to $\dot{\mathcal{L}}(t_0)$ in Section \ref{subsec;diff qk},
we have to rewrite $\dot{\Rot}(t_0)$ replacing $t_0$ by $t$.
since $R^{i,1}_{j,k,\ell}, R^{i,2}_{j,k,\ell}$ consist of the  diffeomorphism $\vphi(\cdot,t)$
which depends on  $t_0$, we rewrite the kernel functions in $\dot{\Rot}(t)$ 
by using $\phi(\cdot,t)$ without $\varphi(\cdot,t)$.
Hence, we define $\mathcal{R}^{i,1}_{\ell}$ and $\mathcal{R}^{i,2}_{k,\ell}$ on $\overline{Q}_\infty$ by
\begin{equation*}
    \begin{split}
        \mathcal{R}^{i,1}_\ell(x,t):=\,&
        \sum_{j,k} \biggl(
            \frac{\pt^3 \phi^{-1}_{\sigma(i+2)}}{\pt y^j \pt y^k \pt s}\bigl(\phi(x,t),t\bigr)
            \frac{\pt \phi^j}{\pt x^{\sigma(i+1)}}(x,t)
           \\ &\quad -
            \frac{\pt^3 \phi^{-1}_{\sigma(i+1)}}{\pt y^j \pt y^k \pt s}\bigl(\phi(x,t),t\bigr)
            \frac{\pt \phi^j}{\pt x^{\sigma(i+2)}}(x,t)
        \biggr)\frac{\pt \phi^k}{\pt x^\ell}(x,t)
        \\
        &+\sum_{j} \biggl(
            \frac{\pt^2 \phi^{-1}_{\sigma(i+2)}}{\pt y^j \pt s}
            \bigl(\phi(x,t),t\bigr) \frac{\pt^2 \phi^j}{\pt x^{\sigma(i+1)}\pt x^\ell}(x,t)
            \\ &\quad -
            \frac{\pt^2 \phi^{-1}_{\sigma(i+1)}}{\pt y^j \pt s}
            \bigl(\phi(x,t),t\bigr) \frac{\pt^2 \phi^j}{\pt x^{\sigma(i+2)}\pt x^\ell}
        \biggr) \quad \text{for } (x,t) \in \overline{Q}_\infty
    \end{split}
\end{equation*}
and noting that for $j,k=1,2,3$, $\sigma(j+1)=k$ is equivalent to $j=\sigma(k+2)$ and $ \sigma(j+2)=\sigma(k+1)$,
and that $\sigma(j+2)=k$ is equivalent to $j=\sigma(k+1)$ and $\sigma(j+1)=\sigma(k+2)$ from \eqref{eq;defi sigma}, 
\begin{equation*}
    \begin{split}
        \mathcal{R}^{i,2}_{k,\ell}(x,t):=\,&
        \sum_j \biggl(
            -\frac{\pt^2 \phi^{-1}_i}{\pt y^j \pt s}\bigl(\phi(x,t),t\bigr)
            \frac{\pt \phi^{j}}{\pt x^{\sigma(k+2)}}(x,t) \delta_{\sigma(k+1),\ell}
            \\
            &\quad +
            \frac{\pt^2 \phi^{-1}_i}{\pt y^j \pt s}\bigl(\phi(x,t),t\bigr)
            \frac{\pt \phi^{j}}{\pt x^{\sigma(k+1)}}(x,t) \delta_{\sigma(k+2),\ell}
            \biggr)\\
            &+\sum_j \biggl(
                -\frac{\pt^2 \phi^{-1}_k}{\pt y^j \pt s}\bigl(\phi(x,t),t\bigr)
                \frac{\pt \phi^{j}}{\pt x^{\sigma(i+1)}}(x,t) \delta_{\sigma(i+2),\ell}
                \\
                &\quad +
                \frac{\pt^2 \phi^{-1}_k}{\pt y^j \pt s}\bigl(\phi(x,t),t\bigr)
                \frac{\pt \phi^{j}}{\pt x^{\sigma(i+2)}}(x,t) \delta_{\sigma(i+1),\ell} 
        \biggr)
        \\&+\sum_j \biggl(
            \delta_{k,\sigma(i+1)}
            \frac{\pt^2 \phi^{-1}_{\sigma(i+2)}}{\pt y^j \pt s}\bigl(\phi(x,t),t\bigr)
            \frac{\pt \phi^{j}}{\pt x^{\ell}}(x,t)
            \\
            &\quad -\delta_{k,\sigma(i+2)}
            \frac{\pt^2 \phi^{-1}_{\sigma(i+1)}}{\pt y^j \pt s}\bigl(\phi(x,t),t\bigr)
            \frac{\pt \phi^{j}}{\pt x^{\ell}}(x,t) 
            \biggr)\quad \text{for } (x,t) \in \overline{Q}_\infty.
    \end{split}
\end{equation*}
Then, we easily see that $\mathcal{R}^{i,1}_\ell, \mathcal{R}^{i,2}_{k,\ell}\in C^\infty(\overline{Q}_\infty)$
and that
\begin{equation*}
    \bigl[\dot{\Rot}(t_0)\, \tilde{v}\bigr]^i = \sum_{\ell} \mathcal{R}^{i,1}_\ell(\tilde{x},t_0) \tilde{v}^\ell+ 
    \sum_{k,\ell} \mathcal{R}^{i,2}_{k,\ell}(\tilde{x},t_0) \frac{\pt \tilde{v}^\ell}{\pt \tilde{x}^\ell},
    \quad i=1,2,3,\; \text{for } \tilde{v} \in H^1\bigl(\Omega(t_0)\bigr).
\end{equation*}
So, we can generalize $\dot{\Rot}(t)$ on $H^1\bigl(\Omega(t)\bigr)$ for all $t\in [0,T]$ with $\mathcal{R}^{i,1}_{\ell}$,
$\mathcal{R}^{i,2}_{k,\ell}$.
Hence, \eqref{eq;rot dotw} holds for each $\Omega(t)$ such as
\begin{equation}\label{eq;rot dotw t}
    \rot\, \dot{w}(t) = \dot{b}(t) - \dot{h}(t) - \dot{\Rot}(t)\,w(t)\quad\text{in }\Omega(t).
\end{equation}

Let us transform the relation \eqref{eq;rot dotw t} into that in $\Omega(t_0)$ using the diffeomorphism $\varphi(\cdot,t)$.
So, putting $\tilde{\mathcal{R}}^{i,1}_{\ell}(\tilde{x},t)=\mathcal{R}^{i,1}_{\ell} \bigl(\vphi^{-1}(\tilde{x},t),t\bigr)$
and $\tilde{\mathcal{R}}^{i,2}_{k,\ell}(\tilde{x},t)=\mathcal{R}^{i,2}_{k,\ell} \bigl(\vphi^{-1}(\tilde{x},t),t\bigr)$
for $\tilde{x} \in \Omega(t_0)$, $t\in [0,T]$ and for $i,k,\ell=1,2,3$, we put and see that
for $\tilde{v} \in H^2\bigl(\Omega(t_0)\bigr)$
\begin{equation*}
    \begin{split}
    \bigl[\tilde{\dot{\Rot}}(t)\, \tilde{v}\bigr]^i(\tilde{x})
    &:=\sum_j \frac{\pt \vphi^i}{\pt x^j}\bigl(\vphi^{-1}(\tilde{x},t),t\bigr)
    \bigl[\dot{\Rot}(t)\,v\bigr]^j\bigl(\vphi(\tilde{x},t)\bigr)
    \\
    &=\sum_{j} \frac{\pt \vphi^i}{\pt x^j}\bigl(\vphi^{-1}(\tilde{x},t),t\bigr)
    \biggl\{
        \sum_{\ell,m}  \tilde{\mathcal{R}}^{j,1}_{\ell}(\tilde{x},t)
        \frac{\pt \vphi^{-1}_\ell}{\pt \tilde{x}^m}(\tilde{x},t)
        \tilde{v}^m(\tilde{x})
        \\
        &\quad+
        \sum_{k,\ell} \tilde{\mathcal{R}}^{j,2}_{k,\ell}(\tilde{x},t)
        \biggl(
            \sum_{m,n}\frac{\pt^2 \vphi^{-1}_\ell}{\pt \tilde{x}^m\tilde{x}^n}(\tilde{x},t) 
            \frac{\pt \vphi_m}{\pt x^k}\bigl(\vphi^{-1}(\tilde{x},t),t\bigr)\tilde{v}^m(\tilde{x})
            \\
            &\quad +
            \sum_{m,n} \frac{\pt \vphi^{-1}_\ell}{\pt \tilde{x}^{m}}(\tilde{x},t)
            \frac{\pt \vphi^n}{\pt x^k}\bigl(\vphi^{-1}(\tilde{x},t),t\bigr)
            \frac{\pt \tilde{v}^m}{\pt \tilde{x}^n}(\tilde{x})
        \biggr)
    \biggr\}.
    \end{split}
\end{equation*}
Hence, from \eqref{eq;eq vphidelta} to \eqref{eq;vphitoId} 
and by Lemma \ref{lem;conti w local}, we remark that for $i=1,2,3$,
\begin{equation*}
    \begin{split}
    \bigl[\tilde{\dot{\Rot}}(t)\,\tilde{v}\bigr]^i
    &\to
    \sum_{j} \delta_{i,j}\biggl\{
        \sum_{m,n} \mathcal{R}^{j,1}_{\ell}(\tilde{x},t_0) \delta_{\ell,m}\tilde{v}^m
        +\sum_{k,\ell} \mathcal{R}^{j,2}_{k,\ell}(\tilde{x},t_0) 
        \biggl(
            \sum_{m,n}\delta_{\ell,m}\delta_{n,k}\frac{\pt\tilde{v}^m}{\pt \tilde{x}^n}
        \biggr)
    \biggr\}
    \\
    &=\bigl[\dot{\Rot}(t_0)\,\tilde{v}\bigr]^i \quad \text{in } H^1\bigl(\Omega(t_0)\bigr)\,\text{as }t\to t_0.
    \end{split}
\end{equation*}
Then, by the transformation for \eqref{eq;rot dotw t}, we obtain that
\begin{equation}\label{eq;transformed rot dotw}
    \Rot(t)\,\tilde{\dot{w}}(t) = \tilde{ \dot{b}}(t) - \tilde{\dot{h}}(t) - \tilde{\dot{\Rot}}(t)\, \tilde{w}(t) 
    \quad\text{in } \Omega(t_0).
\end{equation}
Then, we discuss the time continuity of the derivative $\dot{w}(t)$ at $t_0$ in the following sense.
\begin{lemma}\label{lem;conti local dotw}
    Under the assumption of Lemma \ref{lem;diff local w}, 
    for $t_0 \in [0,T]$ and for $\ep>0$ there exists $\delta>0$ such that
    for every $t\in [0,T]$ if $|t-t_0|<\delta$ then 
    \begin{equation*}
        \| \tilde{\dot{w}}(t) - \dot{w}(t_0)\|_{H^2(\Omega(t_0))}
        < \ep.
    \end{equation*}
\end{lemma}
\begin{proof}
    Firstly, from \eqref{eq;rot dotw} and \eqref{eq;transformed rot dotw} we note that
\begin{equation}\label{eq;conti tildotw}
    \begin{split}
    \rot_{\tilde{x}} \bigl( \tilde{\dot{w}}(t) - \dot{w}(t_0) \bigr) =\,& 
    \tilde{\dot{b}}(t) -  {\dot{b}}(t_0) 
    -  \tilde{\dot{h}}(t) + {\dot{h}}(t_0) -\tilde{\dot{\Rot}}(t)\,\tilde{w}(t)+\dot{\Rot}(t_0)\,w(t_0)
    \\
    &-\bigl(\Rot(t)-\rot_{\tilde{x}}\bigr) \tilde{\dot{w}}(t)
    \qquad \text{in } \Omega(t_0).
    \end{split}
\end{equation}

On the other hand, since $\dot{w}(t) \perp X_{\mathrm{har}}\bigl(\Omega(t)\bigr)$, 
it also holds $\tilde{\dot{w}}(t) \perp X_{\mathrm{har}}\bigl(\Omega(t_0)\bigr)$
from Lemma \ref{lem;Xhar perp}. 
Moreover, $\tilde{\dot{w}}(t)\cdot \nu(t_0)=0$ on $\pt \Omega(t_0)$ and $\Div\, \tilde{\dot{w}}(t)=0$ in $\Omega(t_0)$.
Hence, we see that \eqref{eq;conti tildotw} is the Helmholtz-Weyl decomposition on $\Omega(t_0)$
in the sense of Proposition \ref{prop;KY}.
Therefore, we have the estimate
\begin{equation*}
    \begin{split}
    \| \tilde{\dot{w}}(t) - \dot{w}(t_0) \|_{H^2(\Omega(t_0))} 
    &\leq 
    C\| \text{R.H.S. of \eqref{eq;conti tildotw}} \|_{H^1(\Omega(t_0))}
    \\
    &\leq C\Bigl(\bigl\| \tilde{\dot{b}}(t) -  {\dot{b}}(t_0) \bigr\|_{H^1(\Omega(t_0))}
    +
    \bigl\| \tilde{\dot{h}}(t) - {\dot{h}}(t_0)\bigr\|_{H^1(\Omega(t_0))}
    \\
    & \quad +
        \| \tilde{\dot{\Rot}}(t)\tilde{w}(t) - \dot{\Rot}(t_0)w(t_0)\|_{H^1(\Omega(t_0))}
        \\
    & \quad 
    +  \bigl\| \bigl(\Rot(t) -\rot\bigr) \tilde{\dot{w}}(t) \bigr\|_{H^1(\Omega(t_0))}
        \Bigr)
        \\
        &=: C \bigl( K_1 + K_2 + K_3 + K_4\bigr),
    \end{split}
\end{equation*}
where we can take the constant $C>0$ is independent of $t\in [0,T]$ by Theorem \ref{thm;const KY dec}.

We shall estimate $K_1$.
Since $\displaystyle  \dot{b}^i(x,t)=\sum_{k} \frac{\pt \phi_i^{-1}}{\pt y^k}(\phi(x,t),t)\pt_s
\widetilde{b}^k(\phi(x,t),t)$, by a direct calculation we see that for $\tilde{x}= \vphi(x,t)$
\begin{equation*}
    \tilde{ \dot{b}}^i(\tilde{x},t)=
    \sum_\ell \frac{\pt \vphi^i}{\pt x^\ell} (x,t) \dot{b}^\ell(x,t)= 
    \sum_k 
    \frac{\pt \phi_i^{-1}}{\pt y^k}\bigl(\phi(\tilde{x},t_0),t_0\bigr)
    \pt_s \widetilde{b}^k\bigl(\phi(\tilde{x},t_0),t\bigr).
\end{equation*} 
Hence, we have by Proposition \ref{prop;funcsp}
\begin{equation}\label{eq;K1}
    K_1 \leq C \|\pt_s \widetilde{b}(t) - \pt_s \widetilde{b}(t_0) \|_{H^{1}(\widetilde{\Omega})} \to 0
    \quad \text{as } t\to t_0,
\end{equation}
where the constant $C>$ is independent of $t\in [0,T]$.
By the same way, we have $K_2\to 0$ as $t\to t_0$.

Since $\tilde{\dot{\Rot}}(t) \to \dot{\Rot}(t_0)$ strongly as $t \to t_0$,
and since $\tilde{w}(t) \to w(t_0)$ in $H^2\bigl(\Omega(t_0)\bigr)$ as $t\to t_0$ 
by Lemma \ref{lem;conti w local}, 
\begin{equation*}
    \begin{split}
        K_3 
        &\leq 
        \bigl\| \bigl(\tilde{\dot{\Rot}}(t) -\dot{\Rot}(t_0)\bigr)\tilde{w}(t)
        \bigr\|_{H^1(\Omega(t_0))}
        + \bigl\| \dot{\Rot}(t_0) \bigl( \tilde{w}(t)-w(t_0)\bigr) \bigr\|_{H^1(\Omega(t_0))}
        \\
        &\to 0 \quad \text{as } t\to t_0.
    \end{split}
\end{equation*}

Finally, since $\|\tilde{\dot{w}}(t)\|_{H^2(\Omega(t_0))}$ is bounded for $t \in [0,T]$
from \eqref{eq;bounded dotw}, we obtain $K_4 \to 0$ as $t\to t_0$ by direct calculation.
This completes the proof.
\end{proof}

As a summary, we conclude the following.
\begin{theorem}
    Let $T>0$. Let $b(t) \in H^1\bigl(\Omega(t)\bigr)$ with $\Div\,b(t)=0$ in $\Omega(t)$ for $t \in \R$.
    Let $w(t) \in Z_{\sigma}^2\bigl(\Omega(t)\bigr) \cap H^2\bigl(\Omega(t)\bigr)$
    and let $h(t) \in V_{\mathrm{har}}\bigl(\Omega(t)\bigr)$ satisfy the Helmholtz-Weyl decomposition
    $h(t) + \rot\, w(t) =b(t)$ in $\Omega(t)$.
    If $\widetilde{b} \in C^1\bigl([0,T];H^1(\widetilde{\Omega})\bigr)$, then it holds that
    \begin{equation*}
        \widetilde{w} \in C^1\bigl([0,T];H^2(\widetilde{\Omega})\bigr).
    \end{equation*}
\end{theorem}
\begin{proof} Let $t_0 \in [0,T]$ be fixed.
    To begin with, we recall the relation $\widetilde{\tilde{w}}(t) = \widetilde{w}(t)$ in $\widetilde{\Omega}$
    as in \eqref{eq;tildewidetilde}. 

    We shall show $\pt_s \widetilde{w}(t_0) = \widetilde{\dot{w}}(t_0)$ in $\widetilde{\Omega}$ for $t_0\in [0,T]$.
    Indeed, by Proposition \ref{prop;funcsp} and Lemma \ref{lem;diff local w}, we have that
    \begin{equation*}
        \begin{split}
            \left\| \frac{\widetilde{w}(t_0+\ep)-\widetilde{w}(t_0)}{\ep}
             - \widetilde{\dot{w}}(t_0)\right\|_{H^2(\widetilde{\Omega})}
             &=
             \left\| \frac{\widetilde{\tilde{w}}(t_0+\ep)-\widetilde{w}(t_0)}{\ep}
             - \widetilde{\dot{w}}(t_0)\right\|_{H^2(\widetilde{\Omega})}
             \\
             &\leq 
             C\left\| \frac{\tilde{w}(t_0+\ep)-{w}(t_0)}{\ep}
             - {\dot{w}}(t_0)\right\|_{H^2(\Omega(t_0))} 
             \\
             &\to 0 \quad\text{as } t\to t_0,
        \end{split}
    \end{equation*}
    where we can take the constant $C>0$ independent of sufficiently small $|\ep|$ and $t_0 \in [0,T]$. 

    Moreover, by Proposition \ref{prop;funcsp} and Lemma \ref{lem;conti local dotw}, we have
    \begin{equation*}
        \begin{split}
            \| \pt_s \widetilde{w}(t) - \pt_s \widetilde{w}(t_0) \|_{H^2(\widetilde{\Omega})}
            &=
            \bigl\| \widetilde{\tilde{\dot{w}}}(t)-\widetilde{\dot{w}}(t_0)
            \bigr\|_{H^2(\widetilde{\Omega})}
            \\
            &\leq C \| \tilde{\dot{w}}(t) -\dot{w}(t_0) \|_{H^2(\Omega(t_0))}
            \\
            &\to 0 \quad \text{as } t \to t_0,
        \end{split}
    \end{equation*}
    where the constant $C>0$ is independent of $t \in [0,T]$. This completes the proof.
\end{proof}
\begin{remark}
    The strategy to derive time differentiability needs a qualitative observation for 
    $X_{\mathrm{har}}\bigl(\Omega(t)\bigr)$ which comes from the topology of the domain
    as in Lemma \ref{lem;Xhar perp}.
    Here, we emphasize that this strategy can be also applicable to 
    the proof of the time continuity of $\widetilde{w}(t)$, provided Theorem \ref{thm;const KY dec} is established.
    However, to derive time continuity we adopt a more general approach with 
    the a priori estimate, since the quantitative analysis still valid, 
    regardless of the topology of the domain.
\end{remark}

\subsection{Proof of Corollary \ref{cor;time depend HW dec gen}}
\label{subsec;cor}
%
The scalar potential $p \in H^2\bigl(\Omega(t)\bigr)\cap H^1_0\bigl(\Omega(t)\bigr)$  
in the Helmholtz-Weyl decomposition for $f(t) \in H^1\bigl(\Omega(t)\bigr)$ is uniquely determined
by the following Poisson equation with the homogeneous Dirichlet boundary condition:
\begin{equation}\label{eq;Lap p}
\begin{cases}
    \Delta p(t) = \Div f(t) &\text{in }\Omega(t),\\
    p(t)=0 & \text{on }\pt \Omega(t).
\end{cases}
\end{equation}

So, we shall show the following theorem.
\begin{theorem}\label{thm;timedepend p}
    Let $T>0$. Let $f(t) \in H^1\bigl(\Omega(t)\bigr)$ for $t\in \R$
    and let $p(t)\in H^2\bigl(\Omega(t)\bigr)\cap H^1_0\bigl(\Omega(t)\bigr)$
    be the solution of \eqref{eq;Lap p}.
    If $\widetilde{f} \in C\bigl([0,T];H^1(\widetilde{\Omega})\bigr)$
    then $\widetilde{p} \in C\bigl([0,T];H^2(\widetilde{\Omega})
    \cap H^1_0(\widetilde{\Omega})\bigr)$.
   Furthermore, if $\widetilde{f} \in C^1\bigl([0,T];H^1(\widetilde{\Omega})\bigr)$
    then $\widetilde{p} \in C^1\bigl([0,T];H^2(\widetilde{\Omega})
    \cap H^1_0(\widetilde{\Omega})\bigr)$.
\end{theorem}

In order to prove Corollary \ref{cor;time depend HW dec gen}, we put $b(t):=f(t)-\nabla p(t)$.
Then, we easily see by Theorem \ref{thm;timedepend p} that $\Div\, b(t)=0$ in $\Omega(t)$ and 
$\widetilde{b}\in C^m\bigl([0,T];H^1(\widetilde{\Omega})\bigr)$ 
with some $m\in\{0,1\}$, according to $\widetilde{f}\in C^m\bigl([0,T];H^1(\widetilde{\Omega})\bigr)$.
Therefore, applying Theorem \ref{thm;time depend HW dec div} to $b(t)$,
we immediately obtain Corollary \ref{cor;time depend HW dec gen}.

For the proof, similar to the argument as in Section \ref{sec;Vharmonic}, we introduce 
the transformed function $\tilde{p}(t)$ on $\Omega(t_0)$ defined by
\begin{equation*}
    \tilde{p}(\tilde{x},t)= p\bigl(\varphi^{-1}(\tilde{x},t),t\bigr)
    \quad \text{for }\tilde{x} \in \Omega(t_0),
\end{equation*}
where the diffeomorphism $\varphi(\cdot,t):\Omega(t)\to\Omega(t_0)$ is as in \eqref{eq;diffeo t0} for a fixed $t_0\in[0,T]$.
Then, by the virtue of Proposition \ref{prop;IW div}, \eqref{eq;Lap p} is equivalent to 
\begin{equation}\label{eq;Lt p}
    \begin{cases}
        \mathcal{L}(t)\,\tilde{p}(t) = \Div_{\tilde{x}} \tilde{f}(t) &\text{in } \Omega(t_0),\\
        \tilde{p}(t)=0 & \text{on } \pt \Omega(t_0).
    \end{cases}
\end{equation}
\subsubsection{Time continuity of the solution to \eqref{eq;Lt p} for $t$}
\begin{lemma}\label{lem;conti tildep}
    Under the assumption of Theorem \ref{thm;timedepend p},
    for $t_0\in[0,T]$ and for $\ep>0$ then there exists $\delta>0$ such that
    for every $t\in\R$ if $|t-t_0|<\delta$ then 
    $\| \tilde{p}(t) -p(t_0)\|_{H^2(\Omega(t_0))} < \ep$.
\end{lemma}
\begin{proof}
    The argument is quite an analogy of proof of Lemma \ref{lem;contiq_k}.
    By the a priori estimate \eqref{eq;aprioriq} and by \eqref{eq;Lap p} at $t_0$ and \eqref{eq;Lt p},
    \begin{equation*}
        \begin{split}
        \| \tilde{p}(t) - p(t_0)\|_{H^2(\Omega(t_0))}
        &\leq C
        \bigl\| \mathcal{L}(t)\bigl( \tilde{p}(t)-p(t_0)\bigr)
        \bigr\|_{H^2(\Omega(t_0))}
        \\
        &\leq
        C\Bigl(
            \| \Div_{\tilde{x}} \tilde{f}(t) - \Div_{\tilde{x}} f(t_0)\|_{L^2(\Omega(t_0))}
            +
            \bigl\|
                \bigl(\mathcal{L}(t) -\Delta_{\tilde{x}}\bigr)p(t_0)
            \bigr\|_{L^2(\Omega(t_0))}
        \Bigr)
        \\
        & \to 0 \quad \text{as } t\to t_0,
        \end{split}
    \end{equation*}
    where the constant $C>0$ is independent of $t$. 
    This completes the proof. 
\end{proof}
From Lemma \ref{lem;conti tildep}, we immediately obtain 
$\widetilde{p} \in C\bigl([0,T];H^2(\widetilde{\Omega})\bigr)$.
%
\subsubsection{Time differentiability of the solution to \eqref{eq;Lt p} for $t$}
Firstly, we put $\displaystyle \dot{f}(t_0)=\lim\limits_{\ep\to 0}\frac{\tilde{f}(t_0+\ep)-f(t_0)}{\ep}$
in $H^2\bigl(\Omega(t_0)\bigr)$ for $t_0 \in [0,T]$.
Then we introduce the solution $\dot{p}(t_0)$ of the following Poisson equation 
with the 
homogeneous Dirichlet boundary condition:
\begin{equation}\label{eq;Lap dotp}
    \begin{cases}
        \Delta_{\tilde{x}} \dot{p}(t_0) = -\dot{\mathcal{L}}(t_0)p(t_0) + \Div_{\tilde{x}} \dot{f}(t_0)
        & \text{in } \Omega(t_0),\\
        \dot{p}(t_0)= 0 & \text{on } \pt\Omega(t_0).
    \end{cases}
\end{equation}
\begin{lemma}
    Under the assumption of Theorem \ref{thm;timedepend p}, for $t_0\in[0,T]$, it holds that
    \begin{equation} \label{eq;diff lim p}
        \lim_{\ep \to 0}\frac{\tilde{p}(t_0+\ep)-p(t_0)}{\ep} = \dot{p}(t_0) \quad \text{in } H^2\bigl(\Omega(t_0)\bigr). 
    \end{equation}
\end{lemma}
\begin{proof}
    Similar to the proof of Lemma \ref{lem;diff q_k}, in oder to apply the a priori estimate
    \eqref{eq;aprioriq}, from \eqref{eq;Lap p}, \eqref{eq;Lt p} and \eqref{eq;Lap dotp} we confirm that
    \begin{equation*}
        \begin{split}
            \mathcal{L}(t_0+\ep)
            \left(
                \frac{\tilde{p}(t_0+\ep)-p(t_0)}{\ep} -\dot{p}(t_0)
            \right)
            &=
            -
            \frac{\mathcal{L}(t_0+\ep)-\Delta_{\tilde{x}}}{\ep} p(t_0)
            + \dot{\mathcal{L}}(t_0) p(t_0)
            \\
            & \quad -
            \bigl(\mathcal{L}(t_0+\ep)-\Delta_{\tilde{x}}\bigr)\dot{p}(t_0)
            \\
            &\quad + \Div_{\tilde{x}} \frac{\tilde{f}(t_0 +\ep)-f(t_0)}{\ep}   
            -\Div_{\tilde{x}}\dot{f}(t_0).
        \end{split}
    \end{equation*}
    Hence, by \eqref{eq;aprioriq} we immediately obtain \eqref{eq;diff lim p}.
\end{proof}
From \eqref{eq;Lap dotp}, for $t\in[0,T]$, the transformed function $\tilde{\dot{p}}(t)$ satisfies the following problem:
\begin{equation}\label{eq;Lt dotp}
    \begin{cases}
        \mathcal{L}(t)\, \tilde{\dot{p}}(t) = -\tilde{\dot{\mathcal{L}}}(t) \tilde{p}(t) 
        + \Div_{\tilde{x}} \tilde{\dot{f}}(t).
        &\text{in } \Omega(t_0),
        \\
        \tilde{p}(t) = 0 & \text{on }\pt\Omega(t_0).
    \end{cases}
\end{equation}
Therefore, we have the following lemma.
\begin{lemma}\label{lem;conti local dotp}
    Under the assumption of Theorem \ref{thm;timedepend p}, for $t_0$ and $\ep>0$
    then there exists $\delta>0$ such that for every $t\in[0,T]$ if $|t-t_0|<\delta$
    then $\| \tilde{\dot{p}}(t) -\dot{p}(t_0) \|_{H^2(\Omega(t_0))}<\ep$. 
\end{lemma}
\begin{proof}
    The proof is an analogy of that of Lemma \ref{lem;diff conti q_k}.
    To apply \eqref{eq;aprioriq}, from \eqref{eq;Lap dotp} and \eqref{eq;Lt dotp}, we see that
    \begin{equation*}
        \begin{split}
            \mathcal{L}(t)\bigl(\tilde{\dot{p}}(t) - \dot{p}(t_0)\bigr)
            &=
            -\bigl(\tilde{\dot{\mathcal{L}}}(t) -\dot{\mathcal{L}}(t_0)\bigr) \tilde{p}(t)
            -\dot{\mathcal{L}}(t_0) \bigl(\tilde{p}(t) - p(t_0)\bigr)
            \\
            &\quad -
            \bigl(\mathcal{L}(t)-\Delta\bigr)\dot{p}(t_0)
            +\Div_{\tilde{x}}\bigl(\tilde{\dot{f}}(t)-\dot{f}(t_0)\bigr).
        \end{split}
    \end{equation*}
Hence, noting that $\sup\limits_{0\leq t\leq T}\|\tilde{p}(t)\|_{H^2(\Omega(t_0))}<\infty$ 
from Lemma \ref{lem;conti tildep}
and $\lim\limits_{t\to t_0}\| \tilde{\dot{f}}(t) - \dot{f}(t_0)\|_{H^1(\Omega(t_0))}=0$ 
by the argument as in \eqref{eq;K1}, the a priori estimate \eqref{eq;aprioriq}
yields Lemma \ref{lem;conti local dotp}. 
\end{proof}
By the analogy of the proof of Theorem \ref{thm;contidiff q}, 
we obtain the proof of Theorem \ref{thm;timedepend p} from these lemmata.
%
\section{Construction of time periodic solutions}
%
In this section, we shall construct a weak solution by the Galerkin method, 
along to Miyakawa Teramoto \cite{Miyakawa Teramoto}.
Firstly, we prepare the suitable $L^2$-complete orthonormal basis and its properties.
Then, we investigate the existence of weak solutions of (N-S$^\prime$)
for every $a\in L^2_{\sigma}\bigl(\Omega(0)\bigr)$, $f \in L^2(Q_T)$ and 
$\widetilde{b}\in C^1\bigl([0,T];H^1(\widetilde{\Omega})\bigr)$. 
Finally, we construct a time periodic solution of (N-S$^\prime$).

We note that we identify a vector filed $u$ on $\Omega(t)$ 
with $\widetilde{u}$ on $\widetilde{\Omega}$ under the relations \eqref{eq;identify} or
\eqref{eq;transform u}, throughout the section.

\subsection{Complete orthonormal basis in $L^2(\widetilde{\Omega})$
with respect to $\langle \cdot,\cdot\rangle_t$}
Let $\{ \widetilde{\Upsilon}_k \}_{k=1}^\infty \subset C_{0,\sigma}^\infty(\widetilde{\Omega})$
be linearly independent and total in $H^1_{0,\sigma}(\widetilde{\Omega})$, 
namely, every function in $H^1_{0,\sigma}(\widetilde{\Omega})$
can be approximated by some finite linear combination of $\widetilde{\Upsilon}_k$.
Moreover, we may assume $\{\widetilde{\Upsilon}_k\}_{k\in\N}$ 
is a complete orthonormal basis on $L_\sigma^2(\widetilde{\Omega})$. 

By the Schmidt orthonormalization of $\{\widetilde{\Upsilon}_k\}$ with respect to the inner product
$\langle \cdot,\cdot\rangle_t$ in $L^2(\widetilde{\Omega})$ for each $t\in [0,T]$, 
we introduce $\{ \widetilde{\psi}_k(\cdot,t)\}_{k=1}^\infty$ which is
linearly independent, total in $H^1_{0,\sigma}(\widetilde{\Omega})$ 
and $\bigl\langle \widetilde{\psi}_k(t), \widetilde{\psi}_\ell(t)\bigr\rangle_t=\delta_{k,\ell}$.
We note that $\widetilde{\psi}_k(y,t)$, $k\in \N$, can be realized as a linear combination
$\sum\limits_{\ell=1}^k \mu_{k,\ell}(t)\widetilde{\Upsilon}_\ell(y)$ with smooth coefficients
$\mu_{k,\ell} \in C^\infty\bigl([0,T]\bigr)$, $\ell=1,\dots,k$.
Moreover, we introduce a norm induced from the inner product $\langle\cdot,\cdot\rangle_t$
on $L^2(\widetilde{\Omega})$,
\begin{equation*}
    \| \widetilde{u} \|_{L^2(\widetilde{\Omega};t)}
    := 
    \bigl\langle \widetilde{u},\widetilde{u}\bigr\rangle_t^{\frac{1}{2}}
    \quad
    \text{for } \widetilde{u} \in L^2(\widetilde{\Omega}).
\end{equation*}
Then,  since $\|\widetilde{u}\|_{L^2(\widetilde{\Omega};t)}
=\| u \|_{L^2(\Omega(t))}$ from \eqref{eq;equivalence innerproduct L2}
by Proposition \ref{prop;funcsp}
we note that
$\|\cdot\|_{L^2(\widetilde{\Omega};t)}$ is equivalent to 
the usual norm $\|\cdot\|_{L^2(\widetilde{\Omega})}$ for all $t\in [0,T]$.

Furthermore, we recall the inner product 
\begin{equation}
    \bigl\langle \nabla_g \widetilde{u}, \nabla_g \widetilde{v} \bigr\rangle_t 
    =
    \int_{\widetilde{\Omega}} \sum_{i,j,k,\ell}
    g_{ij}(y,t) g^{k\ell}(y,t) \nabla_k \widetilde{u}^i \nabla_\ell \widetilde{v}^{j}J(t)\,dy 
    \quad \text{for } \widetilde{u}, \widetilde{v} \in H^1_{0}(\widetilde{\Omega}),
\end{equation}
for each $t \in [0,T]$. Then we put
\begin{equation*}
    \| \widetilde{u} \|_{H^1(\widetilde{\Omega};t)} :=  \bigl\langle 
        \nabla_g \widetilde{u},\nabla_g \widetilde{u}\bigr\rangle_t^{\frac{1}{2}}
        \quad\text{for } \widetilde{u}\in H^1_{0}(\widetilde{\Omega}).
\end{equation*}
Similarly as above, by Proposition \ref{prop;funcsp} and by the Poincar\'{e} inequality $\| u \|_{L^2(\Omega(t))} \leq C_p \|\nabla u \|_{L^2(\Omega(t))}$
 we note that
\begin{equation*}
    \begin{split}
    C\| \widetilde{u} \|_{H^1(\widetilde{\Omega})} \geq 
    \| \widetilde{u} \|_{H^1(\widetilde{\Omega};t)} &=\| \nabla u \|_{L^2(\Omega(t))}
    \geq \frac{1}{{C_p}+1} \| u \|_{H^1(\Omega(t))}
    \geq C^\prime \| \widetilde{u} \|_{H^1(\widetilde{\Omega})}, 
    \end{split}
\end{equation*}
where  we can take the constants $C,C^\prime>0$ are independent of $t \in [0,T]$, 
since $C_p \lesssim \mathrm{diam\,}\Omega(t)$ and since $\mathrm{diam\,}\Omega(t)$ is continuous on $[0,T]$ under our assumption on $\Omega(t)$.

Finally, we note that 
every test function $\Psi \in C_0^1\bigl( [0,T); H_{0,\sigma}^1(\widetilde{\Omega})\bigr)$
can be approximated in the sense of the norm $\sup\limits_{0\leq t \leq T}\|\cdot\|_{H^1(\widetilde{\Omega};t)}$ by the function of the following form:
\begin{equation}\label{eq;weak test function app}
    \sum_{\textrm{finite}} \lambda_k(t) \widetilde{\psi}_k(y,t).
\end{equation}
See, Miyakawa and Teramoto \cite{Miyakawa Teramoto}, see also, Masuda \cite[Lemma 2.2]{Masuda}.

\subsection{Construction of weak solutions by the Galerkin method}
%
Using $\{\widetilde{\psi}_k\}$ we construct approximate solution for $m=1,2,\dots$,
\begin{equation}\label{eq;approx sol}
    \widetilde{u}_m (y,t) = \sum_{k=1}^m h_{m,k}(t)\widetilde{\psi}_k(y,t),
\end{equation}
where, the coefficients $h_{m,k}$ are determined by the following system:
\begin{equation}\label{eq;Galerkin}
    \left\{\begin{split}
    \bigl\langle \pt_s\widetilde{u}_m(t), \widetilde{\psi}_k(t) \bigr\rangle_t 
    =\,&
    \bigl\langle L \widetilde{u}_m(t), \widetilde{\psi}_k(t) \bigr\rangle_t
    -
    \bigl\langle M \widetilde{u}_m(t), \widetilde{\psi}_k (t) \bigr\rangle_t
    \\
    &+
    \bigl\langle N[ \widetilde{b}(t), \widetilde{u}_m(t)] + N[\widetilde{u}_m(t),\widetilde{b}(t)]
    +
    N[\widetilde{u}_m(t), \widetilde{u}_m(t)],\widetilde{\psi}_k(t) \bigr\rangle_t
    \\
    &+
    \bigl\langle \widetilde{F}(t), \widetilde{\psi}_k(t) \bigr\rangle_t,
    \\
    \bigl\langle \widetilde{u}_m(0), \widetilde{\psi}_k(0) \bigr\rangle_0
    =\,& 
    \bigl\langle \widetilde{a}, \widetilde{\psi}_k(0) \bigr\rangle_0,
    \end{split}\right.
\end{equation} 
for $k=1,\dots,m$ and $t\geq 0$. 

Multiplying \eqref{eq;Galerkin} by $h_{m,k}(t)$, taking sum in $k$, 
by the Proposition \ref{prop;duds}, for $u_m(t)$ which is the transformation of 
$\widetilde{u}_m(t)$ into $\Omega(t)$, we have
\begin{equation}\label{eq;EDI}
    \frac{1}{2}\frac{d}{dt} \| {u}_m(t) \|^2_{L^2(\Omega(t))} +
    \| \nabla u_m(t) \|^2_{L^2(\Omega(t))} + 
    \bigl(u_m(t)\cdot \nabla b(t), u_m(t)\bigr)
    = \bigl\langle \widetilde{F}(t),\widetilde{u}_m(t)\bigr\rangle_t,
\end{equation}
for $m =1,2,\dots$. 
Furthermore, by the H\"{o}lder (interpolation) inequality, the Sobolev inequality 
the Young inequality and Proposition \ref{prop;funcsp}, we see that
\begin{equation*}
    \begin{split}
        \bigl|\bigl(
            u_m(t) \cdot \nabla b(t),u_m(t)
        \bigr)\bigr|
        &=
        \bigl|\bigl(
            u_m(t)\cdot \nabla u_m(t), b(t)
        \bigr)\bigr|
        \\
        &\leq
        \| \nabla u_m(t) \|_{L^2(\Omega(t))} \|u_m(t)\|_{L^3(\Omega(t))} \| b(t) \|_{L^6(\Omega(t))}
        \\
        &\leq
        \| \nabla u_m(t) \|_{L^2(\Omega(t))} \| u_m(t)\|_{L^2(\Omega(t))}^{\frac{1}{2}}
        \|u_m(t)\|_{L^6(\Omega(t))}^{\frac{1}{2}} 
        C\| \widetilde{b}(t) \|_{L^6(\widetilde{\Omega})}
        \\
        &\leq 
        \ep\| \nabla u_m(t) \|_{L^2(\Omega(t))}^2
        +
        \frac{C^2C_s}{4\ep} \|u_m(t)\|_{L^2(\Omega(t))} \| \nabla u_m(t) \|_{L^2(\Omega(t))}
        \|\widetilde{b}(t) \|_{H^1(\widetilde{\Omega})}^2
        \\
        &\leq 
        2\ep \| \nabla u_m(t) \|_{L^2(\Omega(t))}^2 
        + \frac{C^4C_s^2}{16\ep^2} \|u_m(t)\|_{L^2(\Omega(t))}^2 
        \|\widetilde{b}(t)\|_{H^1(\widetilde{\Omega})}^4,
    \end{split}
\end{equation*}
for any $\ep>0$ and for all $t\in [0,T]$, 
where $C$ is independent of $t$ and we recall the constant $C_s=3^{-1/2}2^{2/3}\pi^{-2/3}$ 
is the best constant of the Sobolev embedding $H^1_0(\Omega) \hookrightarrow L^6(\Omega)$.
Nothing that $\| M b(t) \|_{L^2(\widetilde{\Omega})} \leq C \|b(t) \|_{H^1(\widetilde{\Omega})}$, 
we see that
\begin{equation}\label{eq;Fum}
    \begin{split}
        \bigl|\bigl\langle \widetilde{F}(t), \widetilde{u}_m(t) \bigr\rangle_t\bigr|
        \leq\,&
        \bigl|\bigl( {f}(t), {u}_m(t) \bigr)\bigr|
        +
        \bigl|\bigl\langle \pt_s \widetilde{b}(t), \widetilde{u}_m (t) \bigr\rangle_t \bigr|
        +
        \bigl|\bigl\langle M\widetilde{b}(t), \widetilde{u}_m(t) \bigr\rangle_t \bigr|
        \\ &+
        \bigl|\bigl( \nabla {b}(t), \nabla {u}_m(t) \bigr)\bigr|
        +
        \bigl|\bigl( b(t)\cdot \nabla u_m(t), b(t)\bigr)\bigr| 
        \\
        \leq \, & 
        \| f(t) \|_{L^2(\Omega(t))} \| u_m(t) \|_{L^2(\Omega(t))}
        \\
        &
        + C\|\pt_s \widetilde{b}(t) \|_{L^2(\widetilde{\Omega})} \|u_m(t) \|_{L^2(\Omega(t))}
        + \| M\widetilde{b}(t) \|_{L^2(\widetilde{\Omega})} \|u_m(t)\|_{L^2(\Omega(t))}
        \\
        &+\|\nabla b(t) \|_{L^2(\Omega(t))}\|\nabla u_m(t) \|_{L^2(\Omega(t))}
        + \| b(t) \|^2_{L^4(\Omega(t))} \| \nabla u_m(t) \|_{L^2(\Omega(t))}
        \\
        \leq \,&
        \| u_m(t) \|_{L^2(\Omega(t))}^2 + 2\ep \| \nabla u_m(t) \|^2_{L^2(\Omega)}
        \\&+ 
        C_\ep\Bigl(\|f(t)\|_{L^2(\Omega(t))}^2 + \| \pt_s \widetilde{b}(t) \|^2_{L^2(\widetilde{\Omega})}
        + \|\widetilde{b}(t) \|^2_{H^1(\widetilde{\Omega})}
        +\| \widetilde{b}(t) \|_{H^1(\widetilde{\Omega})}^4
        \Bigr),
    \end{split}
\end{equation}
for any $\ep >0$ and for all $t\in [0,T]$, where the constant $C_\ep>0$ is independent of $t$.

Combining the above estimates with sufficiently small $\ep>0$, 
integrating \eqref{eq;EDI} in $t$ we obtain that
\begin{multline}\label{eq;EI}
    \| u_m(t) \|_{L^2(\Omega(t))}^2 + 2(1-4\ep) \int_0^t \| \nabla u_m(\tau) \|^2_{L^2(\Omega(\tau))}\,d\tau
    \\\leq \| a \|_{L^2(\Omega(0))}^2 
    +C\int_0^t \| u_m(\tau)\|^2_{L^2(\Omega(\tau))} 
    \bigl(\| \widetilde{b}(\tau) \|_{H^1(\Omega(\tau))}^4+1\bigr) \,d\tau
    +\int_0^t K(\tau)\,d\tau
\end{multline}
where
\begin{equation}\label{eq;KFF}
    K(t):= C\Bigl(\|f(t)\|_{L^2(\Omega(t))}^2 + \| \pt_s \widetilde{b}(t) \|^2_{L^2(\widetilde{\Omega})}
    + \|\widetilde{b}(t) \|^2_{H^1(\widetilde{\Omega})}
         + 
        \| \widetilde{b}(t) \|_{H^1(\Omega(t))}^4
        \Bigr).
\end{equation}
Hence, by the Gronwall inequality, the coefficients $h_{m,k}(t)$ exist on the whole interval $[0,T]$.
Moreover, we see that $\{\widetilde{u}_m(t)\}_{m\in\N} $ is a bounded sequence in $L^\infty\bigl(0,T;L_\sigma^2(\widetilde{\Omega})\bigr)$
and also in $L^2\bigl(0,T; H^1_{0,\sigma}(\widetilde{\Omega})\bigr)$.

Hereafter, we may assume, for some $A>0$,
\begin{equation}\label{eq;AEI}
    \sup_{0\leq t \leq T}\| \widetilde{u}_m(t)\|^2_{L^2(\widetilde{\Omega};t)} + 
    \int_0^T \| \widetilde{u}_m(\tau) \|_{H^1(\widetilde{\Omega};\tau)}\,d\tau
    \leq A \quad \text{for } m\in\N.
\end{equation}

\begin{lemma}\label{lem;equiconti}
    Let $\ell \in \N$ be fixed. The family  $\{h_{m,\ell}\}_{m=1}^\infty$ forms a uniformly bounded 
    and equicontinuous family of continuous functions on $[0,T]$. 
\end{lemma}
\begin{proof}
    We note that $h_{m,\ell}(t)= \bigl\langle \widetilde{u}_m(t), \widetilde{\psi}_\ell(t)\bigr\rangle_t
    =\bigl(u_m(t),\psi_\ell(t)\bigr)$. 
    The uniform boundedness follows from \eqref{eq;AEI}. 
    Then, see that
    \begin{equation*}
        \begin{split}
            |h_{m,\ell}(t) - h_{m,\ell}(s) |
            = \,&
            \left|\int_s^t \frac{d}{d\tau} \bigl\langle \widetilde{u}_m(\tau),\widetilde{\psi}_\ell(\tau)
            \bigr\rangle_{\tau}\,d\tau \right|
            \\
            \leq\,&
            \int_s^t \Bigl(
            \bigl|\bigl\langle \nabla_g \widetilde{u}_m(\tau), \nabla_g \widetilde{\psi}_\ell(\tau)
            \bigl\rangle_{\tau}\bigr|
            +
            \bigl|\bigl\langle N[\widetilde{b}(\tau),\widetilde{u}_m(\tau)], \widetilde{\psi}_\ell(\tau)
            \bigr\rangle_{\tau}\bigr|
            \\&+
            \bigl|
                \bigl\langle
                N[\widetilde{u}_m(\tau),\widetilde{b}(\tau)], \widetilde{\psi}_\ell(\tau)
                \bigr\rangle_{\tau}
            \bigr|
            + \bigl|
                \bigl\langle
                 N[\widetilde{u}_m(\tau), \widetilde{u}_m(\tau)], \widetilde{\psi}_\ell(\tau)
                \bigr\rangle_{\tau}
            \bigr|
            \\&+
            \bigl|\bigl\langle \widetilde{F}(\tau), \widetilde{\psi}_\ell(\tau)\bigr\rangle_t\bigr|
            +\bigl|\bigl\langle \widetilde{u}_m(\tau), \pt_s \widetilde{\psi}_\ell(\tau) +M\widetilde{\psi}_\ell(\tau)
            \bigr\rangle_{\tau}\bigr|
            \Bigr)\,d\tau
            \\
            \leq \,&
            \int_s^t \Bigl(
                \bigl|\bigl( \nabla u_m(\tau), \nabla \psi_\ell(\tau)\bigr)\bigr|
                + 
                \bigl|
                    \bigl( b(\tau)\cdot \nabla u_m(\tau),\psi_\ell(\tau)\bigr)
                \bigr|
            \\
            &+ \bigl|\bigl(u_m(\tau)\cdot\nabla b(\tau),\psi_\ell(\tau)\bigr)\bigr|
            + \bigl|\bigl( u_m(\tau)\cdot \nabla u_m(\tau), \psi_\ell(\tau) \bigr)\bigr|
            \\
            &+
            \bigl|\bigl\langle \widetilde{F}(\tau), \widetilde{\psi}_\ell(\tau)\bigr\rangle_t\bigr|
            +\bigl|\bigl\langle \widetilde{u}_m(\tau), \pt_s \widetilde{\psi}_\ell(\tau) +M\widetilde{\psi}_\ell(\tau)
            \bigr\rangle_{\tau}\bigr|
            \Bigr)\,d\tau
            \\=: \,&
            \mathcal{I}_1+ \mathcal{I}_2+\mathcal{I}_3 + \mathcal{I}_4 + \mathcal{I}_5 + \mathcal{I}_6.
        \end{split}
    \end{equation*}
    So, we shall estimate $\mathcal{I}_1,\dots,\mathcal{I}_6$.

    It holds that by Proposition \ref{prop;funcsp}, \eqref{eq;AEI} and the Schwarz inequality,
    \begin{equation*}
        \begin{split}
        |\mathcal{I}_1| 
        &\leq \int_s^t \| \nabla u_m(\tau) \|_{L^2(\Omega(\tau))} \|\nabla \psi(\tau) \|_{L^2(\Omega(\tau))}\,d\tau
        \\
        &\leq
        C\sup_{0\leq \tau \leq T} \|\widetilde{\psi}_\ell(\tau)\|_{H^1(\widetilde{\Omega})} \int_s^t \| \nabla u_m(\tau)\|_{L^2(\Omega(\tau))}\,d\tau
       \\&\leq 
        C\sup_{0\leq \tau \leq T} \|\widetilde{\psi}_\ell(\tau)\|_{H^1(\widetilde{\Omega})} 
        A^{\frac{1}{2}} |t-s|^{\frac{1}{2}},
        \end{split}
    \end{equation*}
    where the constant $C>0$ is independent of $m\in \N$ and of $t,s,\tau \in [0,T]$.
    
    Next, we estimate $\mathcal{I}_2$ and $\mathcal{I}_3$. 
    Especially, by Proposition \ref{prop;funcsp} we see that
    \begin{equation*}
        \begin{split}
            |\mathcal{I}_2|,|\mathcal{I}_3|
            &\leq 
            \int_s^t \| u_m(\tau) \|_{L^6(\Omega(\tau))} \|\nabla \psi_\ell(\tau) \|_{L^2(\Omega(\tau))}
            \| b(\tau) \|_{L^3(\Omega(\tau))} \,d\tau
            \\
            &\leq 
            \int_s^t C_s \| \nabla u_m(\tau) \|_{L^2(\Omega(\tau))} 
            \|\nabla \psi_\ell(\tau) \|_{L^2(\Omega(\tau))}
            \| b(\tau) \|_{L^3(\Omega(\tau))} \,d\tau
            \\
            &\leq C \sup_{0\leq \tau \leq T} \| \widetilde{\psi}_\ell(\tau) \|_{H^1(\widetilde{\Omega})}
            \sup_{0\leq \tau \leq T} \| b(\tau) \|_{L^3(\Omega(\tau))}
            A^{\frac{1}{2}} |t-s|^{\frac{1}{2}},
        \end{split}
    \end{equation*}
    where $C>0$ is independent of $m\in \N$ and of $t,s,\tau \in [0,T]$.

    So we estimate $I_4$. By the H\"{o}lder (interpolation) inequality and the Sobolev inequality we have
    \begin{equation*}
        \begin{split}
            |\mathcal{I}_4| 
            &\leq 
            \int_s^t \|\nabla u_m(\tau) \|_{L^2(\Omega(\tau))} 
            \| u_m(\tau)\|_{L^3(\Omega(\tau))}
            \| \psi_\ell(\tau) \|_{L^6(\Omega(\tau))}\,d\tau
            \\
            &\leq
            \int_s^t 
            \| \nabla u_m(\tau) \|_{L^2(\Omega(\tau))}
            \| u_m(\tau) \|^{\frac{1}{2}}_{L^2(\Omega(\tau))}
            \| u_m (\tau) \|^{\frac{1}{2}}_{L^6(\Omega(\tau))} 
            C_s\| \nabla \psi_\ell (\tau) \|_{L^2(\Omega(\tau))} \,d\tau
            \\
            &\leq
            C_s^{\frac{3}{2}} 
            \int_s^t \| \nabla u_m(\tau) \|_{L^2(\Omega(\tau))}^{\frac{3}{2}} \| u_m (\tau) \|_{L^2(\Omega(\tau))}^{\frac{1}{2}}
            \| \nabla \psi_\ell (\tau) \|_{L^2(\Omega(\tau))}\,d\tau
            \\
            &\leq
            C \sup_{0\leq \tau \leq T} \| \widetilde{\psi}_\ell(\tau) \|_{H^1(\widetilde{\Omega})}
            A |t-s|^{\frac{1}{4}},
        \end{split}
    \end{equation*}
    where $C>0$ is independent of $m\in \N$ and of $t,s,\tau \in [0,T]$.

    By a similar argument as in \eqref{eq;Fum}, we have
    \begin{equation*}
        \begin{split}
            |\mathcal{I}_5|
            \leq \,&
            C\sup_{0\leq \tau \leq T} \| \widetilde{\psi}_\ell(\tau)\|_{H^1(\widetilde{\Omega})}
            \int_s^t \Bigl(
                \| {f}(\tau) \|_{L^2(\Omega(\tau))}
                + 
                \| \pt_s \widetilde{b}(\tau) \|_{L^2(\widetilde{\Omega})}
                +
                \| M\widetilde{b}(\tau) \|_{L^2(\widetilde{\Omega};\tau)}
                \\
                &
                + \| \widetilde{b}(\tau) \|_{H^1(\widetilde{\Omega})}
                + \| \widetilde{b} (\tau) \|^2_{H^1(\widetilde{\Omega})} 
            \Bigr)\,d\tau
            \\
            &\leq 
            C\sup_{0\leq \tau \leq T} \| \widetilde{\psi}_\ell(\tau)\|_{H^1(\widetilde{\Omega})}
            \biggl(
            \|{f}\|_{L^2(Q_T)}
            |t-s|^{\frac{1}{2}}
            \\
            &+
            \Bigl(
            \sup_{0\leq \tau \leq T} \|\pt_s \widetilde{b}(\tau) \|_{L^2(\widetilde{\Omega})}
            + 
            \sup_{0\leq \tau \leq T} \| \widetilde{b}(\tau)\|_{H^1(\widetilde{\Omega})}
            +
            \sup_{0\leq \tau \leq T} \| \widetilde{b}(\tau)\|_{H^1(\widetilde{\Omega})}^2
            \Bigr)|t-s|
            \biggr),
        \end{split}
    \end{equation*} 
    where $C>0$ is independent of $m\in \N$ and of $t,s,\tau \in [0,T]$.

    Finally, we estimate $\mathcal{I}_6$. 
    We have that
    \begin{equation*}
        \begin{split}
            |\mathcal{I}_6 | 
            &\leq 
            A^{\frac{1}{2}} 
            \sup_{0\leq \tau \leq T} \bigl(\|\pt_s \widetilde{\psi}_\ell(\tau)\|_{L^2(\widetilde{\Omega};\tau)}
            +\| M \widetilde{\psi}_\ell(\tau)\|_{L^2(\widetilde{\Omega};\tau)}\bigr)
            |t-s|
            \\
            &\leq
            CA^{\frac{1}{2}} 
            \sup_{0\leq \tau \leq T} \bigl(\|\pt_s \widetilde{\psi}_\ell(\tau)\|_{L^2(\widetilde{\Omega})}
            +\|\widetilde{\psi}_\ell(\tau)\|_{H^1(\widetilde{\Omega})}\bigr) |t-s|,
        \end{split}
    \end{equation*}
    where $C>0$ is independent of $m\in \N$ and of $t,s,\tau \in [0,T]$.
    
    Combining the above estimates, we conclude that $\{h_{m,\ell}\}_{m=1}^\infty$ is the family of equicontinuous on $[0,T]$.
\end{proof}

Due to Lemma \ref{lem;equiconti}, by the Ascoli-Arzel\'{a} theorem and the usual diagonal argument,
there exists a subsequence $\{ \widetilde{u}_{m_k}\}_{k\in \N} \subset \{\widetilde{u}_m\}$ and
a function $\widetilde{u} \in L^\infty\bigl(0,T;L_\sigma^2(\widetilde{\Omega})\bigr) \cap
L^2\bigl(0,T;H^1_{0,\sigma}(\widetilde{\Omega})\bigr)$ such that
\begin{itemize}
    \item[(i)] $\bigl\langle \widetilde{u}_{m_k}(t)-\widetilde{u}(t), \widetilde{\Psi} \bigr\rangle_t \to 0$ uniformly in $t\in [0,T]$ as $k\to \infty$
    for all $\widetilde{\Psi}\in L_\sigma^2(\widetilde{\Omega})$.  
    \item[(ii)] $\displaystyle \int_0^T \bigl\langle \nabla_g \bigl(\widetilde{u}_{m_k}(\tau) - \widetilde{u}(\tau)\bigr),
    \nabla_g \widetilde{\Psi}(\tau) \bigr\rangle_{\tau}\,d\tau \to 0$ as $k\to \infty$ 
    for all $\widetilde{\Psi}\in L^2\bigl(0,T;H^1_{0,\sigma}(\widetilde{\Omega})\bigr)$.
\end{itemize}

Indeed, we shall prove (i). 
Firstly, we that that by the Ascoli-Arzel\'{a} theorem, we may assume $h_{m_k,\ell} (t)$ converges 
to some continuous function $h_{\ell}(t)$ uniformly in $t\in[0,T]$ as $k\to \infty$ for each $\ell\in \N$.
Moreover, we may assume $ \widetilde{u}_{m_k}(t) \rightharpoonup \widetilde{u}(t)$ in $L^2_{\sigma}(\widetilde{\Omega})$
with respect ot $\langle \cdot,\cdot\rangle_t$
for each $t\in [0,T]$.
Here, we note that $\displaystyle\widetilde{u}(y,t)=\sum_{\ell=1}^\infty h_{\ell}(t)\widetilde{\psi}_\ell(y,t)$ and 
$\sup\limits_{0\leq t\leq T}\|\widetilde{u}(t)\|_{L^2(\widetilde{\Omega};t)}^2 \leq A$.

For $\widetilde{\Psi} \in L_\sigma^2(\widetilde{\Omega})$ and for $\ep>0$, there exists 
$\vartheta_{1},\dots, \vartheta_N \in C^\infty\bigl([0,T]\bigr)$ such that 
\begin{equation*}
\left\| \widetilde{\Psi} - \sum\limits_{n=1}^N \vartheta_n(t) \widetilde{\psi}_n(t)\right\|_{L^2(\widetilde{\Omega};t)} 
\leq C
\left\| \widetilde{\Psi} - \sum\limits_{n=1}^N \vartheta_n(t) \widetilde{\psi}_n(t)\right\|_{L^2(\widetilde{\Omega})}<\ep
\end{equation*}
for all $t\in[0,T]$,
since 
$\displaystyle \widetilde{\Upsilon}_j(y) \equiv \sum\limits_{\ell=1}^j \widetilde{\mu}_{j,\ell}(t) \widetilde{\psi}_\ell(y,t)$ for all $t \in [0,T]$
with smooth coefficients $\widetilde{\mu}_{j,\ell}\in C^\infty\bigl([0,T]\bigr)$
and since
$\{\widetilde{\Upsilon}_j\}_{j\in \N}$ is dense in $L^2_\sigma(\widetilde{\Omega})$.
Hence, we have
\begin{equation*}
    \begin{split}
        \bigl|\bigl\langle \widetilde{u}_{m_k}(t)-\widetilde{u}(t),\widetilde{\Psi}\bigr\rangle_t\bigr|
        \leq\,&
        \left|\left\langle \widetilde{u}_{m_k}(t)-\widetilde{u}(t),\widetilde{\Psi} 
        - \sum_{n=1}^N \vartheta_n(t)\widetilde{\psi}_n(t)\right\rangle_t\right|
        \\&+
        \left|\left\langle \widetilde{u}_{m_k}(t)-\widetilde{u}(t),
         \sum_{n=1}^N \vartheta_n(t)\widetilde{\psi}_n(t)\right\rangle_t\right|
        \\
        \leq\,&
        2A^{\frac{1}{2}} \ep + \max_{n=1,\dots,N}\sup_{0\leq t \leq T}|\theta_n(t)|
        \sum_{n=1}^N \bigl|\bigl\langle \widetilde{u}_{m_k}(t) 
        - \widetilde{u}(t), \widetilde{\psi}_n(t)\bigr\rangle_t\bigr| 
        \\ 
        \leq \, &
        2A^{\frac{1}{2}}\ep + \max_{n=1,\dots,N}\sup_{0\leq t \leq T}|\theta_n(t)|
        \sum_{n=1}^N \bigl|h_{m_k,n}(t)-h_{n}(t)\bigr| 
    \end{split}
\end{equation*}
This proves (i). 
Moreover, taking subsequence if necessary, (ii) also holds.
\begin{lemma}\label{lem;L^2L^2}
    It holds that
    \begin{equation*}
        \widetilde{u}_{m_k} \to \widetilde{u} \quad \text{in } 
        L^2\bigl(0,T; L^2_\sigma(\widetilde{\Omega})\bigr)
        \quad \text{as } k\to \infty.
    \end{equation*}
\end{lemma}
\begin{proof}
     By the (modified) Friedrichs inequality in Miyakawa and Teramoto \cite[(2.17)]{Miyakawa Teramoto},
    for every $\ep>0$ there exists $N_\ep \in \N$ independent of $t\in [0,T]$ such that, 
    for each $t\in[0,T]$ it holds that
    \begin{equation}\label{eq;Friedrichs}
        \| \widetilde{f} \|_{L^2(\widetilde{\Omega};t)}^2 \leq
        \ep \|  \widetilde{f} \|_{H^1(\widetilde{\Omega};t)}^2 
        + \sum_{j=1}^{N_\ep} |\langle\widetilde{f}, \widetilde{\psi}_j(t)\rangle_t |^2
            \quad \text{for } \widetilde{f} \in H^1_{0,\sigma}(\widetilde{\Omega}).
    \end{equation}
    Hence, we see that
    \begin{equation*}
        \begin{split}
            \int_0^T \| \widetilde{u}_{m_k}(\tau) - \widetilde{u}(\tau)\|^2_{L^2(\widetilde{\Omega};\tau)}\,d\tau
            \leq \,&
            \ep \int_0^T \| \widetilde{u}_{m_k}(\tau) - \widetilde{u}(\tau) \|^2_{H^1(\widetilde{\Omega};t)} \,d\tau
            \\ &+
            \sum_{j=1}^{N_\ep}
            \int_0^T\bigl|\bigl\langle \widetilde{u}_{m_k}(\tau) - \widetilde{u}(\tau),
            \widetilde{\psi}_j(\tau)\bigr\rangle_\tau\bigr|^2\,d\tau
            \\
            \leq \,&
            4A\ep + \sum_{j=1}^{N_\ep} \int_0^T |h_{m_k,j}(\tau)-h_{j}(\tau)|^2\,d\tau.
        \end{split}
    \end{equation*}
    Since $\|\cdot\|_{L^2(\widetilde{\Omega})} \leq C\| \cdot \|_{L^2(\widetilde{\Omega};t)} $ with $C>0$ is independent of $t \in [0,T]$ by Proposition \ref{prop;funcsp},
    the proof is completed.
\end{proof}
Next, we consider that $\widetilde{u}(t)$ is a weak solution of (N-S). Take a test function 
$\widetilde{\Psi} \in C_0^1\bigl([0,T);H^1_{0,\sigma}(\widetilde{\Omega})\bigr)$
with the form of \eqref{eq;weak test function app}. By \eqref{eq;Galerkin} we see that
\begin{multline}\label{eq;weakformofapprox}
    \int_0^T -\bigl\langle \widetilde{u}_{m_k}(\tau),  \pt_s\widetilde{\Psi}(\tau)\bigr\rangle_\tau\,d\tau
    +\bigl\langle \nabla_g \widetilde{u}_{m_k}(\tau), \nabla_g \widetilde{\Psi}(\tau)\bigr\rangle_\tau
    +\bigl\langle \widetilde{u}_{m_k}(\tau), M\widetilde{\Psi}(\tau) \bigr\rangle_\tau \\
    +\bigl\langle N[\widetilde{b}(\tau),\widetilde{u}_{m_k}(\tau)],\widetilde{\Psi}(\tau)\bigr\rangle_\tau
    +\bigl\langle N[\widetilde{u}_{m_k}(\tau),\widetilde{b}(\tau)], \widetilde{\Psi} (\tau)\bigr\rangle_\tau
    +\bigl\langle N[\widetilde{u}_{m_k}(\tau),\widetilde{u}_{m_k}(\tau)], \widetilde{\Psi} (\tau)\bigr\rangle_\tau
    \\
    =\bigl\langle \widetilde{u}_{m_k}(0), \widetilde{\Psi}(0)\bigr\rangle_0 
    +\int_0^T \bigl\langle \widetilde{F}(\tau), \widetilde{\Psi}(\tau)\bigr\rangle_\tau\,d\tau
\end{multline}
Furthermore, in order to take a limit as $k\to\infty$, we introduce the following proposition.
\begin{proposition}\label{prop;conv nonlinear}
    It holds that
    \begin{equation*}
        \int_0^T \bigl\langle N[\widetilde{u}_{m_k}(\tau),\widetilde{u}_{m_k}(\tau)],\widetilde{\Psi}(\tau)\bigr\rangle_\tau\,d\tau
        \to
        \int_0^T \bigl\langle N[\widetilde{u}(\tau),\widetilde{u}(\tau)],\widetilde{\Psi}(\tau)\bigr\rangle_\tau\,d\tau,
    \end{equation*}
    as $k\to \infty$ for all $\widetilde{\Psi}\in C^1_0\bigl([0,T); H^1_{0,\sigma}(\widetilde{\Omega})\bigr)$.
\end{proposition}
\begin{proof} By the H\"{o}lder (interpolation) inequality, the Sobolev embedding, we have 
    \begin{equation*}
        \begin{split}
            \int_0^T& \bigl|\bigl\langle N[\widetilde{u}_{m_k}(\tau),\widetilde{u}_{m_k}(\tau)],\widetilde{\Psi}(\tau)\bigr\rangle_\tau
        -
         \bigl\langle N[\widetilde{u}(\tau),\widetilde{u}(\tau)],\widetilde{\Psi}(\tau)\bigr\rangle_\tau\bigr|\,d\tau
        \\
        &\leq
            \int_0^T \Bigl\{ 
                \left|\Bigl(\bigl(u_{m_k}(\tau)-u(\tau)\bigr)\cdot \nabla \Psi(\tau),u_{m_k}(\tau)\Bigr)\right|
                +
                \left|\Bigl(u(\tau)\cdot \nabla \Psi(\tau),\bigl(u_{m_k}(\tau)-u(\tau)\bigr)\Bigr)\right|
            \Bigr\}\,d\tau
            \\
        &\leq
            C_s^{\frac{1}{2}} \sup_{0\leq \tau < T} \| \nabla \Psi(\tau) \|_{L^2(\Omega(\tau))}
            \\
            &\quad \times \Biggl\{
            \int_0^T \| u_{m_k}(\tau) - u(\tau) \|_{L^2(\Omega(\tau))}^{\frac{1}{2}}
            \| \nabla u_{m_k}(\tau) -\nabla u(\tau)\|_{L^2(\Omega(\tau))}^{\frac{1}{2}}  
            \| u_{m_k}(\tau)\|_{L^6(\Omega(\tau))}\,d\tau 
            \\
            &\qquad +
            \int_0^T \| u_{m_k}(\tau) - u(\tau) \|_{L^2(\Omega(\tau))}^{\frac{1}{2}}
            \| \nabla u_{m_k}(\tau) -\nabla u(\tau)\|_{L^2(\Omega(\tau))}^{\frac{1}{2}}  
            \| u(\tau)\|_{L^6(\Omega(\tau))}\,d\tau \Biggr\}
            \\
            &=:
            C_s^{\frac{1}{2}} \sup_{0\leq \tau < T} \| \widetilde{\Psi}(\tau) \|_{L^2(\widetilde{\Omega};t)}
            (\mathcal{J}_1+ \mathcal{J}_2).
        \end{split}
    \end{equation*}
    By the H\"{o}lder inequality and the a priori estimate and Lemma \ref{lem;L^2L^2}, we see that
    \begin{equation*}
        \mathcal{J}_1 \leq \left[ \int_0^T \| \widetilde{u}_{m_k}(\tau) - \widetilde{u}(\tau)\|^2_{L^2(\widetilde{\Omega};t)}\,d\tau\right]^{\frac{1}{4}}
        CA^{\frac{3}{4}} \to 0,
    \end{equation*}
    as $k\to \infty$. Similarly, we obtain that $\mathcal{J}_2 \to 0$ as $k\to \infty$. This completes the proof. 
\end{proof}
Finally, by the virtue of assertion (i), (ii) and Proposition \ref{prop;conv nonlinear}, taking the limit of \eqref{eq;weakformofapprox} as $k\to\infty$, 
we easily see that $\widetilde{u}$ is 
a desired weak solution.
\subsection{Construction of time periodic solutions}
To construct a time periodic solutions, we select a suitable extended function of 
$\widetilde{\beta}\in C^1\bigl(\R; H^{\frac{1}{2}}(\pt \widetilde{\Omega})\bigr)$
with the general flux condition $\int_{\pt \widetilde{\Omega}} \widetilde{\beta}(t)\cdot \nu\, dS=0$
for all $t\in\R$.
By the Bogovski\u{\i} formula, we can take 
$\widetilde{b}\in C^1\bigl(\R;H^1(\widetilde{\Omega})\bigr)$
with $\Div\, \widetilde{b}(t)=0$ in $\widetilde{\Omega}$ for all $t\in \R$.
For more detail, see Borchers and Sohr \cite{Borchers Sohr}.

Put for $i=1,2,3$,
\begin{equation*}
    b^i(x,t) := \sum_{\ell} \frac{\pt \phi^{-1}_i}{\pt y^\ell} \widetilde{b}\bigl(\phi(x,t),t\bigr)
    \quad \text{for } x \in \Omega(t),
\end{equation*}
for $t\in \R$.
Then we observe that $\Div\, b(t)=0$ in $\Omega(t)$ and $b(t)|_{\pt \Omega(t)}=\beta(t)$ for $t \in \R$.
By Proposition \ref{prop;KY} we can decompose $b(t)$ in $\Omega(t)$ such as
\begin{equation}\label{eq;dec KY Omegat}
    h(t) + \rot\, w(t) = b(t) \quad \text{in } \Omega(t).
\end{equation}
Here, we recall $\eta_1(t),\dots,\eta_K(t)$ the orthogonal basis of $V_{\mathrm{har}}\bigl(\Omega(t)\bigr)$,
$\alpha_{jk}(t)$, $j,k=1,\dots,K$ are the coefficients in \eqref{eq;cons Vhar} and $q_1(t),\dots,q_K(t)$
are solutions of \eqref{eq;Lap q_k}.
Then, the harmonic vector field $h(t)$ in \eqref{eq;dec KY Omegat} can be expressed as
\begin{equation*}
    \begin{split}
        h(t)&= \sum_{k=1}^K \bigl(b(t),\eta_k(t)\bigr)\eta_k(t)
        =\sum_{k=1}^K \left(b(t), \sum_{\ell=1}^K \alpha_{k\ell}(t)\nabla q_\ell(t)\right)\eta_k(t)
        \\
        &=\sum_{k,\ell =1}^K \alpha_{k\ell}(t) \left(\int_{\Gamma_\ell (t)} \beta(t)\cdot \nu(t) \,dS\right) \eta_k(t).
    \end{split}
\end{equation*}

Next, since $\widetilde{w}\in C^1\bigl([0,T];H^2(\widetilde{\Omega})\bigr)$ by Theorem \ref{thm;conti w}, 
introducing a cut-off function $\widetilde{\theta}$ 
in  Lemma \ref{lem;cut rot w} for $\ep>0$ which will be determined later,
we put
\begin{equation*}
    \widetilde{h}(t) + \Rot(t)[\widetilde{\theta},\widetilde{w}(t)] =: \widetilde{b}_\ep(t)
    \quad \text{in } \widetilde{\Omega}.
\end{equation*}
Then we see that $\Div\, \widetilde{b}_\ep(t) =0$ in $\widetilde{\Omega}$ 
and $\widetilde{b}_\ep(t)|_{\pt \widetilde{\Omega}}=\widetilde{\beta}(t)$ for $t \in [0,T]$
by taking the divergence of  both sides of $h(t) + \rot\bigl(\theta w(t)\bigr)=b_\ep(t)$ on $\Omega(t)$.
Furthermore, since $\widetilde{h} \in C^1\bigl([0,T];H^1(\widetilde{\Omega})\bigr)$ and
$\widetilde{w} \in C^1\bigl([0,T]; Z_\sigma^2(\widetilde{\Omega})\cap H^2(\widetilde{\Omega})\bigr)$,
we see that $\widetilde{b}_\ep \in C^1\bigl([0,T];H^1(\widetilde{\Omega})\bigr)$. 
Hereafter, we adopt $\widetilde{b}_\ep(t)$  as a extension of $\widetilde{\beta}(t)$.

Let assume 
\begin{equation}\label{eq;smallness h}
    \sup_{0\leq t \leq T}\|h(t) \|_{L^3(\Omega(t))} = 
    \sup_{0\leq t \leq T}\left\| 
        \sum_{k,\ell=1}^K \alpha_{k\ell}(t) \left(\int_{\Gamma_\ell(t)} \beta(t)\cdot \nu(t)\,dS\right)\eta_k(t)
    \right\|_{L^3(\Omega(t))} < \frac{1}{C_s},
\end{equation}
where we recall $C_s=3^{-\frac{1}{2}} 2^{\frac{2}{3}} \pi^{-\frac{2}{3}}$ is the best constant
of the Sobolev embedding $H_0^1(\Omega) \hookrightarrow L^6(\Omega)$. 
\begin{proposition}\label{prop;convection term}
    Under \eqref{eq;smallness h} it holds that
    \begin{equation*}
        \bigl|\bigl\langle N[\widetilde{u},\widetilde{u}],\widetilde{b}_\ep(t)\bigr\rangle_t\bigr|
        \leq \Bigl(C_s \sup_{0\leq t \leq T}\| h(t) \|_{L^2(\Omega(t))} + C\ep \Bigr)\| \nabla u \|^2_{L^2(\Omega(t))}
        \quad \text{for } \widetilde{u} \in H^1_{0,\sigma}(\widetilde{\Omega}),
    \end{equation*}
    where
    the constant $C>0$ is independent of $\widetilde{u}$ and $t\in [0,T]$.
\end{proposition}
\begin{proof}
    By the H\"{o}lder inequality, the Sobolev embedding and Lemma \ref{lem;cut rot w}, we have
    \begin{equation*}
        \begin{split}
            \bigl|\bigl\langle N[\widetilde{u},\widetilde{u}],\widetilde{b}_\ep(t)\bigr\rangle_t\bigr|
            &=
            \bigl| 
                \bigl( u\cdot\nabla u, b_\ep(t)\bigr)
            \bigr|
            \\
            &\leq
            \bigl|\bigl(u\cdot\nabla u,h(t)\bigr)\bigr| 
            + \bigl|\bigl\langle N[\widetilde{u},\widetilde{u}],\Rot(t)[\widetilde{\theta},\widetilde{w}(t)]\bigr\rangle_t\bigr|
            \\
            &\leq C_s \sup_{0\leq t \leq T} \| h(t)\|_{L^3(\Omega(t))} \|\nabla u \|^2_{L^2(\Omega(t))}
            +
            \ep \| \nabla_y \widetilde{u}\|_{L^2(\widetilde{\Omega})}^2
            \\
            &\leq C_s \sup_{0\leq t \leq T} \| h(t)\|_{L^3(\Omega(t))} \|\nabla u \|^2_{L^2(\Omega(t))}
            +
            C\ep \| \nabla u\|_{L^2(\Omega(t))}^2,
        \end{split}
    \end{equation*}
    since $\| \nabla_y \widetilde{u} \|_{L^2(\widetilde{\Omega})} \leq C\| \widetilde{u} \|_{H^1(\widetilde{\Omega};t)}=C\|\nabla u(t)\|_{L^2(\Omega(t))}$.
\end{proof}

Let us construct a time periodic (approximate) solutions. Of course, we assume $\phi(\cdot,t)=\phi(\cdot,t+T)$,
$f(t)=f(t+T)$ and $\widetilde{\beta}(t)=\widetilde{\beta}(t+T)$ for $t \in \R$ with the period $T>0$.
Let $\widetilde{u}_m(y,t)$ be a approximate solution to \eqref{eq;Galerkin} of the form \eqref{eq;approx sol} for 
an arbitrary initial data $\displaystyle u_m(0)=\sum_{k=1}^m a_k \widetilde{\psi}_k(y,0)$.

For each $m\in\N$, by \eqref{eq;EDI} and Proposition \ref{prop;convection term}, we have
\begin{equation*}
    \frac{1}{2}\frac{d}{dt} \| u_m(t) \|_{L^2(\Omega(t))}^2 
    +\bigl(1 -C_s \sup_{0\leq t \leq T}\| h(t) \|_{L^2(\Omega(t))} - C\ep -C^\prime\ep\bigr) \|\nabla u_m(t)\|^2_{L^2(\Omega(t))}
    \leq K(t),
\end{equation*}
where $C^\prime>0$ is independent of $t$ and $u_m$, 
and where $K(t)$ is defined in \eqref{eq;KFF} via a similar calculation of \eqref{eq;Fum}. 
By the Poincar\'{e} inequality $\| u  \|_{L^2(\Omega(t))} \leq C_p \| \nabla u \|_{L^2(\Omega(t))}$
for $u \in H^1_0\bigl(\Omega(t)\bigr)$
with the constant $C_p$ independent of $t\in [0,T]$,
taking sufficiently small $\ep>0$ we can take the constant $\gamma>0$ independent of $t$ and $u_m$ 
so that 
\begin{equation*}
    \frac{d}{dt} \| u_m(t) \|^2_{L^2(\Omega(t))} + \gamma \| u_m(t)\|^2_{L^2(\Omega(t))} \leq 2K(t).
\end{equation*}
Therefore, we obtain that
\begin{equation*}
    \| \widetilde{u}_m(T) \|^2_{L^2(\widetilde{\Omega};T)}=\|\widetilde{u}_m(T) \|_{L^2(\widetilde{\Omega};0)}
    \leq e^{-\gamma T} \| u_m(0) \|^2_{L^2(\widetilde{\Omega};0)} 
    + \int_0^T e^{-\gamma(T-\tau)} 2K(\tau)\,d\tau.
\end{equation*}
Here, we choose $R>0$, so that
\begin{equation*}
    R^2(1-e^{-\gamma T}) \geq \int_0^T e^{-\gamma (T-\tau)}2K(\tau)\,d\tau.
\end{equation*}
Hence, it holds that
\begin{equation*}
    \| \widetilde{u}_m(T) \|_{L^2(\widetilde{\Omega};0)} \leq R, 
    \quad \text{if } \| \widetilde{u}_m(0) \|_{L^2(\widetilde{\Omega};0)}\leq R.
\end{equation*}
On the other hand, the map $\widetilde{u}_m(0) \mapsto \widetilde{u}_m(T)$
is continuous. 
Then the Brouwer fixed point theorem ensures the existence of an approximate solution
$\widetilde{u}_m$ such that $\widetilde{u}_m(0)=\widetilde{u}_m(T)$ and
$\| \widetilde{u}_m(0) \|_{L^2(\widetilde{\Omega};0)}= \| \widetilde{u}_m(T) \|_{L^2(\widetilde{\Omega};0)}\leq R$.
Since $R>0$ is independent of $m\in\N$, so we can see that the sequence $\{\widetilde{u}_m\}_{m\in\N}$ 
is bounded in $L^2\bigl(0,T;L_\sigma^2(\widetilde{\Omega})\bigr) \cap 
L^2\bigl(0,T;H^1_{0,\sigma}(\widetilde{\Omega})\bigr)$.
Taking a subsequence if necessary, there is $\widetilde{a}\in L^2_\sigma(\widetilde{\Omega})$ such that
$\widetilde{u}_m(0) \rightharpoonup \widetilde{a}$ weakly in $L^2_\sigma(\widetilde{\Omega})$ as $m\to \infty$
with respect to the inner product $\langle \cdot,\cdot \rangle_0$.
By the same manner as in the previous subsection, we obtain the weak solution $\widetilde{u}$ with
$\widetilde{u}(0)=\widetilde{u}(T)=\widetilde{a}$. This completes the proof.

\medskip
\textbf{Acknowledgements}
The authors would like to thank Professor Takayuki Kobayashi for his fruiteful comments.
The work of the second author is partially supported by JSPS Grant-in-Aid 
for Scientific Research(C) 22K03385.
The work of the third author is partially supported by JSPS 
Grant-in-Aid for Early-Career Scientists 18K13439, and Grant-in-Aid for Scientific Research(C) 22K03370.

\medskip
\noindent\textbf{Data Availability Statement}
Data sharing is not applicable to this article as no
datasets were generated or analyzed during the current study.

\medskip
\noindent\textbf{Declarations}
\\
\textbf{Conflict of interest} The authors declare that they have no conflict of interest.

\end{document}